\documentclass[10pt]{amsart}

\usepackage[margin=1in,headheight=13.6pt]{geometry}

\usepackage{textcmds} 
\usepackage{amsmath}
\usepackage{amsxtra}
\usepackage{amscd}
\usepackage{amsthm}
\usepackage{amsfonts}
\usepackage{amssymb}
\usepackage{eucal}
\usepackage[all]{xy}
\usepackage{graphicx}
\usepackage{comment}
\usepackage{epsfig}
\usepackage{psfrag}
\usepackage{mathrsfs}
\usepackage{amscd}
\usepackage{rotating}
\usepackage{lscape}
\usepackage{amsbsy}
\usepackage{verbatim}
\usepackage{moreverb}
\usepackage{url}
\usepackage{extarrows}
\usepackage{setspace}

\usepackage{cancel}
\usepackage{bbm}

\usepackage[hidelinks]{hyperref}
\usepackage[new]{old-arrows}

\usepackage{tikz-cd}
\usepackage{tikz}
\usetikzlibrary{matrix,arrows,decorations.pathmorphing}

\DeclareMathOperator*{\colim}{colim}
\DeclareMathOperator*{\Glue}{Glue}

\newtheorem{cor}[subsubsection]{Corollary}
\newtheorem{lem}[subsubsection]{Lemma}
\newtheorem{prop}[subsubsection]{Proposition}

\newtheorem{thm}[subsubsection]{Theorem}

\newtheorem{mainthm}{Theorem}

\newenvironment{mainthmbis}[1]
  {\renewcommand{\themainthm}{\ref{#1}$'$}%
   \addtocounter{mainthm}{-1}%
   \begin{mainthm}}
  {\end{mainthm}}

\newtheorem{defn}[subsubsection]{Definition}

\theoremstyle{remark}
\newtheorem{rem}[subsubsection]{Remark}
\newtheorem{example}[subsubsection]{Example}

\theoremstyle{remark}
\newtheorem{remark}[subsubsection]{Remark}

\numberwithin{equation}{section}

\newcommand{\nc}{\newcommand}
\nc{\renc}{\renewcommand}
\nc{\ssec}{\subsection}
\nc{\sssec}{\subsubsection}
\nc{\on}{\operatorname}

\nc\ol{\overline}
\nc\wt{\widetilde}
\nc\tboxtimes{\wt{\boxtimes}}
\nc\tstar{\wt{\star}}
\nc{\alp}{\alpha}

\nc{\ZZ}{{\mathbb Z}}
\nc{\NN}{{\mathbb N}}
\nc{\OO}{{\mathbb O}}
\renc{\SS}{{\mathbb S}}
\nc{\DD}{{\mathbb D}}
\nc{\GG}{{\mathbb G}}
\renewcommand{\AA}{{\mathbb A}}
\nc{\Fq}{{\mathbb F}_q}
\nc{\Fqb}{\ol{{\mathbb F}_q}}
\nc{\Ql}{\ol{{\mathbb Q}_\ell}}
\nc{\id}{\on{id}}
\nc\X{\mathcal X}

\nc{\Hom}{\on{Hom}}
\nc{\Lie}{\on{Lie}}
\nc{\Loc}{\on{Loc}}
\nc{\Pic}{\on{Pic}}
\nc{\Bun}{\on{Bun}}
\nc{\IC}{\on{IC}}
\nc{\Aut}{\on{Aut}}
\nc{\rk}{\on{rk}}
\nc{\Sh}{\on{Sh}}
\nc{\Perv}{\on{Perv}}
\nc{\pos}{{\on{pos}}}
\nc{\Conv}{\on{Conv}}
\nc{\Sph}{\on{Sph}}
\nc{\Sym}{\on{Sym}}
\nc{\BunBb}{\overline{\Bun}_B}
\nc{\BunNb}{\overline{\Bun}_N}
\nc{\BunTb}{\overline{\Bun}_T}
\nc{\BunBbm}{\overline{\Bun}_{B^-}}
\nc{\BunBbel}{\overline{\Bun}_{B,el}}
\nc{\BunBbmel}{\overline{\Bun}_{B^-,el}}
\nc{\Buno}{\overset{o}{\Bun}}
\nc{\BunPb}{{\overline{\Bun}_P}}
\nc{\BunBM}{\Bun_{B(M)}}
\nc{\BunBMb}{\overline{\Bun}_{B(M)}}
\nc{\BunPbw}{{\widetilde{\Bun}_P}}
\nc{\BunBP}{\widetilde{\Bun}_{B,P}}
\nc{\GUb}{\overline{G/U}}
\nc{\GUPb}{\overline{G/U(P)}}

\nc\syminfty{\on{Sym}^{\infty}}
\nc\lal{\ol{\lambda}}
\nc\xl{\ol{x}}
\nc\thl{\ol{\theta}}
\nc\nul{\ol{\nu}}
\nc\mul{\ol{\mu}}
\nc\Sum\Sigma
\nc{\oX}{\overset{\circ}{X}{}}
\nc{\hl}{\overset{\leftarrow}h{}}
\nc{\hr}{\overset{\rightarrow}h{}}
\nc{\M}{{\mathcal M}}
\nc{\N}{{\mathcal N}}
\nc{\Nch}{{\check{\mathcal N}}}

\nc{\F}{{\mathcal F}}
\nc{\D}{{\mathcal D}}
\nc{\Y}{{\mathcal Y}}
\nc{\G}{{\mathcal G}}
\nc{\E}{{\mathcal E}}
\nc{\CalC}{{\mathcal C}}
\nc\Dh{\widehat{\D}}
\renewcommand{\O}{{\mathcal O}}
\nc{\K}{{\mathcal K}}
\renewcommand{\S}{{\mathcal S}}
\nc{\T}{{\mathcal T}}
\nc{\V}{{\mathcal V}}
\renc{\P}{{\mathcal P}}
\nc{\A}{{\AA}}
\nc{\U}{{\mathcal U}}

\nc{\frn}{{\check{\mathfrak u}(P)}}

\nc{\fC}{\mathfrak C}
\nc\f{{\mathfrak f}}

\nc{\qo}{{\mathfrak q}}
\nc{\po}{{\mathfrak p}}
\nc{\s}{{\mathfrak s}}
\nc\w{\text{w}}
\renewcommand{\r}{{\mathfrak r}}

\renewcommand{\mod}{{\on{-}\mathsf{mod}}}

\nc\mathi\iota
\nc\Spec{\on{Spec}}
\nc\Mod{\on{Mod}}
\nc{\tw}{\widetilde{\mathfrak t}}
\nc{\pw}{\widetilde{\mathfrak p}}
\nc{\qw}{\widetilde{\mathfrak q}}
\nc{\jw}{\widetilde j}

\nc{\grb}{\overline{\Gr_{X^{\fset}}}}
\nc{\I}{\mathcal I}
\renewcommand{\i}{\mathfrak i}
\renewcommand{\j}{\mathfrak j}

\nc{\lambdach}{{\check\lambda}}
\nc{\Lambdach}{{\check\Lambda}{}}
\nc{\much}{{\check\mu}}
\nc{\omegach}{{\check\omega}}
\nc{\nuch}{{\check\nu}}
\nc{\etach}{{\check\eta}}
\nc{\alphach}{{\check\alpha}}

\nc{\rhoch}{{\check\rho}}

\nc{\Hb}{\overline{\H}}

\nc{\BA}{{\mathbb{A}}}
\nc{\BB}{\mathbb{B}}
\nc{\BC}{{\mathbb{C}}}
\nc{\BD}{{\mathbb{D}}}
\nc{\BE}{{\mathbb{E}}}
\nc{\BF}{{\mathbb{F}}}
\nc{\BG}{{\mathbb{G}}}
\nc{\BH}{{\mathbb{H}}}
\nc{\BI}{{\mathbb{I}}}
\nc{\BM}{{\mathbb{M}}}
\nc{\BN}{{\mathbb{N}}}
\nc{\BO}{{\mathbb{O}}}
\nc{\BP}{{\mathbb{P}}}
\nc{\BQ}{{\mathbb{Q}}}
\nc{\BR}{{\mathbb{R}}}
\nc{\BS}{{\mathbb{S}}}
\nc{\BT}{{\mathbb{T}}}
\nc{\BV}{{\mathbb{V}}}
\nc{\BZ}{{\mathbb{Z}}}

\nc{\bbone}{\mathbbm{1}}
\nc{\bbA}{{\mathbb{A}}}
\nc{\bbB}{\mathbb{B}}
\nc{\bbC}{{\mathbb{C}}}
\nc{\bbD}{{\mathbb{D}}}
\nc{\bbE}{{\mathbb{E}}}
\nc{\bbF}{{\mathbb{F}}}
\nc{\bbG}{{\mathbb{G}}}
\nc{\bbH}{{\mathbb{H}}}
\nc{\bbI}{{\mathbb{I}}}
\nc{\bbL}{{\mathbb{L}}}
\nc{\bbM}{{\mathbb{M}}}
\nc{\bbN}{{\mathbb{N}}}
\nc{\bbO}{{\mathbb{O}}}
\nc{\bbP}{{\mathbb{P}}}
\nc{\bbQ}{{\mathbb{Q}}}
\nc{\bbR}{{\mathbb{R}}}
\nc{\bbS}{{\mathbb{S}}}
\nc{\bbT}{{\mathbb{T}}}
\nc{\bbU}{{\mathbb{U}}}
\nc{\bbV}{{\mathbb{V}}}
\nc{\bbW}{{\mathbb{W}}}
\nc{\bbX}{{\mathbb{X}}}
\nc{\bbY}{{\mathbb{Y}}}
\nc{\bbZ}{{\mathbb{Z}}}

\nc{\CA}{{\mathcal{A}}}
\nc{\CB}{{\mathcal{B}}}

\nc{\CE}{{\mathcal{E}}}
\nc{\CF}{{\mathcal{F}}}
\nc{\CH}{{\mathcal{H}}}

\nc{\CL}{{\mathcal{L}}}
\nc{\CC}{{\mathcal{C}}}
\nc{\CG}{{\mathcal{G}}}
\nc{\CM}{{\mathcal{M}}}
\nc{\CN}{{\mathcal{N}}}
\nc{\CK}{{\mathcal{K}}}
\nc{\CO}{{\mathcal{O}}}
\nc{\CP}{{\mathcal{P}}}
\nc{\CQ}{{\mathcal{Q}}}
\nc{\CR}{{\mathcal{R}}}
\nc{\CS}{{\mathcal{S}}}
\nc{\CU}{{\mathcal{U}}}
\nc{\CV}{{\mathcal{V}}}
\nc{\CW}{{\mathcal{W}}}
\nc{\CX}{{\mathcal{X}}}
\nc{\CY}{{\mathcal{Y}}}
\nc{\CZ}{{\mathcal{Z}}}
\nc{\CI}{{\mathcal{I}}}

\nc{\csM}{{\check{\mathcal A}}{}}
\nc{\oM}{{\overset{\circ}{\mathcal M}}{}}
\nc{\obM}{{\overset{\circ}{\mathbf M}}{}}
\nc{\oCA}{{\overset{\circ}{\mathcal A}}{}}
\nc{\obA}{{\overset{\circ}{\mathbf A}}{}}
\nc{\ooM}{{\overset{\circ}{M}}{}}
\nc{\osM}{{\overset{\circ}{\mathsf M}}{}}
\nc{\vM}{{\overset{\bullet}{\mathcal M}}{}}
\nc{\nM}{{\underset{\bullet}{\mathcal M}}{}}
\nc{\oD}{{\overset{\circ}{\mathcal D}}{}}
\nc{\obC}{{\overset{\circ}{\mathbf C}}{}}
\nc{\obD}{{\overset{\circ}{\mathbf D}}{}}
\nc{\oA}{{\overset{\circ}{\mathbb A}}{}}
\nc{\op}{{\overset{\bullet}{\mathbf p}}{}}
\nc{\oU}{{\overset{\bullet}{\mathcal U}}{}}
\nc{\oZ}{{\overset{\circ}{\mathcal Z}}{}}
\nc{\ofZ}{{\overset{\circ}{\mathfrak Z}}{}}
\nc{\oF}{{\overset{\circ}{\fF}}}

\nc{\fa}{{\mathfrak{a}}}
\nc{\fb}{{\mathfrak{b}}}
\nc{\fc}{{\mathfrak{c}}}
\nc{\fd}{{\mathfrak{d}}}
\nc{\ff}{{\mathfrak{f}}}
\nc{\fg}{{\mathfrak{g}}}
\nc{\fgl}{{\mathfrak{gl}}}
\nc{\fh}{{\mathfrak{h}}}
\nc{\fj}{{\mathfrak{j}}}
\nc{\fl}{{\mathfrak{l}}}
\nc{\fm}{{\mathfrak{m}}}
\nc{\fn}{{\mathfrak{n}}}
\nc{\fu}{{\mathfrak{u}}}
\nc{\fp}{{\mathfrak{p}}}
\nc{\fr}{{\mathfrak{r}}}
\nc{\fs}{{\mathfrak{s}}}
\nc{\ft}{{\mathfrak{t}}}
\nc{\fz}{{\mathfrak{z}}}
\nc{\fsl}{{\mathfrak{sl}}}
\nc{\hsl}{{\widehat{\mathfrak{sl}}}}
\nc{\hgl}{{\widehat{\mathfrak{gl}}}}
\nc{\hg}{{\widehat{\mathfrak{g}}}}
\nc{\chg}{{\widehat{\mathfrak{g}}}{}^\vee}
\nc{\hn}{{\widehat{\mathfrak{n}}}}
\nc{\chn}{{\widehat{\mathfrak{n}}}{}^\vee}

\nc{\fA}{{\mathfrak{A}}}
\nc{\fB}{{\mathfrak{B}}}
\nc{\fD}{{\mathfrak{D}}}
\nc{\fE}{{\mathfrak{E}}}
\nc{\fF}{{\mathfrak{F}}}
\nc{\fG}{{\mathfrak{G}}}
\nc{\fK}{{\mathfrak{K}}}
\nc{\fL}{{\mathfrak{L}}}
\nc{\fM}{{\mathfrak{M}}}
\nc{\fN}{{\mathfrak{N}}}
\nc{\fP}{{\mathfrak{P}}}
\nc{\fU}{{\mathfrak{U}}}
\nc{\fV}{{\mathfrak{V}}}
\nc{\fX}{{\mathfrak{X}}}
\nc{\fY}{{\mathfrak{Y}}}
\nc{\fZ}{{\mathfrak{Z}}}

\nc{\bb}{{\mathbf{b}}}
\nc{\bc}{{\mathbf{c}}}
\nc{\bd}{{\mathbf{d}}}
\nc{\bbf}{{\mathbf{f}}}
\nc{\be}{{\mathbf{e}}}
\nc{\bg}{{\mathbf{g}}}
\nc{\bi}{{\mathbf{i}}}
\nc{\bj}{{\mathbf{j}}}
\nc{\bn}{{\mathbf{n}}}
\nc{\bo}{{\mathbf{o}}}
\nc{\bp}{{\mathbf{p}}}
\nc{\bq}{{\mathbf{q}}}
\nc{\bt}{{\mathbf{t}}}
\nc{\bu}{{\mathbf{u}}}
\nc{\bv}{{\mathbf{v}}}
\nc{\bx}{{\mathbf{x}}}
\nc{\bs}{{\mathbf{s}}}
\nc{\by}{{\mathbf{y}}}
\nc{\bw}{{\mathbf{w}}}
\nc{\bA}{{\mathbf{A}}}
\nc{\bK}{{\mathbf{K}}}
\nc{\bB}{{\mathbf{B}}}
\nc{\bC}{{\mathbf{C}}}
\nc{\bG}{{\mathbf{G}}}
\nc{\bD}{{\mathbf{D}}}
\nc{\bH}{{\mathbf{H}}}
\nc{\bM}{{\mathbf{M}}}
\nc{\bN}{{\mathbf{N}}}
\nc{\bO}{{\mathbf{O}}}
\nc{\bT}{{\mathbf{T}}}
\nc{\bV}{{\mathbf{V}}}
\nc{\bW}{{\mathbf{W}}}
\nc{\bX}{{\mathbf{X}}}
\nc{\bZ}{{\mathbf{Z}}}
\nc{\bS}{{\mathbf{S}}}

\nc{\sA}{{\mathsf{A}}}
\nc{\sB}{{\mathsf{B}}}
\nc{\sC}{{\mathsf{C}}}
\nc{\sD}{{\mathsf{D}}}
\nc{\sF}{{\mathsf{F}}}
\nc{\sG}{{\mathsf{G}}}
\nc{\sK}{{\mathsf{K}}}
\nc{\sM}{{\mathsf{M}}}
\nc{\sO}{{\mathsf{O}}}
\nc{\sW}{{\mathsf{W}}}
\nc{\sQ}{{\mathsf{Q}}}
\nc{\sP}{{\mathsf{P}}}
\nc{\sV}{{\mathsf{V}}}
\nc{\sS}{{\mathsf{S}}}
\nc{\sT}{{\mathsf{T}}}
\nc{\sZ}{{\mathsf{Z}}}
\nc{\sfp}{{\mathsf{p}}}
\nc{\sll}{{\mathsf{l}}}
\nc{\sr}{{\mathsf{r}}}
\nc{\bk}{{\mathsf{k}}}
\nc{\sg}{{\mathsf{g}}}

\nc{\sff}{{\mathsf{f}}}
\nc{\sfb}{{\mathsf{b}}}
\nc{\sfc}{{\mathsf{c}}}
\nc{\sd}{{\mathsf{d}}}
\nc{\se}{{\mathsf{e}}}

\nc{\BK}{{\bar{K}}}

\nc{\tA}{{\widetilde{\mathbf{A}}}}
\nc{\tB}{{\widetilde{\mathcal{B}}}}
\nc{\tg}{{\widetilde{\mathfrak{g}}}}
\nc{\tG}{{\widetilde{G}}}
\nc{\TM}{{\widetilde{\mathbb{M}}}{}}
\nc{\tO}{{\widetilde{\mathsf{O}}}{}}
\nc{\tU}{{\widetilde{\mathfrak{U}}}{}}
\nc{\TZ}{{\tilde{Z}}}
\nc{\tx}{{\tilde{x}}}
\nc{\tbv}{{\tilde{\bv}}}
\nc{\tfP}{{\widetilde{\mathfrak{P}}}{}}
\nc{\tz}{{\tilde{\zeta}}}
\nc{\tmu}{{\tilde{\mu}}}

\nc{\urho}{\underline{\rho}}
\nc{\uB}{\underline{B}}
\nc{\uC}{{\underline{\mathbb{C}}}}
\nc{\ui}{\underline{i}}
\nc{\uj}{\underline{j}}
\nc{\ofP}{{\overline{\mathfrak{P}}}}
\nc{\oB}{{\overline{\mathcal{B}}}}
\nc{\og}{{\overline{\mathfrak{g}}}}
\nc{\oI}{{\overline{I}}}

\nc{\eps}{\varepsilon}
\nc{\hrho}{{\hat{\rho}}}

\nc{\one}{{\mathbf{1}}}
\nc{\two}{{\mathbf{t}}}

\nc{\Rep}{{\mathop{\operatorname{\rm Rep}}}}
\nc{\Tot}{{\mathop{\operatorname{\rm Tot}}}}
\nc{\Ker}{{\mathop{\operatorname{\rm Ker}}}}
\nc{\Hilb}{{\mathop{\operatorname{\rm Hilb}}}}
\nc{\Ext}{{\mathop{\operatorname{\rm Ext}}}}
\nc{\CHom}{{\mathop{\operatorname{{\mathcal{H}}\it om}}}}
\nc{\GL}{{\mathop{\operatorname{\rm GL}}}}
\nc{\gr}{{\mathop{\operatorname{\rm gr}}}}
\nc{\Id}{{\mathop{\operatorname{\rm Id}}}}
\nc{\de}{{\mathop{\operatorname{\rm def}}}}
\nc{\length}{{\mathop{\operatorname{\rm length}}}}
\nc{\supp}{{\mathop{\operatorname{\rm supp}}}}

\nc{\Cliff}{{\mathsf{Cliff}}}
\nc{\Fl}{\on{Fl}}
\nc{\Fib}{{\mathsf{Fib}}}
\nc{\Coh}{{\on{Coh}}}
\nc{\coh}{{\on{coh}}}

\nc{\QCoh}{{\on{QCoh}}}
\nc{\IndCoh}{{\on{IndCoh}}}
\nc{\FCoh}{{\mathsf{FCoh}}}

\nc{\reg}{{\text{\rm reg}}}

\nc{\cplus}{{\mathbf{C}_+}}
\nc{\cminus}{{\mathbf{C}_-}}
\nc{\cthree}{{\mathbf{C}_*}}
\nc{\Qbar}{{\bar{Q}}}
\nc\Eis{\on{Eis}}
\nc\CT{\on{CT}}
\nc\Eisb{\ol\Eis{}}
\nc\Eisr{\on{Eis}^{rat}{}}
\nc\wh{\widehat}
\nc{\Def}{\on{Def_{\check{\fb}}(E)}}
\nc{\barZ}{\overline{Z}{}}
\nc{\barbarZ}{\overline{\barZ}{}}
\nc{\barpi}{\overline\pi}
\nc{\barbarpi}{\overline\barpi}
\nc{\barpip}{\overline\pi{}^+}
\nc{\barpim}{\overline\pi{}^-}

\nc{\fq}{\mathfrak q}

\nc{\fqb}{\ol{\fq}{}}
\nc{\fpb}{\ol{\fp}{}}
\nc{\fpr}{{\fp^{rat}}{}}
\nc{\fqr}{{\fq^{rat}}{}}

\nc{\hattimes}{\wh\otimes}

\nc{\bh}{{\bar{h}}}
\nc{\bOmega}{{\overline{\Omega(\check \fn)}}}

\nc{\seq}[1]{\stackrel{#1}{\sim}}

\nc{\cT}{{\check{T}}}
\nc{\cG}{{\check{G}}}
\nc{\cM}{{\check{M}}}
\nc{\cB}{{\check{B}}}
\nc{\cP}{{\check{P}}}

\nc{\ct}{{\check{\mathfrak t}}}
\nc{\cg}{{\check{\fg}}}
\nc{\cb}{{\check{\fb}}}
\nc{\cn}{{\check{\fn}}}
\nc{\cp}{{\check{\fp}}}
\nc{\cm}{{\check{\fm}}}

\nc{\cLambda}{{\check\Lambda}}

\nc{\cla}{{\check\lambda}}
\nc{\cmu}{{\check\mu}}
\nc{\cnu}{{\check\nu}}
\nc{\ceta}{{\check\eta}}

\nc{\DefbE}{{\on{Def}_{\cB}(E_\cT)}}

\nc{\imathb}{{\ol{\imath}}}
\nc{\rlr}{\overset{\longrightarrow}{\underset{\longrightarrow}\longleftarrow}}

\nc{\oBun}{\overset{\circ}\Bun}
\nc{\BunBbb}{\ol{\ol{Bun}}_B}
\nc{\BunBr}{\Bun_B^{rat}}
\nc{\BunBrsg}{\Bun_B^{rat,\on{s.g.}}}
\nc{\BunBrp}{\Bun_B^{rat,polar}}
\nc{\BunBrpbg}{\Bun_B^{rat,polar,\on{b.g.}}}
\nc{\BunBrpsg}{\Bun_B^{rat,polar,\on{s.g.}}}
\nc{\BunTrp}{\Bun_T^{rat,polar}}
\nc{\BunTrpbg}{\Bun_T^{rat,polar,\on{b.g.}}}
\nc{\BunTrpsg}{\Bun_T^{rat,polar,\on{s.g.}}}
\nc{\BunNr}{\Bun_N^{rat}}
\nc{\BunNre}{\Bun_N^{enh,rat}}
\nc{\BunTr}{\Bun_T^{rat}}
\nc{\Vect}{\on{Vect}}
\nc{\Whit}{\on{Whit}}
\nc{\bTb}{\ol{\on{CT}}}

\nc{\bTr}{\on{CT}^{rat}{}}
\nc\jmathr{\jmath^{rat}{}}
\nc{\ux}{\underline{x}}
\nc{\clambda}{{\check\lambda}}
\nc{\calpha}{{\check\alpha}}

\nc{\inftyGrpd}{{\mathsf{Grpd}_\infty}}

\nc{\fset}{\mathsf{fSet}}
\nc{\fSet}{\fset}

\nc{\LocSysG}{\LocSys_{\cG}}
\nc{\Sing}{{\on{Sing}}}
\nc{\dr}{{\on{dR}}}
\nc{\Ind}{\on{Ind}}
\nc{\Sat}{\on{Sat}}
\nc{\Ho}{\on{Ho}}
\nc{\Res}{\on{Res}}
\nc{\sotimes}{\overset{!}\otimes}
\nc{\mmod}{{\on{-}}{\mathbf{mod}}}
\nc{\Maps}{\on{Maps}}
\nc{\CMaps}{{\mathcal Maps}}
\nc{\bMaps}{{\mathbf{Maps}}}

\nc{\dgSch}{\on{DGSch}}
\nc{\dgindSch}{\on{DGindSch}}
\nc{\indSch}{\on{indSch}}
\nc{\Sch}{\mathsf{Sch}}
\nc{\affdgSch}{\on{DGSch}^{\on{aff}}}
\nc{\affSch}{\on{Sch}^{\on{aff}}}

\nc{\Groupoids}{\on{Grpd}}
\nc{\inftypic}{\infty\on{-PicGrpd}}
\nc{\inftyCat}{{\mathsf{Cat}_{\infty}}}
\nc{\MoninftyCat}{\infty\on{-Cat}^{Mon}}
\nc{\SymMoninftyCat}{\infty\on{-Cat}^{\on{SymMon}}}
\nc{\SymMonStinftyCat}{\on{DGCat}^{\on{SymMon}}}
\nc{\MonStinftyCat}{\on{DGCat}^{Mon}}
\nc{\inftystack}{\on{Stk}}
\nc{\inftystackalg}{Stk^{1\text{-}alg}}
\nc{\inftyprestack}{\on{PreStk}}
\nc{\inftydgnearstack}{\on{NearStk}}
\nc{\inftydgstack}{\on{Stk}}
\nc{\inftydgstackalg}{DGStk^{1\text{-}alg}}
\nc{\inftydgprestack}{\on{PreStk}}

\nc{\HC}{\CH\bC}

\nc{\csupp}{\supp}
\nc{\Arth}{\on{Arth}}
\nc{\ArthG}{{\on{Arth}_\cG}}
\nc{\ul}{\underline}

\nc{\Z}{\mathcal{Z}}

\nc{\calN}{\N}
\nc{\calW}{\mathcal{W}}
\nc{\calF}{\mathcal{F}}
\nc{\calH}{\mathcal{H}}
\nc{\calO}{\mathcal{O}}
\nc{\calK}{\mathcal{K}}

\nc{\Ran}{\mathsf{Ran}}
\nc{\Jets}{\on{Jets}}
\nc{\act}{\mathsf{act}}
\nc{\Av}{\mathsf{Av}}
\nc{\Ad}{\on{Ad}}
\nc{\BGRan}{BG_{\Ran}}
\nc{\codim}{\on{codim}}
\nc{\cpt}{{\on{cpt}}}
\nc{\dR}{{\on{dR}}}
\nc{\DGCat}{\mathsf{DGCat}}
\nc{\DGCatcont}{\on{DGCat}_{cont}}
\nc{\glob}{{\on{glob}}}
\nc{\loc}{{\on{loc}}}

\renewcommand{\op}{{\on{op}}}
\nc{\pt}{{\on{pt}}}
\nc{\PreStk}{{\mathsf{PreStk}}}
\nc{\Cat}{{\mathsf{Cat}}}

\nc{\ShvCat}{{\mathsf{ShvCat}}}

\nc{\restr}[2]{\left. #1 \right |_{#2}}
\nc{\uprestr}[2]{\left. #1 \right |^{#2}}

\nc{\bLoc}{{\mathbf{Loc}}}
\nc{\bGamma}{{\mathbf{\Gamma}}}
\nc{\bLocA}{\mathbf{Loc}^\A}
\nc{\bGammaA}{\mathbf{\Gamma}^\A}
\nc{\bLocB}{\mathbf{Loc}^\B}
\nc{\bGammaB}{\mathbf{\Gamma}^\B}
\nc{\bLocH}{\mathbf{Loc}^\H}
\nc{\bGammaH}{\mathbf{\Gamma}^\H}

\nc{\gen}{{\on{gen}}}
\nc{\ggen}{\on{-gen}}

\nc{\hto}{\hookrightarrow}

\nc{\ext}{\mathsf{ext}}

\nc{\ev}{\mathsf{ev}}
\nc{\rat}{\mathsf{rat}}

\nc{\usotimes}[1]{\underset{#1}{\otimes}}
\nc{\ustimes}[1]{\underset{#1}{\times}}
\nc{\uscolim}[1]{\underset{#1}{\colim}}

\nc{\ch}{{\mathfrak{ch}}}

\renc{\fD}{{\Dmod}}
\nc{\fH}{{\mathfrak{H}}}

\nc{\p}{{\mathfrak{p}}}
\renc{\r}{{\mathfrak{r}}}

\nc{\xto}{\xrightarrow}

\renc{\sec}{\section}
\nc{\enh}{{\on{enh}}}

\nc{\BunGBgen}{\Bun_G^{B\ggen}}
\nc{\BunGHgen}{\Bun_G^{H\ggen}}
\nc{\BunGNgen}{\Bun_G^{N\ggen}}

\nc{\Fun}{\mathsf{Fun}}
\nc{\End}{\mathsf{End}}

\nc{\lr}{\xymatrix{ \ar@<-0.4ex>[r] \ar@<.5ex>[l]  & } }
\nc{\rr}{\xymatrix{ \ar@<-0.2ex>[r] \ar@<.7ex>[r]  & } }
\nc{\rrr}{\xymatrix{ \ar@<.0ex>[r] \ar@<.7ex>[r] \ar@<-0.7ex>[r] & } }

\nc{\Stab}{\mathsf{Stab}}
\nc{\Orb}{\mathsf{Orb}}

\renc{\exp}{\mathit{exp}}

\renc{\q}{\mathfrak{q}}

\nc{\virg}[1]{``#1"}

\nc{\QA}[2]
{\textbf{Question:} {#1} 
\\
\textbf{Answer:} {#2}}

\renc{\bold}[1]{\boldsymbol{#1}}

\nc{\bigt}[1]{\big( #1 \big) }
\nc{\Bigt}[1]{\Big( #1 \Big) }

\nc{\extwhit}{{\CW h}(G,\mathsf{ext})}

\nc{\footcite}{\footnote}

\nc{\GA}{{G(\AA)}}
\nc{\GO}{{G(\OO)}}
\nc{\GK}{G(\KK)}

\nc{\Shv}{\mathsf{Shv}}
\nc{\inc}{\mathsf{inc}}

\nc{\Par}{\mathsf{Par}}

\renc{\i}{\mathfrak{i}}

\nc{\NA}{N(\AA)}
\nc{\VA}{V(\AA)}

\nc{\laxlim}{\text{laxlim}}

\nc{\FT}{\mathsf{FT}}

\nc{\out}{\mathsf{out}}

\nc{\hol}{\mathsf{hol}}
\nc{\Hol}{\on{Hol}}
\nc{\add}{\mathsf{add}}

\nc{\sto}{\rightsquigarrow}
\nc{\squigto}{\rightsquigarrow}

\nc{\fW}{\mathfrak{W}}

\nc{\vrho}{\varrho}

\nc{\counit}{\mathsf{counit}}
\nc{\unit}{\mathsf{unit}}
\nc{\corr}{\mathsf{corr}}
\nc{\Corr}{\mathsf{Corr}}

\nc{\IndSch}{\mathsf{IndSch}}
\nc{\Tate}{{\mathsf{Tate}}}

\nc{\surjto}{\twoheadrightarrow}

\renc{\j}{\mathfrak{j}}

\nc{\J}{\mathcal{J}}

\nc{\pro}{\mathsf{pro}}
\nc{\fty}{\mathsf{ft}}
\nc{\Pro}{\mathsf{Pro}}

\nc{\coact}{\mathsf{coact}}
\nc{\aff}{\mathsf{aff}}

\nc{\Nilp}{\on{Nilp}}

\nc{\Gch}{{\check{G}}}
\nc{\Pch}{{\check{P}}}
\nc{\Mch}{{\check{M}}}
\nc{\Qch}{{\check{Q}}}

\nc{\LL}{\mathbb{L}}

\nc{\LS}{{\on{LS}}}

\nc{\x}{\varkappa} 

\nc{\Otimes}{\boldsymbol{\otimes}}
\nc{\Times}{\boldsymbol{\times}}
\nc{\flip}{\text{<}}

\nc{\coeffRan}{\mathsf{coeff}^{\Ran}}

\nc{\Ha}{H(\sA)}
\nc{\Groups}{\mathsf{Groups}}

\nc{\Groth}{\mathsf{Groth}}

\nc{\rlto}{\rightleftarrows}

\nc{\DGCatRan}{\ShvCatCrys(\Ran)}

\nc{\longto}{\longrightarrow}

\renc{\Jets}{\mathsf{Jets}}
\nc{\mer}{\mathsf{mer}}

\nc{\W}{\mathcal{W}}

\nc{\Sect}{\mathsf{Sect}}
\renc{\Maps}{\mathsf{Maps}}

\renc{\bf}{\mathbf{f}}

\nc{\y}{\mathtt{y}}
\renc{\x}{\mathtt{x}}

\nc{\un}{{\it un}}
\nc{\indep}{\mathsf{indep}}
\nc{\CoAlg}{\mathsf{CoAlg}}

\nc{\coeff}{\mathsf{coeff}}

\nc{\R}{\mathcal{R}}

\renc{\hat}{\widehat}

\nc{\KK}{\mathbb{K}}

\nc{\Dmod}{\mathfrak{D}}

\nc{\Bshv}{\bold{\B}}
\nc{\Bind}{\H_{\indep}}
\nc{\BRan}{\H_{\Ran}}
\nc{\ARan}{\A_{\Ran}}
\nc{\Aind}{\A_{\indep}}

\nc{\GrRan}{\Gr}
\nc{\Gr}{\mathsf{Gr}}
\nc{\GrGRan}{\Gr_{G}}
\nc{\GrGind}{\Gr_{G}^{\indep}}
\nc{\Grind}[1]{\Gr_{#1}^{\indep} }
\nc{\GrGdom}{\curs{Gr}_G}

\nc{\GMapsRan}[1]{\mathsf{GMaps}(X,{#1})}
\nc{\GSectRan}[1]{\mathsf{GSect}({#1}/X)}
\nc{\GMapsind}[1]{\mathsf{GMaps}(X,{#1})^\indep}
\nc{\GSectind}[1]{\mathsf{GSect}({#1}/X)^\indep}
\nc{\GMapsdom}[1]{\curs{GMaps}(X,{#1})}
\nc{\GSectdom}[1]{\curs{GSect}({#1}/X)}

\nc{\chind}{\ch^{\indep}}
\nc{\chdom}{\curs{ch}}

\nc{\QSect}[1]{\curs{QSect}(#1/X)} 
\nc{\QMaps}[1]{\curs{QMaps}(X,#1)} 

\nc{\Zar}{\mathit{Zar}}

\nc{\loccit}{\textit{loc.$\,$cit.}}
\nc{\Crys}{\on{Crys}}
\nc{\ShvCatCrys}{\ShvCat^{\Crys}}

\nc{\BPE}{{\BP E}}
\nc{\BVE}{{\BV E}}
\nc{\BBE}{{\BB E}}

\nc{\Wh}{{{\CW}h}}

\nc{\ChiralCat}{\mathsf{ChiralCat}}
\nc{\RRep}{\mathfrak{R}ep}
\nc{\SSph}{\mathfrak{S}ph}

\nc{\tto}{\twoheadrightarrow}
\nc{\disj}{{\mathsf{disj}}}

\nc{\C}{\CC}
\nc{\Tch}{{\check{T}}}

\nc{\good}{\mathsf{good}}
\nc{\triv}{\mathsf{triv}}

\nc{\Alg}{\mathsf{Alg}}
\nc{\CAlg}{\mathsf{CAlg}}
\nc{\Spread}{\mathsf{Spread}}
\nc{\Dom}{\mathsf{Dom}}

\nc{\Jac}{\on{Jac}}
\renc{\CD}[1]{{#1}^{\on{CD}}}

\nc{\String}{\on{String}}

\renc{\min}{{\mathit{min}}}

\nc{\rrep}{\on-\!\mathbf{rep}}

\nc{\WWh}{\mathfrak{W}h}
\nc{\Grpd}{\mathsf{Grpd}}

\nc{\timesdisj}{\overset{\circ}\times}

\renc{\NA}{N(\sA)}
\nc{\chiral}{\mathsf{chiral}}

\nc{\Hopf}{\mathsf{Hopf}}

\nc{\heart}{\heartsuit}
\nc{\kk}{\mathbbm{k}} 

\nc{\HHom}{\CH{om}} 
\nc{\Cone}{\on{Cone}}

\nc{\EE}{\mathbb{E}}
\renc{\HC}{{\on{HC}}}
\nc{\HH}{{\on{HH}}}
\nc{\even}{{\on{even}}}
\nc{\SingSupp}{\on{SingSupp}}
\nc{\Supp}{\on{Supp}}

\nc{\temp}{{\on{temp}}}
\nc{\cusp}{{\on{cusp}}}
\nc{\geom}{{\on{geom}}}
\nc{\ren}{{\on{ren}}}
\nc{\naive}{{\on{naive}}}

\nc{\spec}{{\on{spec}}}

\nc{\gch}{\mathfrak{\check{g}}}

\nc{\Hecke}{\on{Hecke}}

\nc{\LSGch}{{\LS_\Gch}}
\nc{\LSPch}{{\LS_\Pch}}
\nc{\LSMch}{{\LS_\Mch}}

\nc{\Hsx}[2]{\H_{{#1} \leftarrow {#2}}}
\nc{\Hdx}[2]{\H_{{#1} \to {#2}}}
\nc{\Hcorr}[3]{ \H_{{#1} \leftarrow {#2} \to {#3}} }
\nc{\Hopcorr}[3]{ \H_{{#1} \to {#2} \leftto {#3}} }

\nc{\ICohsx}[2]{\ICohW_{{#1} \leftarrow {#2}}}
\nc{\ICohdx}[2]{\ICohW_{{#1} \to {#2}}}
\nc{\ICohcorr}[3]{ \ICohW_{{#1} \leftarrow {#2} \to {#3}} }
\nc{\ICohopcorr}[3]{ \ICohW_{{#1} \to {#2} \leftto {#3}} }

\nc{\QCohsx}[2]{\QCohW_{{#1} \leftarrow {#2}}}
\nc{\QCohdx}[2]{\QCohW_{{#1} \to {#2}}}
\nc{\QCohcorr}[3]{ \QCohW_{{#1} \leftarrow {#2} \to {#3}} }
\nc{\QCohopcorr}[3]{ \QCohW_{{#1} \to {#2} \leftto {#3}} }

\renc{\AA}{\bbA}
\nc{\Asx}[2]{\AA_{{#1} \leftarrow {#2}}}
\nc{\Adx}[2]{\AA_{{#1} \to {#2}}}
\nc{\Acorr}[3]{ \AA_{{#1} \leftarrow {#2} \to {#3}} }
\nc{\Aopcorr}[3]{ \AA_{{#1} \to {#2} \leftto {#3}} }

\nc{\Bsx}[2]{\B_{{#1} \leftarrow {#2}}}
\nc{\Bdx}[2]{\B_{{#1} \to {#2}}}
\nc{\Bcorr}[3]{ \B_{{#1} \leftarrow {#2} \to {#3}} }
\nc{\Bopcorr}[3]{ \B_{{#1} \to {#2} \leftto {#3}} }

\nc{\ICohzero}[3]{\ICoh_0 \bigt{#1 \times_{{#2}_\dR} #3}}
\nc{\IndCohzero}{\ICohzero}

\nc{\form}[3]{#1 \times_{{#2}_\dR} #3 }

\nc{\ind}{{\mathsf{ind}}}
\nc{\oblv}{{\mathsf{oblv}}}

\nc{\Aff}{\mathsf{Aff}}
\nc{\dgAff}{\Aff}

\nc{\deloop}{\mathsf{deloop}}
\renc{\loop}{\mathsf{loop}}

\nc{\coev}{\mathsf{coev}}

\nc{\bE}{\mathbf{E}}

\nc{\ShvCatH}{{\ShvCat^{\bbH}}}
\nc{\ShvCatQW}{\ShvCat^{\QCohW}}

\nc{\bbimod}{\on{-}\mathbf{bimod}}

\nc{\Tw}{\mathsf{Tw}}
\nc{\Arr}{\mathsf{Arr}}

\nc{\bDelta}{\bold\Delta}
\nc{\BiCat}{\mathsf{BiCat}}
\nc{\Seg}{\mathsf{Seg}}
\nc{\Cart}{\mathsf{Cart}}
\nc{\Bimod}{\mathsf{Bimod}}
\nc{\lax}{\mathit{lax}}

\nc{\pr}{\mathsf{pr}}

\nc{\zero}{ \{ 0 \}   }
\nc{\Perf}{\on{Perf}}

\nc{\leftto}{\leftarrow}
\nc{\lto}{\leftto}
\nc{\xlto}[1]{\xleftarrow{#1}}

\nc{\ltemp}{{}^\temp}
\nc{\lcusp}{{}^\cusp}
\nc{\TwCorr}{\mathsf{TwCorr}}

\nc{\Affover}[1]{{\Aff_{/#1}}}
\nc{\Affoverop}[1]{{( \Affover{#1})^\op}}

\nc{\AffOver}[2]{{(\Aff_{#1})_{/#2}}}

\nc{\AffOverop}[2]{{( \AffOver{#1}{#2})^\op}}

\nc{\aft}{{\mathit{aft}}}

\renc{\vert}{{\mathit{vert}}}
\nc{\horiz}{{\mathit{horiz}}}
\nc{\type}{{\mathit{type}}}
\nc{\adm}{{\mathit{adm}}}

\nc{\g}{\mathfrak{g}}

\nc{\free}{\mathsf{free}}

\nc{\Sform}{{S \times_{S_\dR} S}}
\nc{\Yform}{{\Y \times_{\Y_\dR} \Y}}
\nc{\SdR}{ {S_{\dR}}}

\nc{\laft}{{\mathit{laft}}}

\nc{\Affevcocaft}{\Aff_{\aft}^{< \infty}}
\nc{\Affaftevcoc}{\Aff_{\aft}^{< \infty}}

\nc{\Affevcoclfp}{\Aff_{\lfp}^{< \infty}}

\nc{\Schevcoclfp }{\Sch_{\lfp}^{< \infty}}
\nc{\Schevcocaft}{\Sch_{\aft}^{< \infty}}
\nc{\Schaftevcoc}{\Sch_{\aft}^{< \infty}}

\nc{\Stkevcoc}{\Stk^{< \infty}}
\nc{\Stkevcoclfp}{\Stk_{\lfp}^{< \infty}}
\nc{\Stkperfevcoclfp}{\Stk_{\mathit{perf},\lfp}^{< \infty}}
\nc{\Stkperflfp}{\Stk_{\mathit{perf},\lfp}}
\nc{\Stklfp}{\Stk_{\lfp}}

\nc{\evcoc}{\mathit{e.c.}}

\nc{\ICoh}{\IndCoh}

\nc{\citep}{\cite}

\renc{\H}{\bbH}

\nc{\uno}{\mathbbm{1}}

\nc{\CohBig}{{\Coh^{-\infty}}}

\nc{\Tang}{\mathbb{T}}

\nc{\LieAlg}{\mathsf{LieAlg}}

\nc{\Serre}{{\mathsf{Serre}}}

\nc{\MPreStk}{\mathsf{MPreStk}}

\nc{\all}{{\on{all}}}

\nc{\QCohwedge}{\bbQ^\wedge}
\nc{\ICohwedge}{\bbI^\wedge}
\nc{\ICohW}{\ICohwedge}
\nc{\QCohW}{\QCohwedge}

\nc{\ShvCatA}{\ShvCat^{\AA}}
\nc{\ShvCatB}{{\ShvCat^\B}}

\nc{\naiveto}{{\xto{\naive}}}
\nc{\conaiveto}{{\xto{\conaive}}}

\nc{\strong}{\mathit{strong}}
\nc{\costrong}{\mathit{costrong}}

\nc{\conv}{\mathit{conv}}

\nc{\Q}{\bbQ}

\nc{\bY}{\mathbf{Y}}
\nc{\Loop}{\mathsf{LOOP}}

\nc{\DG}{{\on{DG}}}

\nc{\coind}{\mathsf{coind}}

\nc{\co}{{\on{co}}}

\nc{\laftdef}{{\mathit{laft-def}}}

\nc{\qsmooth}{{\mathit{qs.smooth}}}
\nc{\smooth}{{\mathit{smooth}}}

\nc{\LKE}{\on{LKE}}
\nc{\RKE}{\on{RKE}}

\nc{\ShvCatAco}{\ShvCatA_{\co}}
\nc{\ShvCatHco}{\ShvCatH_{\co}}

\nc{\Stk}{\mathsf{Stk}}

\renc{\2}{(\infty,2)}
\renc{\1}{(\infty,1)}

\nc{\doubleCat}{\mathsf{doubleCat}}
\nc{\Spaces}{\mathcal{S}\!\mathit{paces}}

\nc{\ALG}{\mathsf{ALG}}
\nc{\MAPS}{\mathsf{MAPS}}

\nc{\CAT}{\mathsf{CAT}}

\nc{\oneCat}{{\Cat_{\1}}}
\nc{\oneCAT}{{\CAT_{\1}}}

\nc{\twoCat}{{\Cat_{\2}}}
\nc{\twoCAT}{{\CAT_{\2}}}

\nc{\DGCAT}{\mathsf{DGCAT}}
\nc{\twoCatDG}{{\CAT_{\2}^\DG}}
\nc{\twoCATDG}{{\CAT_{\2}^\DG}}

\nc{\twoCATDGw}{{\CAT_{\2, w*}^\DG}}
\nc{\twoCATDGww}{{\CAT_{\2, ww*}^\DG}}

\nc{\AlgBimod}{\Alg^{\mathit{bimod}}}
\nc{\AlgBimodDGCat}{\AlgBimod(\DGCat)}
\nc{\ALGBimod}{\ALG^{\mathit{bimod}}}
\nc{\twoAlgBimod}{\ALGBimod}

\nc{\rev}{{\on{rev}}}

\nc{\lfp}{{\mathit{lfp}}}

\nc{\RBeck}{{\on{R-BC}}}
\nc{\LBeck}{{\on{L-BC}}}

\nc{\schem}{\mathit{schem}}
\nc{\proper}{\mathit{proper}}

\nc{\res}{{\mathit{res}}}

\nc{\UQCoh}{\U^{\QCoh}}
\nc{\UQ}{\UQCoh}

\nc{\LieAlgbd}{{\on{Lie-algbd}}}

\nc{\LY}{{L\Y}}

\nc{\TangQ}{\Tang^{\QCoh}}

\nc{\Fil}{{\on{Fil}}}
\nc{\AssGr}{\on{assoc-gr}}

\nc{\Cech}{\on{Cech}}

\nc{\FormMod}{\mathsf{FormMod}}
\nc{\FormModunderStkevcoc} {\FormMod_{\Stk^{< \infty}/}^\lfp }

\nc{\vDmod}{\virg{\Dmod}}
\nc{\LSG}{\LS_G}
\nc{\LSM}{\LS_M}
\nc{\LSP}{\LS_P}
\nc{\LST}{\LS_T}
\nc{\LSB}{\LS_B}
\nc{\Gm}{{\GG_m}}
\nc{\GIT}{/\!/}
\renc{\t}{\ft}

\nc{\PP}{\bbP}
\nc{\Nglob}{\check{\N}_\glob}

\nc{\Psid}{\on{Ps-Id}}
\nc{\PsId}{\Psid}
\nc{\unshift}{\Rightarrow}
\nc{\coker}{\on{cone}}

\nc{\irred}{{\on{irred}}}
\nc{\red}{{\on{red}}}
\nc{\Spr}{{\on{Spr}}}
\nc{\DL}{\mathsf{DL}}
\nc{\St}{{\on{St}}}
\nc{\Glued}{{\mathsf{Glued}}}

\nc{\LSGchLeviirred}{\LS_\Gch^{\on{Levi-irred}}}
\nc{\LSLeviirred}{\LS_G^{\on{Levi-irred}}}

\nc{\olBun}{\ol{\Bun}}
\nc{\cl}{\on{cl}}
\nc{\stargen}{{*\on{-gen}}}
\nc{\Whitgen}{{\Whit\on{-gen}}}

\nc{\RHom}{\CH\!\on{om}}

\nc{\Poinc}{\on{Poinc}}
\nc{\Betti}{\on{Betti}}
\nc{\LSGcoarse}{\LS_{G,\on{coarse}}}

\nc{\mon}{{\Gm\on{-mon}}}

\nc{\TBunG}{\sT_{\Bun_G}}
\nc{\lperp}{{}^\perp}

\nc{\Levi}{\on{Levi}}
\nc{\Verdier}{{\on{Verdier}}}

\nc{\pscpt}{{\on{ps-cpt}}}

\nc{\Calk}{\on{Calk}}

\nc{\ICohNSt}{\ICoh_{\N + \St}(\LSG)}
\nc{\NSt}{\ICohNSt}

\nc{\DmodBig}{\Dmod_{!+*}}

\nc{\WSph}{\on{WS}}

\nc{\backsl}{\backslash}

\nc{\AJ}{\on{AJ}}

\nc{\GR}{G(R)}

\nc{\B}{{\bbone_{\Sph_G}^\temp}}
\nc{\unittempG}{{\bbone_{\Sph_G}^\temp}}
\nc{\unittempM}{{\bbone_{\Sph_M}^\temp}}

\nc{\BunPinfty}
{{}_{x,\infty} \wt\Bun_P}

\nc{\loccpt}{\on{loc-cpt}}
\nc{\loccoh}{\loccpt}

\nc{\dom}{\on{dom}}
\nc{\Mat}{\mathit{Mat}}

\nc{\adj}{\on{adj}}

\nc{\bigcell}[1]{\wt{#1}}

\nc{\Hren}{H_{\BM}}
\renc{\BM}{\on{BM}}

\nc{\univ}{\on{univ}}
\nc{\sing}{\on{sing}}

\nc{\new}{\on{new}}

\nc{\Gcirc}{G^\circ}

\nc{\ppart}{(\!(t)\!)}

\nc{\equivariant}{\on{-equiv}}

\newcommand\blfootnote[1]{%
  \begingroup
  \renewcommand\thefootnote{}\footnote{#1}%
  \addtocounter{footnote}{-1}%
  \endgroup
}

\nc{\NchGlob}{\Nch^{\glob}}

\nc{\sR}{\mathsf{R}}

\begin{document}


\title{Tempered D-modules and Borel-Moore homology vanishing}
\author{Dario Beraldo}

\begin{abstract}
We characterize the tempered part of the automorphic Langlands category $\Dmod(\Bun_G)$ using the geometry of the big cell in the affine Grassmannian. 
We deduce that, for $G$ non-abelian, tempered D-modules have no de Rham cohomology with compact supports.
The latter fact boils down to a concrete statement, which we prove using the Ran space and some explicit t-structure estimates: for $G$ non-abelian and $\Sigma$ a smooth affine curve, the Borel-Moore homology of the indscheme $\Maps(\Sigma,G)$ vanishes.
\end{abstract}

\maketitle

\blfootnote{\address{UCL, London, UK; \textit{email}: d.beraldo@ucl.ac.uk}. \\
ORCID: 0000-0002-4881-9846. }

\sec{Introduction}

\ssec{Overview}

The present paper is devoted to the study of the tempered condition appearing on the automorphic side of the geometric Langlands conjecture.
In this overview, we recall the statement of the geometric Langlands conjecture, review how the tempered condition comes about and explain why this condition is important. We will then state the goals of the paper.

\sssec{}

Let $G$ be a connected reductive group and $X$ a smooth complete curve, both defined over an algebraically closed field $\kk$ of characteristic zero.
Denote by $\Bun_G:=\Bun_G(X)$ the stack of $G$-bundles on $X$ and by $\Dmod(\Bun_G)$ the differential graded (DG, from now on) category of D-modules on it.
This is the DG category appearing on the automorphic side of the geometric Langlands correspondence.

For details on the geometry of $\Bun_G$ and on the DG category of D-modules on a stack, we recommend the papers \cite{DG-cptgen}, \cite{finiteness} and \cite{contract}.

\sssec{}

The geometric Langlands conjecture calls for an equivalence between $\Dmod(\Bun_G)$ and a different-looking DG category whose definition explicitly involves $\Gch$, the Langlands dual group of $G$.
At first approximation, the candidate is $\QCoh(\LSGch)$, the DG category of quasi-coherent sheaves on the stack $\LSGch := \LSGch(X)$ of de Rham $\Gch$-local systems on $X$. 
This is the so-called spectral side of the conjecture.
For details on the definition of $\LSGch$, and in particular for its derived nature, we refer to \cite[Section 2.11]{BD-quantization} and \cite[Section 10]{AG1}.

\sssec{}

As in \cite[Section 1]{AG1}, the \virg{best hope} form of the geometric Langlands conjecture is the statement that these two DG categories are equivalent:
\begin{equation} \label{eqn:best hope}
\Dmod(\Bun_G) 
\stackrel ? \simeq
\QCoh(\LSGch).
\end{equation}
While this is known to be true for $G$ abelian (see \cite[Remark 11.2.7]{AG1}, as well as \cite{Laumon1, Laumon2}, \cite{Roth1, Roth2}), it is false for more general groups: two reasons for this failure are explained in \cite[Section 1.1.2]{AG1} and \cite[Section 0.2.1]{Outline}. Both reasons point at the fact that the spectral side of \eqref{eqn:best hope} is too small to match the automorphic side. 
 To fix this discrepancy, one either enlarges the spectral side or shrinks the automorphic one.
 
 \sssec{}

A viable candidate for the enlarged spectral DG category was introduced in \cite{AG1}; to define it, one needs the theory of ind-coherent sheaves (developed in \cite{ICoh}) and the theory of singular support for coherent sheaves on quasi-smooth stacks (see \cite[Sections 2-9]{AG1}). We will briefly discuss this material in Section \ref{sssec:review of singular support}; for now, we just need to know that:
\begin{itemize}
\item
 $\IndCoh(\LSGch)$, the DG category of ind-coherent sheaves on $\LSGch$, contains $\QCoh(\LSGch)$ as a full subcategory;
\item
ind-coherent sheaves on $\LSGch$ can be classified according to their singular support, with possible singular supports being closed subsets of the space of \emph{geometric Arthur parameters} 
\begin{equation} \label{eqn:Arthur params}
\Arth_\Gch := \{(\sigma, A) | \sigma \in \LSGch, A \in H^0(X_\dR, \gch_{\sigma})\}.
\end{equation}
\end{itemize}

\sssec{}

Among such closed subsets, consider the \emph{global nilpotent cone} 
$$
\NchGlob := \{(\sigma, A) \in \Arth_\Gch \ | \, \mbox{$A$ is nilpotent}\}
$$
and let $\ICoh_{\NchGlob}(\LSGch)$ be the full subcategory of $\ICoh(\LSGch)$ spanned by objects with singular support contained in $\NchGlob$.
This is a DG category that sits between $\QCoh(\LSGch)$ and $\ICoh(\LSGch)$.

\sssec{}

The corrected version of the geometric Langlands conjecture states that
\begin{equation} \label{eqn:official geom Langlands }
\Dmod(\Bun_G) \stackrel ? \simeq
\ICoh_{\NchGlob}(\LSGch).
\end{equation}
After \cite{AG1}, the above is the \emph{official form} of the conjecture. It should come with a series of compatibilities (with the Hecke action, with the Whittaker normalization, with Eisenstein series) that we do not discuss here.

\sssec{}

The alternative way to correct \eqref{eqn:best hope} is to shrink the automorphic side. This form of the conjecture, to be called the \emph{tempered geometric Langlands conjecture}, was also introduced in \cite{AG1}. It calls for an equivalence
\begin{equation} \label{eqn:temp geom Langlands }
\ltemp\Dmod(\Bun_G) 
\stackrel ? \simeq
\QCoh(\LSGch),
\end{equation}
where $\ltemp\Dmod(\Bun_G)$ is the full-subcategory of $\Dmod(\Bun_G)$ consisting of \emph{tempered} D-modules. 
%
We will recall the definition of the tempered condition in Section \ref{ssec:intro-temp}; meanwhile, let us explain why the latter form of the conjecture is more fundamental than the official one. This has to do with the gluing statements appearing on the two sides of the conjecture.

\sssec{}

In \cite{AG2} and \cite{strong-gluing}, it is proven that $\ICoh_{\NchGlob}(\LSGch)$ can be reconstructed using $\QCoh(\LSGch)$ as well as similar DG categories for smaller Levi subgroups of $\Gch$. The details of those two papers are rather technical, but luckily we do not need them here; very informally%
\footnote{Warning: the usage of the symbol $\Glue$ here does not match with the one adopted in \cite{AG2}.}%
, we have
\begin{equation} \label{eqn:spectral gluing}
\ICoh_{\NchGlob}(\LSGch)
\simeq
\Glue_{\Mch \subseteq \Gch}
\Bigt{
\QCoh(\LSMch)
}
,
\end{equation}
where $\Mch$ runs through the poset of standard Levi subgroups of $\Gch$, while the symbol $\Glue$ means \virg{glue the DG categories $\QCoh(\LSMch)$ in a certain precise and explicit way that we do not explain in this paper}.
This is the content of the \emph{spectral gluing theorem}.
\sssec{}

The official version of the geometric Langlands conjecture then predicts that a similar decomposition must occur on the automorphic side; in other words, we expect the following \emph{automorphic gluing} statement:
\begin{equation} \label{eqn:automorphic gluing}
\Dmod(\Bun_G)
\simeq
\Glue_{M \subseteq G}
\Bigt{
\ltemp\Dmod(\Bun_M)
}
.
\end{equation}
For more details on the latter, we refer to \cite[Section 1.11]{shvcatHH} and \cite[Section 1.4]{strong-gluing}.
While \eqref{eqn:spectral gluing} is settled, the equivalence \eqref{eqn:automorphic gluing} is still in progress. Namely, it is possible to properly define the glued DG category and it is easy to write down a functor from $\Dmod(\Bun_G)$ to that glued DG category. The delicate part is proving that such a functor is an equivalence: this requires a good understanding of the tempered condition. The present paper (in particular, Theorem \ref{mainthm:B as unit of Sph-temp} below) is the first step towards our ongoing proof of \eqref{eqn:automorphic gluing}.

\sssec{}

These two gluing results are crucial for the current developments of the geometric Langlands program. Indeed, at the time of writing, the only publicly known strategy%
\footnote{More precisely, the strategy outline here is a slight modification, proposed by the author, of Gaitsgory's one. In our version, we highlight the role played by tempered categories, while keeping Whittaker categories hidden as much as possible. For instance, the automorphic glued DG category is a substitute for the extended Whittaker category.} to prove the official equivalence (outlined in \cite{Outline}) consists roughly of two parts: first, prove tempered geometric Langlands for all Levi subgroups of $G$, including $G$ itself; second, assemble these equivalences to match the right-hand-sides of \eqref{eqn:spectral gluing} and \eqref{eqn:automorphic gluing}.
In particular, both parts of this strategy require working with tempered objects.
One obstacle with this approach is that the definition of the tempered condition, to be reviewed in Section \ref{ssec:intro-temp}, is clumsy and hard to deal with: it uses the derived Satake equivalence, and in particular the geometry of the spectral side, in an essential way.

\sssec{}

The \emph{first goal} of the present work is to remove this obstacle: we will provide a novel characterization of the tempered condition that is purely automorphic (in other words, a characterization that only involves the geometry of $\Bun_G$, not the one of $\LSGch$).
This is the content of Theorem \ref{mainthm:B as unit of Sph-temp}.

The \emph{second goal} is to illustrate that this new characterization is useful in practice: we use it to settle a conjecture of D. Gaitsgory and V. Lafforgue concerning the anti-temperedness of the dualizing sheaf $\omega_{\Bun_G}$. This is the content of Theorem \ref{mainthm: omega antitemp}, which in turn hinges on the Borel-Moore homology computation of Theorem \ref{mainthm:Hren vanishing}.
However, we could not prove the latter directly (except when $G= GL_n$ or $SL_n$).

The \emph{third goal} of the paper is the proof of Theorem \ref{mainthm:omega-inf-connective}, a t-structure estimate that quickly yields Theorem \ref{mainthm:Hren vanishing}. We also prove Theorem \ref{mainthm:estimate with equations}, which is another (easier) t-structure estimate of similar type.

\ssec{The first two main results}

Let us proceed to a more precise account of our results. As mentioned, the following statement was a conjecture of D. Gaitsgory, also hinted at by V. Lafforgue in \cite[Section 5 and Page 1]{VLaff}.

\begin{mainthm} \label{mainthm: omega antitemp}
For $G$ a reductive group with semisimple rank $\geq 1$, the dualizing sheaf $\omega_{\Bun_G} \in \Dmod(\Bun_G)$ is \emph{anti-tempered}, that is, right orthogonal to $\ltemp\Dmod(\Bun_G)$.
\end{mainthm}

Explicitly, the theorem states that $\RHom_{\Dmod(\Bun_G)}(\F,\omega_{\Bun_G}) \simeq 0$ whenever $\F$ is tempered. Recall from \citep[Corollary 5.3.2]{contract} that the $!$-pushforward along the structure projection $p_{\Bun_G}: \Bun_G \to \pt := \Spec(\kk)$ is well-defined on the entire $\Dmod(\Bun_G)$. Thus, by adjunction, Theorem \ref{mainthm: omega antitemp} is equivalent to:

\begin{mainthmbis}{mainthm: omega antitemp} \label{mainthm: antitemp-reformulation}
If $\F \in \Dmod(\Bun_G)$ is tempered, then $(p_{\Bun_G})_!(\F) \simeq 0$.
\end{mainthmbis}

\begin{rem}
We immediately deduce that, for $G$ as above and any $G$-bundle $i_E: \Spec(\kk) \to \Bun_G$, the object $(i_E)_!(\kk)$ is not tempered. 
On the other hand, we conjecture that the D-modules $(i_E)_*(\kk)$ are all tempered. The latter statement does not follow easily from the results of this paper and it will be treated elsewhere.
\end{rem}

\begin{rem}
The implication of Theorem \ref{mainthm: antitemp-reformulation} is not reversible in general: $(p_{\Bun_G})_!(\F) \simeq 0$ does not imply that $\F$ is tempered. 
However, we expect that $\F \in \ltemp\Dmod(\Bun_G) \iff (p_{\Bun_G})_!(\F) \simeq 0$ for $G$ of semisimple rank equal to one.
If this is true, then Theorem \ref{mainthm: antitemp-reformulation} provides a very pleasant characterization of tempered D-modules in semisimple rank $1$, and thus a very pleasant correction of the best hope conjecture \eqref{eqn:best hope} in that case.
\end{rem}

\sssec{}

We will prove Theorem \ref{mainthm: omega antitemp} by establishing two other main results, Theorems \ref{mainthm:Hren vanishing} and \ref{mainthm:B as unit of Sph-temp}, which are possibly more interesting than Theorem \ref{mainthm: omega antitemp} itself. 
The first of these results is the following concrete statement.

\begin{mainthm} \label{mainthm:Hren vanishing} 
Let $\Sigma$ be a smooth affine curve. For $G$ a reductive group with semisimple rank $\geq 1$, the Borel-Moore homology $\Hren(G[\Sigma])$ of the mapping indscheme $G[\Sigma] :=\Maps(\Sigma, G)$ vanishes.
\end{mainthm}

\begin{rem}[Related results]
The ordinary homology of $G[\Sigma]$ was computed by C. Teleman in \cite{CT}. The homology of $G[\Sigma]^\gen$, the space of rational (alias: generic) maps from $\Sigma$ to $G$, was computed by D. Gaitsgory in \cite{contract}.
\end{rem}

\sssec{}

Let us look at $G[\Sigma]$ and at its Borel-Moore homology more closely. First off, note that $\A^N[\Sigma]$ is an indscheme isomorphic to $\A^\infty := \colim_{m \geq 0} \A^m$.
Indeed,
$$
\A^N[\Sigma] \simeq \colim_{d \geq 0} {\A^N[\Sigma]}_{\leq d},
$$
where $\A^N[\Sigma]_{\leq d}$ is the finite dimensional vector space of maps whose poles at the points at infinity of $\Sigma$ have order bounded by $d$; by Riemann-Roch, the dimension of $\A^N[\Sigma]_{\leq d}$ increases to $\infty$ with $d$.

Next, recall that any affine scheme $Y$ of finite type (and in particular $G$) can be realized as a closed subscheme of $\A^N$. Since the induced map $Y[\Sigma] \hto \A^N[\Sigma] \simeq \A^\infty$ is a closed embedding, it follows that $Y[\Sigma]$ is an ind-affine indscheme of ind-finite type.

\sssec{}

The Borel-Moore homology of a scheme $Y$ of finite type can be defined using D-modules: we set 
$$
\Hren(Y) :=(p_Y)_{*,\dR} (\omega_Y),
$$
where $\omega_Y \in \Dmod(Y)$ is the dualizing D-module and $(p_Y)_{*,\dR} $ the functor of de Rham cohomology.
It follows formally that $\Hren$ is covariant with respect to proper maps, hence it is well-defined on indschemes (of ind-finite type).
For example, for $\A^\infty$ we have 
$$
\Hren(\A^\infty) \simeq
 \uscolim{n \geq 0} \; \Hren(\A^n) 
 \simeq 
 \uscolim{n \geq 0} \; \kk[2n] 
 \simeq 0.
$$

\sssec{}

The proof of Theorem \ref{mainthm:Hren vanishing} will be discussed in Section \ref{ssec:intro-BM vanishing}. Meanwhile, let us explain how to deduce Theorem \ref{mainthm: omega antitemp} from Theorem \ref{mainthm:Hren vanishing}. For this, we must first digress and recall the definition of the tempered condition.

\ssec{Tempered objects} \label{ssec:intro-temp}

The phenomenon of temperedness (and non-temperedness) was first observed in \cite[Sections 1.1.10 and 12.1]{AG1}.
It arises as a consequence of three facts: the Hecke action on $\Dmod(\Bun_G)$, the derived Satake theorem, the discrepancy between ind-coherent sheaves and quasi-coherent sheaves on a quasi-smooth stack.
Let us review these facts in order.

\sssec{}

Let $\GK := G (\!( t)\!)$ and $\GO := G[[t]]$ be the loop group and the arc group of $G$, see for instance \cite[Definition 1]{Faltings}.
The Hecke action is a certain natural (once a point $x \in X$ and a local coordinate at $x$ have been chosen) action of the spherical monoidal DG category 
$$
\Sph_G := \Dmod(\GO \backsl \GK/\GO)
$$
on $\Dmod(\Bun_G)$.
The actual definition of the Hecke action is recalled and used in Section \ref{ssec:Hecke action of B}.

To fix the conventions, we regard $\Dmod(\Bun_G)$ as acted upon by $\Sph_G$ from the left. 
Similarly, for $\Gr_G := \GO \backsl \GK$ the affine Grassmannian, we regard $\Dmod(\Gr_G)$ as acted on by $\Sph_G$ from the left (and compatibly by $\GK$ from the right).
We denote by $\bbone_{\Sph_G}$ the unit object of $\Sph_G$, described explicitly in Section \ref{sssec:unit of Sph concrete} below.

\sssec{}

The next ingredient is the derived Satake theorem (see \cite{BF} and \cite[Section 12]{AG1}), that is, the description of $\Sph_G$ in Langlands dual terms. To appreciate this theorem, a certain familiarity with ind-coherent sheaves on quasi-smooth stacks and with the theory of singular support is desirable: we refer to Section \ref{sssec:review of singular support} for the main tenets of these theories and to \cite[Section 12]{AG1} for the full treatment.

\begin{thm}[Derived Satake] \label{thm:derived satake}
There exists a canonical monoidal\footnote{The monoidal structure on $ \ICoh_\Nch(\Omega \gch / \Gch)$ is described in \cite[Section 12.5]{AG1} and reviewed in Section \ref{sssec:monoidal str on Sph spectral}.}
equivalence
\begin{equation} \label{intro:derived Sat}
 \Sat_G :
 \ICoh_\Nch(\Omega \gch / \Gch)
 \xto{\simeq}
 \Sph_G.
\end{equation}
\end{thm}

\begin{rem}
We should stress the fact that the core of the proof of the geometric Satake is the second equivalence of \cite[Theorem 5]{BF}. As explained in \cite[Section 12.5]{AG1}, this equivalence is related to 
\eqref{intro:derived Sat} by renormalization and Koszul duality.
\end{rem}

\sssec{}

In the above formula, $\Omega \gch$ denotes the self-intersection of the origin of the vector space $\gch$, that is, the derived scheme $\pt \times_{\gch} \pt \simeq \Spec \Sym(\gch^*[1])$. It is equipped with a $\Gch$-action induced by the usual (co)adjoint action. The quotient stack $\Omega \gch/\Gch$ is quasi-smooth with space of singularities equal to $\gch^*/\Gch$.
Hence, we can consider ind-coherent sheaves on $\Omega \gch/\Gch$ with singular support contained in any chosen closed conical $G$-invariant subset of $\gch^*$.
In particular, the choice of the nilpotent cone $\Nch \subseteq \gch^*$ yields the DG category appearing on the LHS of Theorem \ref{thm:derived satake}.
On the other hand, the choice of $0 \in \gch^*$ yields the DG category $\QCoh(\Omega \gch/\Gch)$.
These two DG categories are related by a natural colocalization (that is, an adjunction with fully faithful left adjoint)
$$ 
\begin{tikzpicture}[scale=1.5]
\node (a) at (0,1) {$\QCoh(\Omega \gch/\Gch)$};
\node (b) at (2.7,1) {$\ICoh_\Nch(\Omega \gch/\Gch)$.};
\path[right hook ->,font=\scriptsize,>=angle 90]
([yshift= 1.5pt]a.east) edge node[above] {$\Xi_{0 \to \Nch}$} ([yshift= 1.5pt]b.west);
\path[->>,font=\scriptsize,>=angle 90]
([yshift= -1.5pt]b.west) edge node[below] {$\Psi_{0 \to \Nch}$} ([yshift= -1.5pt]a.east);
\end{tikzpicture}
$$

\sssec{}

Define $\ltemp\Sph_G$ to be the full subcategory of $\Sph_G$ corresponding to $\QCoh(\Omega \gch / \Gch)$ under derived Satake.
By construction, there is a colocalization 
$$ 
\begin{tikzpicture}[scale=1.5]
\node (a) at (0,1) {$\ltemp\Sph_G$};
\node (b) at (2,1) {$\Sph_G$,};
\path[right hook ->,font=\scriptsize,>=angle 90]
([yshift= 1.5pt]a.east) edge node[above] {$\Xi_{0 \to \Nch}$} ([yshift= 1.5pt]b.west);
\path[->>,font=\scriptsize,>=angle 90]
([yshift= -1.5pt]b.west) edge node[below] {$\Psi_{0 \to \Nch}$} ([yshift= -1.5pt]a.east);
\end{tikzpicture}
$$
where, abusing notation, we have denoted the two adjoint functors with the same symbols as above.

\sssec{}

For $\C$ a DG category with a left action of $\Sph_G$, we set
$$
\ltemp\C := \ltemp\Sph_G \usotimes{\Sph_G} \C.
$$
As above, and abusing notation again, there is a colocalization
\begin{equation} \label{diag:temp-C}
\begin{tikzpicture}[scale=1.5]
\node (a) at (0,1) {$\ltemp\C$};
\node (b) at (1.5,1) {$\C$.};
\path[right hook ->,font=\scriptsize,>=angle 90]
([yshift= 1.5pt]a.east) edge node[above] {$\Xi_{0 \to \Nch}$} ([yshift= 1.5pt]b.west);
\path[->>,font=\scriptsize,>=angle 90]
([yshift= -1.5pt]b.west) edge node[below] {$\Psi_{0 \to \Nch}$} ([yshift= -1.5pt]a.east);
\end{tikzpicture}
\end{equation}
We always regard $\ltemp\C$ as a full subcategory of $\C$ via the functor $\Xi_{0 \to \Nch}$.

\begin{defn}

We say that an object of $\C$ is \emph{tempered} if it belongs to $\ltemp\C$. We say that an object of $\C$ is \emph{anti-tempered} iff it is annihilated by the projection $\Psi_{0 \to \Nch}: \C \tto \ltemp\C$.
Equivalently, by adjunction, $c \in \C$ is anti-tempered iff $\RHom_\C(t,c) \simeq 0 $ for all $t \in \ltemp\C$.
\end{defn}

\begin{rem}

The endofunctor $\Xi_{0 \to \Nch} \circ \Psi_{0 \to \Nch} : \C \to \C$ will be often called the \emph{temperization functor}, since it is the projector onto the tempered subcategory. 
\end{rem}

\sssec{}

The above construction, applied to the Hecke action of $\Sph_G$ on $\Dmod(\Bun_G)$ at a chosen point $x \in X$, yields the DG category $\ltemp\Dmod(\Bun_G)$ we are interested in.

\medskip

In principle, a different choice of $x \in X$ might yield a different DG category. Thus, to be precise, we should write ${}^{x \on{-temp}}\Dmod(\Bun_G)$ in place of $\ltemp\Dmod(\Bun_G)$. However, \cite[Conjecture 12.8.5]{AG1} states that  ${}^{x \on{-temp}}\Dmod(\Bun_G)$ ought to be independent of the choice of the point $x \in X$. See \cite[Section 1.4.2]{shvcatHH} for a sketch of the proof of this statement.
Regardless of this conjecture and of its solution, our proof of Theorem \ref{mainthm: omega antitemp} will show that $\omega_{\Bun_G}$ is right-orthogonal to ${}^{x \on{-temp}}\Dmod(\Bun_G)$ \emph{for any $x$}.

\sssec{}

Denote by $\B$ the temperization of the unit $\bbone_{\Sph_G}$, that is, the object
$$
\B 
:=
\Xi_{0 \to \Nch} \circ \Psi_{0 \to \Nch} 
(
\bbone_{\Sph_G}
).
$$
We emphasize that this object is not very explicit, since the functors $\Xi_{0 \to \Nch}$ and $\Psi_{0 \to \Nch} $ are defined using the derived Satake equivalence. (Our theorem below will make it explicit.)

For $\C$ a DG category endowed with a $\Sph_G$-action, indicate by the symbol $\star$ the action of $\Sph_G$ on $\C$.
By construction, the temperization functor coincides with the functor $\B \star -: \C \to \C$.
Hence, we immediately deduce that:
\begin{itemize}
\item
an object $c \in \C$ is tempered iff it is isomorphic to $\B \star c$;

\smallskip

\item
an object $d \in \C$ is anti-tempered iff $\B \star d \simeq 0$.
\end{itemize}

\sssec{} \label{sssec:intro-reformulate-thmA}

In view of the second item above, the idea of the proof of Theorem \ref{mainthm: omega antitemp} is clear: as a first step, we should describe $\B$ explicitly (that is, only in terms of $\Sph_G$, without appealing to geometric Satake at all) and then, as a second step, we should prove that 
$$
\B \star \omega_{\Bun_G} \simeq 0.
$$
The first step is exactly the content of Theorem \ref{mainthm:B as unit of Sph-temp} below, while the second one will turn out to be a quick consequence of Theorem \ref{mainthm:Hren vanishing} in the special case of $\Sigma = \A^1$.

\ssec{The tempered unit of the spherical category}

The explicit description of $\B$ is the subject of our next main result. 

\sssec{} \label{sssec:unit of Sph concrete}

Let us first describe the monoidal unit $\bbone_{\Sph_G} \in \Sph_G$. It is given by the de Rham pushforward of $\omega_{\pt/\GO}$ along the closed embedding
$$
\pt/\GO
\simeq
\GO \backsl \GO/\GO
\hto
\GO \backslash \GK /\GO.
$$
Note that $\GO \simeq G \ltimes H$, where $H$ the first congruence subgroup of $\GO$. Since $H$ is pro-unipotent, pullback along $\pt/\GO \to \pt/G$ induces an equivalence $\Dmod(\pt/\GO) \simeq\Dmod(\pt/G)$ that sends $\omega_{\pt/\GO}$ to $\omega_{\pt/G} \in \Dmod(\pt/G)$.

\sssec{}

Now let $\GR \subseteq \GK$ denote the negative part of the loop group $\GK := G \ppart$, that is, the group indscheme $\GR := G[t^{-1}]$. 
Consider the tautological map 
$$
f:
G \backsl \GR / G
\longto
\GO \backslash \GK /\GO,
$$
and the associated pullback at the level of D-modules:
$$
f^!:
\Sph_G 
\longto
\Dmod
(
G \backsl \GR / G).
$$

\begin{mainthm} \label{mainthm:B as unit of Sph-temp}
There is a canonical isomorphism
$$
\B
\simeq
(f^!)^R
\bigt{
\omega_{G \backslash G(R) /G}
}
$$
in $\Sph_G$, where $(f^!)^R$ is the right adjoint to $f^!$.
\end{mainthm}

\sssec{}

It turns out that $(f^!)^R$ agrees with the \emph{renormalized pushforward} $f_{*,\ren}$ along $f$. The latter notion will be discussed in Section \ref{ssec:self-duality}; in any case, the functor $(f^!)^R$ can be described really explicitly as follows. 
Recall, \cite{Faltings}, that the map 
$$
\Gr_G^\circ :=
G \backsl G(R) \simeq \GO \backsl \GO G(R) 
\longto
 \GO \backsl \GK =: \Gr_G
$$
is an open embedding, the inclusion of the \virg{big cell} of the affine Grassmannian.
Hence $f$ is the composition of an open embedding with a quotient by a pro-unipotent group (the first congruence subgroup of $\GO$):
$$
\Gr_G^\circ/G 
=
G \backsl G(R) / G
\stackrel j \longhookrightarrow
\Gr_G / G
\tto
\Gr_G / \GO
=
\GO \backslash \GK /\GO.
$$
Thus, $f^! \simeq j^! \circ \oblv^{G \to \GO}$ and its right adjoint is the composition $(f^!)^R \simeq \Av_*^{G \to \GO} \circ j_{*,\dR}$. 
Here we have used the notation of \cite{thesis} for group actions on DG categories. This notation is reviewed in Section \ref{ssec:group actions on cats}: in short, $\oblv^{G \to \GO}$ is the functor that forgets the $\GO$-invariance while retaining only the residual $G$-invariance; $\Av_*^{G \to \GO}$ (the $*$-averaging functor) is its continuous right adjoint.

\begin{example}
Suppose for a moment that $G=T$ is a torus. In this case, the nilpotent cone equals the origin: this implies that every object of $\Sph_T$ is tempered. In particular, $\bbone^{\temp}_{\Sph_T}$ must agree with $\bbone_{\Sph_T}$. Let us verify this fact using the formula of Theorem \ref{mainthm:B as unit of Sph-temp}. 
The key observation is that, at the reduced level, $T(R) \simeq T$. 
It follows that $\Dmod(T \backsl T(R) /T) \simeq \Dmod(\pt/T)$ and that $j$ coincides with the closed embedding $i$ induced by the unit point of $\Gr_T$. 
We obtain that
$$
\bbone_{\Sph_T}^\temp \simeq (f^!)^R(\omega_{\pt/T}) \simeq i_{*,\dR}(\Av_*^{T \to T(\OO)}(\omega_{\pt/T}))
\simeq
i_{*,\dR}(\omega_{\pt/T(\OO)}),
$$
which is indeed the unit $\bbone_{\Sph_T}$.
\end{example}

\sssec{}

Even though the statement of Theorem \ref{mainthm:B as unit of Sph-temp} involves only the automorphic version of the spherical category, most of the work takes place on the spectral side: it amounts to computing the Serre functor of the DG category $\ICoh_\Nch(\Omega \gch/\Gch)$. This computation will be then transferred to $\Sph_G$ using derived Satake and its relation with geometric Langlands for $X=\PP^1$.
A detailed outline of the proof of Theorem \ref{mainthm:B as unit of Sph-temp} appears in Section \ref{ssec:outline of proof of theorem C}.

\sssec{}

In Section \ref{sec:characterize temp via ThmC, prove ThmA}, we will use Theorem \ref{mainthm:B as unit of Sph-temp} to see that Theorem \ref{mainthm: omega antitemp} is a quick corollary of the following Borel-Moore homology vanishing result: $\Hren(\GR) \simeq 0$, as soon as $G$ is not a torus. Since $\GR \simeq G[\A^1]$, such a statement is the simplest instance of Theorem \ref{mainthm:Hren vanishing}.

\ssec{Borel-Moore homology vanishing} \label{ssec:intro-BM vanishing}

Let us now comment on the proof of Theorem \ref{mainthm:Hren vanishing}.

\sssec{}

The DG category of D-modules on an indscheme of ind-finite type $\Y$ admits a natural t-structure, reviewed in Section \ref{defn t-structure}.
We say that an object of $\Dmod(\Y)$ is \emph{infinitely connective} if it belongs to the full subcategory
$$
\Dmod(\Y)^{\leq -\infty}:= \bigcap_{m \gg 0} \Dmod(\Y)^{\leq - m}.
$$
When $Y$ is a scheme of finite type, one easily proves that $\Dmod(Y)^{\leq -\infty} \simeq 0$. On the other hand, the theorems below will exhibit several indschemes $\Y$ for which $\Dmod(\Y)^{\leq -\infty}$ is nontrivial.

\sssec{}

We will deduce the vanishing of $\Hren(G[\Sigma])$ from the ind-affineness of $G[\Sigma]$ and the following t-structure estimate.

\begin{mainthm}\label{mainthm:omega-inf-connective}
For $G$ a reductive group of semisimple rank $\geq 1$, the dualizing sheaf $\omega_{G[\Sigma]} \in \Dmod(G[\Sigma])$ is infinitely connective.
 \end{mainthm}
 
 Theorem \ref{mainthm:omega-inf-connective} has the following immediate consequence:

\begin{cor} \label{cor:GrG infinitely connective}
For $G$ as above, the dualizing sheaves of $\Gr_G$ and of $\Gr_G^\circ$ are infinitely connective.
\end{cor}

\begin{proof}
Since the t-structure of $\Dmod(\Gr_G)$ is Zariski local (see \cite[Lemma 7.8.7]{ker-adj}), it suffices to prove the claim for $\Gr_G^\circ \simeq G(R)/G$. The latter is clear from Theorem \ref{mainthm:omega-inf-connective} for $\Sigma = \A^1$.
\end{proof}

\begin{rem}

The statement of the corollary was proved in \cite{ker-adj} modulo one mistake in the proof: precisely, contrarily to the claim of \cite[page 547]{ker-adj}, it is not true that $\Gr_G^\circ \simeq \A^\infty$. Indeed, as pointed out by D. Gaitsgory, this would contradict \cite[Theorem 5.4]{FGT}.

\end{rem}

\sssec{}

Given an affine scheme $Y$, one might ask what conditions on $Y$ ensure that $\omega_{Y[\Sigma]}$ is infinitely connective. The following result, whose proof uses only Riemann-Roch and an elementary t-structure estimate, gives a sufficient (but certainly not necessary) condition. 

\begin{mainthm}\label{mainthm:estimate with equations}
Let $Y \subseteq \A^N$ be a closed subscheme defined as the zero locus of $k$ polynomials of degrees $n_1, \ldots, n_k$. If $\sum_i n_i < N$, then $\omega_{Y[\Sigma]}$ is infinitely connective.
 \end{mainthm}

\sssec{}

Obviously, Theorem \ref{mainthm:estimate with equations} settles Theorem \ref{mainthm:omega-inf-connective} in the cases $G=GL_n$ and $G= SL_n$.
For more general groups, we take a completely different route, which uses the Ran space and a bit of representation theory. The proof is outlined in Section \ref{ssec:outline proof of thm D}.

\ssec{The structure of the paper}

\sssec{}
In Section \ref{sec:prelim}, we collect some basic notions that we will need throughout.

\sssec{}
In Section \ref{sec:GL for P1 and outline Thm C}, we discuss geometric Langlands for $X = \PP^1$, outline the proof of Theorem \ref{mainthm:B as unit of Sph-temp} and compute the Serre functor of the DG category $\Dmod(\Bun_G(\bbP^1))$.

\sssec{}

In Section \ref{sec:Serre}, we complete the proof of Theorem \ref{mainthm:B as unit of Sph-temp} by computing the Serre functor of three DG categories related to the nilpotent cone. This section and the previous one are the only ones that require some derived algebraic geometry.

\sssec{}

In Section \ref{sec:characterize temp via ThmC, prove ThmA}, we use Theorem \ref{mainthm:B as unit of Sph-temp} to characterize tempered D-modules on $\Bun_G$. We also show that Theorem \ref{mainthm: omega antitemp} follows from the vanishing of $\Hren(\Gr_G^\circ)$.

\sssec{}

In Section \ref{sec:prove ThmE}, we show that Theorem \ref{mainthm:Hren vanishing} is a simple corollary of Theorem \ref{mainthm:omega-inf-connective}. We then prove Theorem \ref{mainthm:estimate with equations}, which settles Theorem \ref{mainthm:omega-inf-connective} for $GL_n$ and $SL_n$.

\sssec{}

Finally, in Section \ref{sec: prove ThmD in general}, we prove Theorem \ref{mainthm:omega-inf-connective} for all reductive groups.

\ssec{The main techniques and ideas}

For the reader's convenience, let us highlight the eight most important notions and ideas employed in this paper. The first four are rather technical, while the second four form the geometric core of the paper.
We refer to Section \ref{sec:prelim} for any undefined notation and terminology, as well as for the appropriate references.

\sssec{Singular support}

The notion of singular support for ind-coherent sheaves is unavoidable, as it is at the heart of the notion of temperedness. In Section \ref{sec:Serre}, we will perform several singular support computations using \emph{Koszul duality} and the \emph{shearing} operation. These two devices allow to transform singular support for ind-coherent sheaves on a space into ordinary set-theoretic support for quasi-coherent sheaves on a different space.

\sssec{Ind-coherent sheaves and formal geometry}

In turn, quasi-coherent sheaves on a space $Y$ with support on a closed subset $Z \subseteq Y$ can be understood as quasi-coherent sheaves on $Y^\wedge_Z$, the formal completion of $Y$ along $Z$.
Contrarily to the case of ind-coherent sheaves, the functoriality of quasi-coherent sheaves is not well-adapted to working with formal completions. For this reason, a number of passages between quasi-coherent sheaves and ind-coherent sheaves will occur.

\sssec{Group actions on DG categories}

The categories appearing in the geometric Langlands program are often DG categories of D-modules on (double) quotients: for instance, consider $\Sph_G$, $\Dmod(\Bun_G(\PP^1))$, $\Dmod(\Gr_G)$. 
In particular, we often take quotients by infinite dimensional groups like $\GO$. The theory of loop group actions on DG categories is very convenient when dealing with such situations, and in particular when dealing with the Hecke action. It will be used in Section \ref{sec:characterize temp via ThmC, prove ThmA}.

\sssec{Serre functors}

The notion of Serre functor of a (proper) DG category is a very useful piece of abstract nonsense. It is clear that, given an equivalence $F: \C \to \D$ of proper DG categories, we have $\Serre_\D \simeq F \circ \Serre_\C \circ F^{-1}$. This intertwining property is essential for our proof of Theorem \ref{mainthm:B as unit of Sph-temp}, to be explained in Section \ref{ssec:outline of proof of theorem C}.

\sssec{Serre functor calculations}

While the definition of the Serre functor is abstract nonsense, the computation of the Serre functor in a given geometric situation is very much not abstract. 

In Section \ref{ssec: Serre functor on automorphic side}, we use Drinfeld's miraculous duality to compute the Serre functor on the DG category $\Dmod(\Bun_G(\PP^1))$.
In Section \ref{sec:Serre}, we compute the Serre functor of the DG category $\QCoh(\N/G)$ and of some related DG categories. The calculations hinge on the fact that $H^*(\N - \{0\}, \O)$ is nonzero only in two degrees.

We believe that a systematic study of the behaviour of Serre functors of DG categories of quasi-coherent sheaves on quotient stacks could be really fruitful.

\sssec{$\Sph_G$ and $\Dmod(\Bun_G(\PP^1))$}

In Section \ref{sec:GL for P1 and outline Thm C}, we crucially use the equivalence (due to V. Lafforgue) between the spherical DG category $\Sph_G$ and the automorphic Langlands DG category $\Dmod(\Bun_G(\PP^1))$. 
This equivalence is the reason for the appearance of $\GR$ in our Theorem \ref{mainthm:B as unit of Sph-temp}. 

As in the previous point, we believe that a deeper study of this relation will yield interesting results.

\sssec{Indschemes and t-structures}

As mentioned at the beginning, when $G$ is not of semisimple type $A$, we do not have a direct proof of the vanishing of $\Hren(G[\Sigma])$. Instead, we will show that $\omega_{G[\Sigma]} \in \Dmod(G[\Sigma])^{\leq -\infty}$ and then easily deduce that $\Hren(G[\Sigma]) \simeq 0$.

In general, it would be worthwhile to find several examples of indschemes (beyond the ones of Theorems \ref{mainthm:omega-inf-connective} and \ref{mainthm:estimate with equations}) whose dualizing sheaf is infinitely connective. 

\sssec{Ran space and the big cell}

The main idea to prove that $\omega_{G[\Sigma]} \in \Dmod(G[\Sigma])^{\leq -\infty}$ is to approximate $G[\Sigma]$ with $G^\circ[\Sigma]$, where $G^\circ \subseteq G$ is the big open cell. Indeed, the fact that $\omega_{G^\circ[\Sigma]} \in \Dmod(G^\circ[\Sigma])^{\leq -\infty}$ is immediate since $G^\circ[\Sigma]$ admits $\A^\infty$ as a direct factor (when $G$ is not abelian).

However, the map $\Gcirc[\Sigma] \to G[\Sigma]$ is not an open embedding, so we cannot invoke the Zariski-local nature of the t-structure on $\Dmod(G[\Sigma])$. 
Rather, we consider the open embedding $G[\Sigma]^{\Gcirc\ggen} \subseteq G[\Sigma]$ of maps $\Sigma \to G$ that generically land in $\Gcirc$. Then, roughly speaking, it remains to compare $G[\Sigma]^{\Gcirc\ggen}$ with $G[\Sigma']$, where $\Sigma' \subseteq \Sigma$ is a nonempty open subcurve. We do this in Section \ref{sec: prove ThmD in general} using the Ran space and some basic facts on the affine Grassmannian $\Gr_G$.

\ssec{Acknowledgements}

I would like to thank M. Pippi, R. Svaldi and B. To\"en for useful conversations; D. Gaitsgory and the anonymous referee for corrections and comments on earlier versions of this paper. Research supported by ERC-2016-ADG-741501.

\sec{Preliminaries and basic notations} \label{sec:prelim}

In this section, we collect the notations, the basic notions, and the basic results that we use. We advise the reader to skip this material and return here only when it is necessary.

\ssec{Representation theory and algebraic geometry}

We follow the conventions of \cite{AG1} and \cite{AG2}. Let us recall the most relevant ones.

\sssec{}

By the term \virg{space}, we mean a space of algebraic geometry: for instance, a (derived) scheme, an indscheme, a stack or a prestack. 
We fix a ground field $\kk$, algebraically closed and of characteristic zero, and set $\pt := \Spec(\kk)$. Every space $\Y$ appearing in this paper is defined over $\kk$ and the structure map $\Y \to \pt$ will be denoted by $p_\Y$.

\sssec{}

As mentioned before, $G$ always denotes a connected reductive group over $\kk$. We choose a Borel subgroup $B \subseteq G$ and a maximal torus $T \subseteq B$ once and for all. Let $N \subset B$ the maximal unipotent subgroup and $N^-$ the opposite unipotent subgroup.
Let $\g, \fb, \fn, \ft$ be the Lie algebras of $G, B, N, T$.

The Langlands dual group of $G$, defined using the duality of root data, is denoted by $\Gch$. It comes with a maximal torus $\check{T}$ and a Borel subgroup $\check{B}$.

\sssec{} \label{sssec:Chevalley space}

Let $\fc_G:= \g^*/ \! \! /G := \Spec((\Sym \g)^G)$. By Chevalley's restriction theorem and the theory of exponents, $(\Sym \g)^G \simeq \Sym (\fz_G \oplus \fa_G)$, where $\fz_G = \Lie(Z_G)$, and $\fa_G$ is a $\kk$-vector space generated by $r$-polynomials of degrees $d_1, \ldots, d_r$ (with $d_i$ equal one plus the $i$-th exponent of $G$).

\sssec{} \label{sssec:basics-Nilp}

We let $\N$ be the nilpotent cone of $G$ (accordingly, $\Nch$ the nilpotent cone of $\Gch$). 
By definition, $\N$ is the closed subscheme of $\g^*$ defined by
$$
\N :=
\g^*
\ustimes{\fc_G} \pt.
$$
Since the map $\g^* \to \fc_G$ is known to be flat, the fiber product defining $\N$ can be understood either in classical or in derived algebraic geometry.
We usually wish to regard $\N$ as a closed subscheme of $\g$: we do this by choosing a $G$-equivariant identification $\g \simeq \g^*$.

\sssec{}

We denote by $\Lambda$ the lattice of (co)weights of $G$. Precisely, $\Lambda$ means \virg{coweights} in Sections \ref{sec:GL for P1 and outline Thm C} and \ref{sec: prove ThmD in general}, while it means \virg{weights} in Section \ref{sec:Serre}.
Accordingly, the cone of dominant (co)weights is denoted by $\Lambda^{\dom}$.
This changing notation is the price to pay to avoid using $\lambdach$ in formulas. 
For two (co)weights $\mu$ and $\lambda$, the notation $\mu \leq \lambda$ means that $\lambda - \mu$ is a sum of positive (co)roots.

\sssec{}

Given an affine scheme $Y$ and a smooth curve $\Sigma$, we denote by $Y[\Sigma]$ the indscheme parametrizing maps from $\Sigma$ to $Y$. As a functor of points, $Y[\Sigma]$ sends a test affine scheme $S$ to the set $Y(S \times \Sigma)$. We will often use the shortcut $\Sigma_S := S \times \Sigma$.

The complement of $\Sigma$ inside its compactification is a finite set of \virg{points at infinity}, which we will denote by $D_\infty$. Let $h$ be the cardinality of $D_\infty$: this is the number of \virg{holes} that $\Sigma$ has. 

\sssec{}

In Sections \ref{sec:GL for P1 and outline Thm C} and \ref{sec:Serre}, we will need some formal and derived algebraic geometry.\footnote{The guiding principle is that derived algebraic geometry is required anytime we are dealing with quasi-coherent sheaves or ind-coherent sheaves.
On the other hand, Sections \ref{sec:characterize temp via ThmC, prove ThmA}-\ref{sec: prove ThmD in general} only deal with D-modules, so derived algebraic geometry is not needed there.}
The conventions for derived algebraic geometry follow \cite{Book}. In particular, fiber products of schemes in those sections are always derived. So, for example, the self-intersection of the origin in a finite dimensional vector space $V$ is the derived affine scheme
\begin{equation} \label{eqn:Omega V}
\Omega V := \pt \times_V \pt \simeq \Spec(\Sym(V^*[1])).
\end{equation}
This derived scheme has appeared before with $V = \gch$, and it will appear later with $V = \fc_G$.

\sssec{}

A derived scheme is said to be a \emph{(derived) globally complete intersection} if  it can be written as the derived zero locus $U \times_V \pt$ of a map $U \to V$ from a smooth scheme to a finite dimensional vector space. 
For instance, the scheme $\Omega V$ above and the nilpotent cone $\N$ are derived globally complete intersections.

\sssec{}

Our conventions on derived stacks follow those of \cite{AG1, finiteness}.
A derived scheme is \emph{quasi-smooth} if it is Zariski locally a global complete intersection. A derived stack is said to be quasi-smooth if it admits a quasi-smooth atlas. It follows that the adjoint quotients $\Omega \g/G$ and $\N/G$ are quasi-smooth. By \cite[Section 10]{AG1}, the derived stack $\LSG$ is quasi-smooth (it is even a global complete intersection in the natural stacky sense).

\sssec{}

We denote by $\Y_\dR$ the de Rham prestack of a prestack $\Y$, see \cite{Crystals}. For a map $\phi: \Y \to \Z$, we denote by $\Z^\wedge_{\Y}$ its formal completion, which is by definition the (derived) fiber product $\Z \times_{\Z_\dR} \Y_\dR$. For a quick review of the conventions regarding formal completions and the de Rham construction, the reader might consult \cite[Section 2]{centerH}.

\ssec{DG categories}

The conventions regarding higher category theory and differential graded (DG) categories follow \cite[Volume 1, Chapter 1]{Book}.
Let us recall the most important ones.

\sssec{}

Denote by $\DGCat$ the $\infty$-category whose objects are ($\kk$-linear) cocomplete DG categories and whose $1$-arrows are continuous (i.e., colimit preserving) functors.
By default, when we say that $\C$ is a DG category, we mean that $\C \in \DGCat$, that is, we assume that $\C$ is cocomplete. When $\C$ is not cocomplete, we say so explicitly. Similarly, a functor between DG categories is assume to be continuous unless otherwise stated.

\sssec{}

If $\C$ is a DG category (cocomplete or not) with two objects $c, c'$, we denote by $\RHom_\C(c,c')$ the DG vector space of morphisms $c \to c'$.

\sssec{}

For $\C \in \DGCat$, we let $\C^\cpt$ be its non-cocomplete full subcategory of compact objects. We assume familiarity with the notions of dualizability, compact generation and ind-completion. When a DG category $\C$ is dualizable, we denote by $\C^\vee$ its dual.

\sssec{}

By $\Vect$, we denote the DG category of complexes of $\kk$-vector spaces. We use cohomological indexing conventions throughout. Note that $\Vect^{\cpt}$ consists of those complexes with finite dimensional total cohomology.
We usually say that $V \in \Vect$ is \emph{finite dimensional} if it belongs to $\Vect^{\cpt}$.

\sssec{} \label{sssec:properness-definition}

Let $\C$ be a compactly generated DG category. Following \cite{GYD}, we say that $\C$ is \emph{proper} if 
$$
c, c' \in \C^\cpt
\implies
\RHom_{\C}(c,c') \in \Vect^\cpt.
$$
When $\C$ is (compactly generated and) proper, we consider its \emph{Serre functor} $\Serre_\C: \C \to \C$. This is the continuous functor uniquely characterized by:
\begin{equation} \label{eqn:Serre functor univ property}
\RHom_\C(c', \Serre_\C(c)) \simeq \RHom_\C(c,c')^*, 
\;  \;\mbox{for every $c, c' \in \C^\cpt$}.
\end{equation}
Here, $(-)^*$ denotes the dual of a complex of vector spaces. 
If $\C$ is clear from the context, we sometimes write $\Serre$ instead of the more precise $\Serre_\C$.

\begin{remark} \label{rem: remark on Serre functor}
Observe that, in the defining formula for $\Serre_\C$, the objects $c, c'$ are required to be compact. A simple colimit computation shows that in \eqref{eqn:Serre functor univ property} we might just as well require $c$ compact and $c'$ arbitrary. 
\end{remark}

\sssec{}

Our DG categories are sometimes equipped with t-structures. We use cohomological indexing, which means that, at the level of the underlying triangulated category, $\C^{\leq 0}$ is left orthogonal to $\C^{\geq 1}$.
A (continuous) functor $F:\C \to \D$ between DG categories with t-structures is said to be \emph{right t-exact} if it sends $\C^{\leq 0}$ to $\D^{\leq 0}$. Likewise, $F$ is said to be \emph{left t-exact} if it sends $\C^{\geq 0}$ to $\D^{\geq 0}$. Finally, $F$ is t-exact if it is both left and right t-exact.

\sssec{}

A continuous functor $\beta: \C \to \D$ of DG categories is said to be conservative if $\beta(d) \simeq 0$ implies $d \simeq 0$. If $\beta$ admits a left adjoint $\alpha$, then $\beta$ is conservative if and only if the essential image of $\alpha$ generates $\D$ under colimits. This is a consequence of the Barr-Beck-Lurie theorem.

\ssec{Ind-coherent sheaves and singular support} \label{ssec:indcoh-sing-supp}

\sssec{}

We denote by $\QCoh(\Y)$ and $\IndCoh(\Y)$ the DG categories of quasi-coherent and ind-coherent sheaves on a derived prestack $\Y$. While $\QCoh(\Y)$ is defined for arbitrary $\Y$, some conditions are required for $\IndCoh(\Y)$ to make sense.
We will only consider ind-coherent sheaves on algebraic stacks (such as $\N/G$, $\Omega \fg/G$, $\LSGch$) and of formal completions of maps of algebraic stacks.
Pushforwards, pullbacks and tensor products of sheaves are understood in the derived sense, unless otherwise stated.
There is a natural action of $\QCoh(\Y)$, equipped with its natural symmetric monoidal structure, on $\IndCoh(\Y)$.

\sssec{}

Formation of $\ICoh$ is contravariant: we denote by $f^!$ the structure pullback. 
We let $\omega_\Y := (p_\Y)^!(\kk)$ denote the ind-coherent dualizing sheaf.\footnote{Warning: the D-module pullback and the D-module dualizing sheaf are denoted in the same way. We hope that the context will make it clear which one we are referring to.} 
The action of $\QCoh(\Y)$ on $\omega_{\Y}$ yields a functor $\Upsilon_\Y: \QCoh(\Y) \to \ICoh(\Y)$. As $\Y$ varies, we obtain a natural transformation $\Upsilon_-: \QCoh(-) \to \ICoh(-)$ that intertwines quasi-coherent pullbacks with ind-coherent pullbacks.

\sssec{} \label{sssec:inf stuff}

When $f$ is nice (for instance: inf-schematic), there is a well-defined pushforward $f_{*}^\ICoh$ for ind-coherent shaves. When $f$ is furthermore inf-proper, $f_{*}^\ICoh$ is left adjoint to $f^!$.

The meaning of the terms \virg{inf-schematic} and \virg{inf-proper} appears in \cite[Vol. 2, Chapter 2, Section 4]{Book}. However, we will use these notions only in the  following situation. Let
\begin{equation} 
\nonumber
\begin{tikzpicture}[scale=1.5]
\node (00) at (0,0) {$  \X$};
\node (10) at (1,0) {$ \Z$};
\node (01) at (0,.7) {$\W$};
\node (11) at (1,.7) {$\Y$};
\path[->,font=\scriptsize,>=angle 90]
(00.east) edge node[above] {$ $} (10.west); 
\path[->,font=\scriptsize,>=angle 90]
(01.east) edge node[above] {$ $} (11.west); 
\path[->,font=\scriptsize,>=angle 90]
(01.south) edge node[right] {} (00.north);
\path[->,font=\scriptsize,>=angle 90]
(11.south) edge node[right] {$ $} (10.north);
\end{tikzpicture}
\end{equation}
be a commutative diagram of algebraic stacks. Then the obvious map $\xi: \X^\wedge_\W \to \Z^\wedge_\Y$ is inf-schematic (respectively: inf-proper, and inf-closed embedding) as soon as $\W \to \Y$ is schematic (respectively: proper, a closed embedding). If $\W_\dR \simeq \Y_\dR$, we say that $\xi$ is a \emph{nil-isomorphism}; in this case $\xi^!$ is conservative.

\sssec{}

When $\Y$ is an algebraic stack, we have a canonical functor $\Psi_\Y: \IndCoh(\Y) \to \QCoh(\Y)$, which is t-exact for the natural t-structures on both sides. When $\Y$ is quasi-smooth (and much more generally when $\Y$ is eventually coconnective), we have that:
\begin{itemize}
\item
 $\Psi_\Y$ admits a fully faithful left adjoint, denoted by $\Xi_\Y$;
\item
$\Upsilon_\Y$ is fully faithful.
\end{itemize}
When $\Y$ is smooth, $\Upsilon_\Y,   \Psi_\Y, \Xi_\Y$ are equivalences and 
\begin{equation} \label{eqn:anomaly}
\Psi_\Y(\omega_{\Y}) \simeq \det(\bbL_\Y)[\dim(\Y)].
\end{equation}

\sssec{}

The object $\Psi_\X(\omega_{\X})$ is a shifted line bundle more generally when $\X$ is quasi-smooth (and even more generally when $\X$ is \emph{Gorenstein}), see \cite[Section 7.3]{ICoh}. 
The following two computations will be useful.

\begin{lem} \label{lem:anomaly for nilcone}
Let $\N \subset \g$ be the nilpotent cone, as introduced above, and $\N/G$ the quotient stack given by the coajoint action. Then $\Psi_{\N/G} (\omega_{\N/G}) \simeq \O_{\N/G}[\dim(\N/G)]$.
\end{lem} 

\begin{proof}
Using \eqref{eqn:anomaly} applied to $\Y = BG$, we quickly obtain that $\Psi_{BG}(\omega_{BG}) \simeq \O_{BG}[\dim(BG)]$.
Next we claim that $\Psi_{\g/G}(\omega_{\g/G}) \simeq \O_{\g/G} $: this can be proven directly, or by using \cite[Vol. 2, Chapter 9, Proposition 7.3.4]{Book}, which is a relative version of \eqref{eqn:anomaly}.
 
Now, recalling the notation of Section \ref{sssec:basics-Nilp}, we see that $\N/G \simeq \g/G \times_{\fc_G} 0$. In particular, the inclusion $\N/G \hto \g/G$ is a regular embedding of relative codimension equal to $\dim(\fc_G)$. Then the assertion follows from Grothendieck's formula, see \cite[Vol. 2, Chapter 9, Section 7]{Book}. 
\end{proof}

\begin{lem} \label{lem:Omega V global sections}
Recall the derived scheme $\Omega V = \Spec \Sym(V^*[1])$ that appeared in \eqref{eqn:Omega V}. 
We have: 
$$
(p_{\Omega V})_*^{\ICoh} (\omega_{\Omega V}) \simeq \Sym(V[-1]).
$$
In particular (but we will not need this), $\Sym(V[-1])$ is the vector space underlying $\Psi_{\Omega V} (\omega_{\Omega V}) \in \Sym(V^*[1])\mod$.
\end{lem}

\begin{proof}
The derived scheme $\Omega V$ is proper (indeed, by definition, properness is checked at the level of classical truncations, and the classical truncation of $\Omega V$ is $\pt$).
Hence, by adjunction and by the fully faithfulness of $\Upsilon_{\Omega V}$, we obtain:
$$
\Bigt{ 
(p_{\Omega V})_*^{\ICoh} (\omega_{\Omega V}) 
}^*
\simeq
\RHom_{\ICoh(\Omega V)} (\omega_{\Omega V}, \omega_{\Omega V})
\simeq
\RHom_{\QCoh(\Omega V)} (\O_{\Omega V}, \O_{\Omega V})
\simeq
\Sym(V^*[1]).
$$
The assertion follows.
\end{proof}

\sssec{} \label{sssec:review of singular support}

Let $\Y$ be a quasi-smooth derived stack. 
We regard $\ICoh(\Y)$ as an enlargement of $\QCoh(\Y)$ by means of the functor $\Xi_\Y$.
Let us recall the main tenets of the theory of singular support of ind-coherent sheaves on $\Y$.
\begin{itemize}
\item
The (classical) stack of singularities of $\Y$ is defined as
$$
\Sing(\Y) := H^{-1}(T^* \Y) \simeq \{(y, \xi) \, | \, y \in \Y, \; \xi \in H^{-1}(T^*_y \Y)\}.
$$
It admits a $\Gm$-action defined by $\lambda \cdot (y, \xi) = (y, \lambda \xi)$.

\item
Ind-coherent sheaves on $\Y$ can be classified according to their singular support, where possible singular supports are closed conical subsets of $\Sing(\Y)$.
Given one such closed conical subset $M$, denote by $\ICoh_M(\Y)$ the full subcategory of $\ICoh(\Y)$ spanned by objects with singular support contained in $M$.

\item

Given $M \subseteq M' \subseteq \Sing(\Y)$, the canonical inclusion $\Xi_{M \to M'}: \ICoh_M(\Y) \hto \ICoh_{M'}(\Y)$ admits a continuous right adjoint that we denote by $\Psi_{M \to M'}$.

\item
The natural action of $\QCoh(\Y)$ on $\ICoh(\Y)$ preserves any $\ICoh_M(\Y)$ and it commutes with any $\Xi_{M \to M'}$ and $\Psi_{M \to M'}$.

\item
When $M =O_\Y$ is the zero section of $\Sing(\Y)$ corresponding to the condition that $\xi = 0$, then 
$\ICoh_M(\Y) \simeq \QCoh(\Y)$. Moreover, the adjunction $(\Xi_{O_\Y \to \Sing(\Y)}, \Psi_{O_\Y \to \Sing(\Y)})$ corresponds to the adjunction $(\Xi_\Y, \Psi_\Y)$ mentioned before.

\item
One computes that $\Sing(\Omega V) \simeq V^*$ and that $\Sing( \Omega \gch/\Gch ) \simeq \gch^*/\Gch$.
Upon fixing a $G$-equivariant isomorphism $\gch^* \simeq \gch$, it is clear that $\Sing(\LSGch)$ is isomorphic to the stack $\Arth_\Gch$ of geometric Arthur parameters that appeared in \eqref{eqn:Arthur params}.

\end{itemize}

\sssec{}

For $f: \X \to \Y$ a map of algebraic stacks, denote by $\Y^\wedge_\X$ its formal completion and let
$$
\X 
\xto{' \!f}
\Y^\wedge_\X
\xto {\wh f}
\Y
$$
be the canonical factorization of $f$. Since $' \!f$ is a nil-isomorphism, we have an adjunction
$$ 
\begin{tikzpicture}[scale=1.5]
\node (a) at (0,1) {$\ICoh(\X)$};
\node (b) at (3,1) {$\ICoh(\Y^\wedge_\X)$};
\path[right hook ->,font=\scriptsize,>=angle 90]
([yshift= 1.5pt]a.east) edge node[above] {$(' \!f)_*^\ICoh$} ([yshift= 1.5pt]b.west);
\path[->>,font=\scriptsize,>=angle 90]
([yshift= -1.5pt]b.west) edge node[below] {$(' \!f)^!$} ([yshift= -1.5pt]a.east);
\end{tikzpicture}
$$
with conservative right adjoint. The monad $(' \!f)^! \circ (' \!f)_*^\ICoh$ is isomorphic to $\U(\Tang_{\X/\Y})$, the universal envelope of the tangent Lie algebroid $\Tang_{\X/\Y} \to \Tang_\X$. 

\sssec{} \label{sssec:monoidal str on Sph spectral}

Let us describe the monoidal structure on the spectral side of derived Satake, see Theorem \ref{thm:derived satake}.
First, consider $\ICoh(\Omega \gch/\Gch)$. This DG category is monoidal under the convolution product $\star$ induced by the presentation $\Omega \gch/\Gch \simeq \pt/\Gch \times_{\gch/\Gch} \pt/\Gch$. Let $i: \pt/\Gch \hto \Omega \gch/\Gch$ be the diagonal closed embedding and 
$$
i_*^{\ICoh}: \Rep(\Gch) 
\longto 
\ICoh(\Omega \gch/\Gch)
$$
the associated functor, which is easily seen to be monoidal. Since $i$ is a nil-isomorphism, it follows that $i_*^{\ICoh}$ generates the target under colimits.

Now, let us turn to $\ICoh_\Nch(\Omega \gch/\Gch)$. Its monoidal structure is the unique one making $\Psi_{\Nch \to \gch^*}$ monoidal. To make sense of this, we need the following observation:
\begin{lem}
Let $\F, \G \in \ICoh(\Omega \gch/\Gch)$. If $\F$ is killed by $\Psi_{\Nch \to \gch^*}$, then the same is true for $\F \star \G$.
\end{lem}

\begin{proof}
Since $\ICoh(\Omega \gch/\Gch)$ is generated under colimits by the essential image of $i_*^{\ICoh}$, it suffices to check the statement for $\G = i_*^\ICoh(V)$, with $V$ arbitrary. By base-change, we have that 
$$
\F \star i_*^\ICoh(V) \simeq \pi^*(V) \stackrel{\act}\otimes \F.
$$
where $\pi : \Omega \gch/\Gch \to \pt/\Gch$ is the projection and $\stackrel{\act}\otimes$ denotes the action of $\QCoh$ on $\ICoh$.
The assertion follows from the fact that such action commutes with $\Psi$ functors in general.
\end{proof}

In particular, the unit object of $\ICoh_\Nch(\Omega \gch/\Gch)$ is given by $\Psi_{\Nch \to \gch^*}(i_*^\ICoh(\kk_0))$, where $\kk_0 \in \Rep(\Gch)$ is the trivial representation.

\ssec{D-modules}

With the exception of Section \ref{sec:characterize temp via ThmC, prove ThmA} and Section \ref{ssec:group actions on cats}, we will only encounter D-modules on indschemes and algebraic stacks of ind-finite type.
Let us recall our main conventions in this case. Our main references are \cite{Crystals} and \cite{finiteness}.

\sssec{} \label{sssec:ind-oblv-D-modules}

For $\Y$ as above, we let $\Dmod(\Y)$ be its DG category of D-modules. We denote by $f^!$ the D-module pullback and by $\omega_\Y:= (p_\Y)^!(\kk)$ the D-module diualizing sheaf.
We denote by $f_{*,\dR}$ (or simply by $f_*$ when the context its clear) the de Rham pushforward of D-modules. 
We let $\ind_{\Y}: \ICoh(\Y) \to \Dmod(\Y)$ be the induction functor and $\oblv_{\Y}: \Dmod(\Y) \to \ICoh(\Y)$ its (continuous) right adjoint. Note that $\oblv_{\Y}$ sends the D-module dualizing sheaf to the ind-coherent dualizing sheaf.

\sssec{} \label{defn t-structure}

The DG categories of D-modules on (ind)schemes are equipped with their right t-structure, see \cite[Section 4.3]{Crystals}. Let us recall the definition. 
For $\Y$ an (ind)scheme of (ind-)finite type, the t-structure on $\Dmod(\Y)$ is defined by declaring that $\F \in \Dmod(\Y)^{\geq 0}$ iff, for any closed subscheme $i:Y \hto \Y$, the resulting object $\oblv_{Y}(i^!(\F))$ belongs to $\IndCoh(Y)^{\geq 0}$.

As mentioned earlier in Corollary \ref{cor:GrG infinitely connective}, this t-structure is Zariski local: this means that the (co)connectivity of objects can be tested on Zariski open covers.
A proof is given in \cite[Lemma 7.8.7]{ker-adj}). We will use this result several times.

If $i: Y \to Z$ is a closed embedding of schemes (of finite type), then $i_{*,\dR}$ is t-exact. If $Y$ is a smooth scheme, then $\oblv_{Y}$ is t-exact.

\sssec{} \label{sssec:ICoh on formal as tensor product}

Let $Y \to Z$ be a map of schemes of finite type. As explained in \cite[Section 2.3.2]{strong-gluing}, we can write $\QCoh(Z^\wedge_Y)$ and $\ICoh(Z^\wedge_Y)$  as tensor products of DG categories, using the action of D-modules on quasi-coherent and ind-coherent sheaves:
$$
\QCoh(Z^\wedge_Y)
\simeq
\QCoh(Z)
\usotimes{\Dmod(Z)}
\Dmod(Y),
$$
$$
\ICoh(Z^\wedge_Y)
\simeq
\ICoh(Z)
\usotimes{\Dmod(Z)}
\Dmod(Y),
$$
If $Z$ is smooth, then the functors $\Xi_Z, \Psi_Z: \QCoh(Z) \to \ICoh(Z)$ are equivalences, hence the same holds true for the functors
$$
\Xi_Z \usotimes{\Dmod(Z)} \id_{\Dmod(Y)}
:
\QCoh(Z^\wedge_Y) 
\longto
 \ICoh(Z^\wedge_Y),
$$
$$
 \Upsilon_Z \usotimes{\Dmod(Z)} \id_{\Dmod(Y)} \simeq \Upsilon_{Z^\wedge_Y}
 :
 \QCoh(Z^\wedge_Y) 
\longto
 \ICoh(Z^\wedge_Y).
$$

\sssec{} \label{sssec:Dmods on BG}

When $L$ is a connected affine algebraic group, $\Dmod(\pt/L) \simeq H_*(L) \mod$, where $H_*(L)$ is the singular homology of $L$ equipped with the convolution product (see \cite[Section 7.2]{finiteness}). This equivalence is obtained by applying the Barr-Beck-Lurie theorem to the adjunction $(q_!, q^!)$, with $q: \pt \to \pt/L$ the quotient map.
Under this equivalence, $\omega_{\pt/L}$ goes over to the augmentation module of $H_*(L)$: it follows that $\omega_{\pt/L}$ is a (possibly not compact) generator of $\Dmod(\pt/L)$. 

\begin{lem} \label{lem: generators Dmod on BG}
For $M \to L$ a morphism of connected algebraic groups, denote by $f: \pt/M \to \pt/L$ the induced map. Then $f_!(\omega_{\pt/M})$ is a generator of $\Dmod(\pt/L)$.
\end{lem}

\begin{proof}
Since the quotient $q: \pt \to \pt/L$ factors through $f$, it follows that $f^!$ is conservative. Hence, the essential image of $f_!$ generates the target under colimits. To conclude, it remains to use the fact that $\omega_{\pt/M}$ generates $\Dmod(\pt/M)$.
\end{proof}

\ssec{Strong group actions on DG categories} \label{ssec:group actions on cats}

Following \cite{thesis}, let us recall some notions from the theory of strong loop group actions on DG categories.

\sssec{}

Let $H$ be a group (ind)scheme. We say that $H$ acts strongly\footnote{we usually omit the word \virg{strongly} and simply say that $H$ acts on $\C$} on a DG category $\C$ if the convolution monoidal DG-category $\Dmod(H)$ acts on $\C$. 
We often use this notion when $H$ is of infinite type (e.g. $H = \GO$ or $H = \GK$). For that, we need the theory of D-modules on indschemes of pro-finite type; this complication will not bother us in practice, except for Section \ref{ssec:Hecke action of B} where we will (marginally) need $\Dmod^!$.

\sssec{}

We denote by $\C^H$ (or ${}^H \C$, if we want to emphasize that $H$ acts on $\C$ from the left) the invariant DG category. If $H' \hto H$ is a subgroup, let $\oblv^{H' \to H}: \C^H \to \C^{H'}$ denote the conservative forgetful functor.
We denote by $\Av_!^{H' \to H}$ its (possibly only partially defined) left adjoint, and by $\Av_*^{H' \to H}$ its (possibly discontinuous) right adjoint.
When $H$ is a scheme (even of pro-finite type), $\Av_*^{H' \to H}$ is continuous.

\begin{rem} \label{rem:Dmods on BG}

If a pro-unipotent group $H$ acts trivially on $\C$, then $\oblv^{1 \to H}: \C^H \simeq \C$. This is not the case for groups that are not pro-unipotent. For instance, when $L$ is a connected affine algebraic group, we obtain that $\Vect^{L(\OO)} \simeq \Vect^L \simeq \Dmod(\pt/L) \simeq H_*(L) \mod$.
\end{rem}

\sssec{}

If $H$ acts on a space $\Y$ (say from the right), then $H$ acts on the DG category of $\Dmod(\Y)$ (again from the right) and $\Dmod(\Y)^H \simeq \Dmod(\Y/H)$. Under this isomorphism and the similar one for $H'$, the forgetful functor $\oblv^{H' \to H}$ goes over to the D-module pullback along $\Y/{H'} \to \Y/H$.

\ssec{Shearing}

The shearing (alias: shift of grading) operation was introduced in \cite[Appendix A]{AG1}. It will be needed in Sections \ref{ssec:shifts of grading}-\ref{ssec:main serre computation}.

\sssec{}

Let $\Gm \rrep^{weak} := (\QCoh(\Gm), \star) \mmod$ denote the $\infty$-category of DG categories with a weak action of $\Gm$ and $\Gm$-equivariant functors.
The shearing operation is an explicit automorphism of the $\infty$-category $\Gm \rrep^{weak}$, which we denote by $\C \squigto \C^\Rightarrow$.
We use the same notation for morphisms in $\Gm \rrep^{weak}$: namely, whenever $\phi: \C \to \D$ is a $\Gm$-equivariant functor, we denote by $\phi^\Rightarrow: \C^\Rightarrow \to \D^\Rightarrow$ the associated one.
We refer to \cite[Section 2.1]{strong-gluing}, which is a section dedicated to this topic.

\sssec{}

If $A$ is a graded DG algebra, then $A \mod$ acquires a natural $\Gm$-action and $(A \mod)^\Rightarrow \simeq A^\Rightarrow \mod$. 
Here $A^\Rightarrow$ is defined as follows. Let us view objects of $\Rep(\Gm)$ as graded complexes $(M_{i,k}, d)$ of vector spaces, where $i$ refers to the cohomological grading, $j$ (called weight) refers to the grading given by the $\Gm$-action, and $d$ is a horizontal differential. Then $A^\Rightarrow$ is the algebra whose underlying graded complex is $(A_{i+2k, k}, d)$.
Here is the main example to keep in mind: for $V$ a finite dimensional vector space, regard $A = \Sym(V)$ a graded DG algebra with $V$ in weight $1$; then $A^\Rightarrow \simeq \Sym(V[-2])$.

The following fact will be used in the sequel.

\begin{lem} \label{lem:conservative Rightarrows}
Let $\C, \D \in \Gm \rrep^{weak}$ and let $\phi: \C \to \D$ be a weakly $\Gm$-equivariant functor. Then $\phi$ is conservative if and only if so is $\phi^\Rightarrow$.
\end{lem}

\begin{proof}
The key is to show that $\phi$ is conservative if and only if so is the induced functor 
$\phi^\Gm: \C^\Gm \to \D^\Gm$ of (weak) invariant DG categories. One direction is obvious: forgetting $\Gm$-invariance $\C^\Gm \to \C$ is conservative, so the conservativity of $\phi$ implies that of $\phi^\Gm$.
The opposite implication follows from Gaitsgory's $1$-affineness theorem, \cite{shvcat}, which states that $\C$ can be recovered as
$\C \simeq \C^\Gm \usotimes{\Rep(\Gm)} \Vect$.
Now, to conclude the proof of the lemma, it suffices to observe that, by construction, there is a natural isomorphism\footnote{We emphasize that this isomorphism does not preserve that $\Rep(\Gm)$-module structure.}
 $(\phi^\Rightarrow)^\Gm \simeq \phi^\Gm$ of DG functors. 
\end{proof}

\sec{The tempered unit of the spherical category} \label{sec:GL for P1 and outline Thm C}

In this section, we start the proof of Theorem \ref{mainthm:B as unit of Sph-temp}.
Our strategy is explained in Section \ref{ssec:outline of proof of theorem C}: it relies on a certain explicit equivalence $\Sph_G \xto{\simeq} \Dmod(\Bun_G(\PP^1))$, reviewed in Section \ref{ssec:P1}, and on the computation of two Serre functors. The latter computations will be performed partly in Section \ref{ssec: Serre functor on automorphic side} and partly in Section \ref{sec:Serre}.

\ssec{Geometric Langlands for $\bbP^1$} \label{ssec:P1}

Recall that $\Bun_G(X)$ denotes the stack of $G$-bundles on the curve $X$. In this section, we focus solely on the case $X=\PP^1$. Let us fix a point $x \in X$ and choose a local coordinate $t$ at $x$. The data of $X$, $x$ and $t$ are regarded as fixed until the end of Section \ref{ssec:outline of proof of theorem C}.
Our goal is to construct an equivalence
$$
\gamma:
\Sph_G
\xto{\;\; \simeq \;\;}
\Dmod(\Bun_G(\PP^1)).
$$
The present section does not contain any new mathematics: the functor $\gamma$ is closely related to the same named functor introduced by V. Lafforgue in \cite{VLaff} and we limit ourselves to fill in some details. In passing, we will review, again following \cite{VLaff}, the role of $\gamma$ in the proof of the geometric Langlands equivalence for $X= \bbP^1$.

\sssec{}

We denote by $\GK = G\ppart$ and $\GO = G[[t]]$ the loop group and the arc group of $G$ at the point $x$.
For any $\lambda \in \Lambda$, we let $t^\lambda$ be the corresponding $\kk$-point of $T(K) \subseteq G(K)$.
The Birkhoff decomposition yields a stratification
$$
\GK
\simeq
\bigsqcup_{\lambda \in \Lambda^{\dom}} 
\GO t^\lambda \GR,
$$ 
with the stratum $\GO \GR \subseteq \GK$ being open.

\sssec{}

Let $\Gr_G \simeq \GO \backsl \GK $ be the affine Grassmannian at $x$, equipped with the obvious right action of $\GK$. By \cite{BL95}, we have:
\begin{equation} \label{eqn:BunG on P1}
\Bun_G(\PP^1)
 \simeq 
 \Gr_G/\GR 
 \simeq
\GO \backsl \GK / \GR.
\end{equation}
In particular, the open embedding $\GO \GR \subseteq \GK$ descends to an open embedding
$$
j: BG \simeq \GO \backsl \GO \GR / \GR 
\hto
\GO \backsl \GK / \GR  \simeq \Bun_G(\PP^1),
$$
which is the inclusion of the locus of trivializable $G$-bundles.
Consider the object $j_!(\omega_{BG}) \in \Dmod(\Bun_G(\PP^1))$.

\sssec{}

The above expression of $\Bun_G(\bbP^1)$ as a quotient equips $ \Dmod(\Bun_G(\bbP^1))$ with a left action\footnote{This is nothing but the action of $\Sph_G$ by Hecke functors at $x$, see Section \ref{sec:characterize temp via ThmC, prove ThmA} and in particular Section \ref{sssec:Hecke action}}
 of $\Sph_G$. 
We denote this action by $\star$ and set
$$
\gamma := - \star j_!(\omega_{BG}):
\Sph_G
\longto
\Dmod(\Bun_G(\bbP^1)).
$$
Our current goal is to prove that $\gamma$ is an equivalence and to provide a formula for its inverse. To get there, we need some preliminary results.

\sssec{}

We will use a few notions from the theory of strong group actions on DG categories, see \cite{thesis} and Section \ref{ssec:group actions on cats}.
Since $\GK$ acts on $\Gr_G$ from the right, it also acts on the DG category $\Dmod(\Gr_G)$, again from the right. 
Our two DG categories of interest are invariant categories for the action of $\GO$ and $\GR$ on $\Dmod(\Gr_G)$:
$$
\Dmod(\Gr_G)^\GO
\simeq
\Sph_G,
\hspace{.4cm}
\Dmod(\Gr_G)^{\GR} 
\simeq 
\Dmod(\Bun_G(\bbP^1)).
$$
By forgetting invariance along $G \hto \GO$ and $G \hto \GR$, we obtain the two functors:
\begin{equation} \label{eqn:oblv functors for Gr}
\Dmod(\Gr_G)^\GO
\xto{\oblv^{G \to \GO}}
\Dmod(\Gr_G)^G
\xleftarrow {\, \oblv^{G \to \GR}  \,}
\Dmod(\Gr_G)^{\GR} .
\end{equation}

\sssec{}

Consider the partially defined left adjoint to $\oblv^{G \to \GR} $, to be denoted by $\Av_!^{G \to \GR}$. Concretely, under the general equivalence $\Dmod(\Y)^H \simeq \Dmod(\Y/H)$, the functor $\Av_!^{G \to \GR}$ corresponds to the $!$-pushforward along the map $\Gr_G/G \tto \Gr_G /\GR$. Such pushforward is well-defined on holonomic D-modules\footnote{more precisely, by \virg{holonomic D-module on $\Y$} we mean a an object of $\Dmod(\Y)$ that is smooth-locally represented as a complex of D-modules with holonomic cohomology sheaves}, but potentially not on all D-modules. Nevertheless, we have:

\begin{lem} \label{lem:Grassmannian}
The composition
$$
\Sph_G =
\Dmod(\Gr_G)^\GO
\xto{\oblv^{G \to \GO}}
\Dmod(\Gr_G)^G
\xto{\, \Av_!^{G \to \GR}  \,}
\Dmod(\Gr_G)^{\GR} 
\simeq 
\Dmod(\Bun_G(\bbP^1))
$$
is well-defined.
 \end{lem}

 \begin{proof}
Clearly, the functor in question is well-defined on any holonomic object of $\Sph_G$. So, it suffices to prove that $\Sph_G = \Dmod(\Gr_G)^\GO$ is generated under colimits by holonomic D-modules.
Recall that the $\GO$-orbits on $\Gr_G$ are in bijection with $\Lambda^{\dom}$ as follows: $\Gr_G^\lambda$ is the $\GO$-orbit containing the point $[t^\lambda] := \GO t^\lambda \in \Gr_G(\kk)$. 
Using this, we can concoct an indscheme presentation $\Gr_G \simeq \colim_{n \geq 0} Z_n$, where each $Z_n$ is a finite dimensional scheme such that:
\begin{itemize}
\item $\GO$ acts on $Z_n$ with finitely many orbits and via a finite-dimensional quotient group $H_n$;
\item the kernel of the quotient $\GO \tto H_n$ is pro-unipotent.
\end{itemize}
Since $\Dmod(\Gr_G)^\GO \simeq \colim_{n \geq 0} \Dmod(Z_n)^\GO$, it remains to prove that each $\Dmod(Z_n)^\GO \simeq \Dmod(Z_n)^{H_n}$ is generated under colimits by holonomic D-modules. Since $H_n$ acts on $Z_n$ with finitely many orbits, the assertion easily follows by devissage along the $H_n$-orbit stratification.
\end{proof}

\begin{lem}
The composition $\Av_!^{G \to \GR} \circ \oblv^{G \to \GO}: \Sph_G \to \Dmod(\Bun_G(\bbP^1))$, well-defined by the above lemma, is canonically isomorphic to $\gamma$.
\end{lem}

\begin{proof}

By its very construction, $\gamma$ is $\Sph_G$-linear for the natural $\Sph_G$-actions on both sides.
Let us first show that $\Av_!^{G \to \GR} \circ \oblv^{G \to \GO}$ is $\Sph_G$-linear, too. 
Tautologically, the left $\Sph_G$-action on $\Dmod(\Gr_G)$ is compatible with the $\GK$-action on the right, hence the forgetful functors appearing in \eqref{eqn:oblv functors for Gr} are both $\Sph_G$-linear.

It follows that $\Av_!^{G \to \GR}$, being the left adjoint of a $\Sph_G$-linear functor, is colax linear. This means that, for any $\S \in \Sph_G$ and any $\F \in \Dmod(\Gr_G)^G$, there are natural arrows\footnote{together with the usual higher coherences present anytime one is dealing with higher categories; we will not dwell on them here}
$$
\Av_!^{G \to \GR}(\S \star \F)
\longto
\S \star \Av_!^{G \to \GR}(\F)
$$
that are possibly not isomorphisms. 
We will now use derived Satake, and its relation with the usual (underived) geometric Satake, to show that this colax module structure is actually a genuine $\Sph_G$-module structure.
Referring to the discussion of Section \ref{sssec:monoidal str on Sph spectral}, observe that the natural functor
$$
\Rep(\Gch)
\simeq
\IndCoh(\pt/\Gch)
\xto{i_*^\ICoh}
\ICoh(\Omega \gch/\Gch)
\xto{\Psi_{\Nch \to \gch}}
\ICoh_\Nch(\Omega \gch/\Gch)
$$
is monoidal and generates the target under colimits (indeed, the right functor is essentially surjective, while the left one is left adjoint to a conservative functor).
Combined with derived Satake, we obtain a monoidal functor $ \Rep(\Gch) \to \Sph_G$ that generates the target under colimits. Thus, to prove the strictness of the colax $\Sph$-linear structure of $\Av_!^{G \to \GR}$, it suffices to prove that the induced colax $\Rep(\Gch)$-linear structure is strict. The latter is clear: $\Rep(\Gch)$ is a rigid monoidal DG category, so we can apply \cite[Vol. 1, Chapter 1, Lemma 9.3.6]{Book}.

We have established that $\gamma$ and $\Av_!^{G \to \GR} \circ \oblv_{G \to \GO}$ are both $\Sph_G$-linear. To see they are isomorphic, it suffices to show that they agree when evaluated on the unit $\bbone_{\Sph_G}$.
A straightforward base-change yields
$$
\Av_!^{G \to \GR} \circ \oblv^{G \to \GO} (\bbone_{\Sph_G})
\simeq
j_!(\omega_{BG})
$$
as desired.
 \end{proof}

\begin{cor}
The functor $\gamma$ sends compact objects to compact objects.
\end{cor}

\begin{proof}
Since $\GO$ is a group scheme, the forgetful functor $\oblv^{G \to \GO}$ admits a continuous right adjoint: this is the $*$-averaging functor, denoted by $\Av_*^{G \to \GO}$. Then the right adjoint to $\gamma$ can be expressed as the continuous functor
$$
\Av_*^{G \to \GO} \circ \oblv^{G \to \GR}.
$$
By abstract nonsense, any functor with a continuous right adjoint preserves compact objects.
\end{proof}

\begin{thm}
The functor $\gamma: \Sph_G \to \Dmod(\Bun_G(\PP^1))$ is an equivalence.
\end{thm}

\begin{proof}
In the first three steps, we show that $\gamma$ is fully faithful; in the remaining ones, we show that it is essentially surjective.

\sssec*{Step 1}

Denote by $\phi: \Rep(\Gch) \to \Sph_G$ the monoidal functor introduced in the proof above. 
As we know, $\phi$ generates the target under colimits. Hence, to show that $\gamma$ is fully faithful, it suffices to prove that
$$
\RHom_{\Sph_G}(\phi(V),\F)
\xto{\;\; \simeq \;\;}
\RHom_{\Dmod(\Bun_G(\bbP^1))}
(
\gamma(\phi(V)),
\gamma(\F)
)
$$
for all $\F \in \Sph_G$ and all $V \in \Rep(G)^\cpt$.
Note that $\phi(V)$ is dualizable (indeed, compact objects of $\Rep(G)$ are dualizable and $\phi$ is monoidal); thus, up to replacing $\F$ with $\phi(V^*) \star \F$, we may restrict to the case where $V$ is the trivial $G$-representation.
In other words, it suffices to prove that
\begin{equation} \label{eqn:ff-Lafforuge}
\RHom_{\Sph_G}(\bbone_{\Sph_G},\F)
\xto{\;\; \simeq \;\;}
\RHom_{\Dmod(\Bun_G(\bbP^1))}
(
\gamma(\bbone_{\Sph_G}),
\gamma(\F)
)
\end{equation}
for arbitrary $\F \in \Sph_G$.

\sssec*{Step 2}

We would like to make sure that it is enough to verify the above claim in the special case where $\F \in \Sph_G$ is compact. This is not immediate, as $\bbone_{\Sph_G}$ is not compact.%
\footnote{To see that $\bbone_{\Sph_G}$ is indeed non-compact, notice that it is the image of the non-compact object $\omega_{\pt/G} \in \Dmod(\pt/G)$ under the (compact preserving) fully faithful functor
$$
\iota:
\Dmod(\pt/G) \simeq \Dmod(\pt)^\GO \hto \Dmod(\Gr_G)^\GO.
$$ }
To get around this issue, recall from Section \ref{sssec:Dmods on BG} that $\Dmod(\pt/G)$ is generated by a single compact object $q_!(\kk)$. Setting $\M := \iota(q_!(\kk)) \in \Sph_G$, we see that $\M$ is compact and that $\bbone_{\Sph_G}$ can be written as a colimit of copies of $\M$.
Hence, it suffices to prove that
\begin{equation} \label{eqn:ff-Lafforuge-for-cpt}
\RHom_{\Sph_G}(\M,\F)
\xto{\;\; \simeq \;\;}
\RHom_{\Dmod(\Bun_G(\bbP^1))}
(
\gamma(\M),
\gamma(\F)
)
\end{equation}
for arbitrary $\F \in \Sph_G$.
Since $\gamma$ preserves compact objects \emph{and} $\M$ is compact, it is enough to prove \eqref{eqn:ff-Lafforuge-for-cpt} under the assumption that $\F$ is compact. 

Next, to return to $\bbone_{\Sph_G}$, we use the fact that $\omega_{\pt/G}$ is a (non-compact) generator of $\Dmod(\pt/G)$, so that $\M$ can be expressed as a colimit of copies of $\bbone_{\Sph_G}$. Consequently, it is enough to prove that 
\begin{equation} \label{eqn:ff-Lafforuge-for-cpt-finally}
\RHom_{\Sph_G}(\bbone_{\Sph_G},\F)
\xto{\;\; \simeq \;\;}
\RHom_{\Dmod(\Bun_G(\bbP^1))}
(
\gamma(\bbone_{\Sph_G}),
\gamma(\F)
)
\end{equation}
for \emph{compact} $\F \in \Sph_G$. 

\sssec*{Step 3}

The latter isomorphism for $\F$ compact is proven in \cite[Pages 7-9]{VLaff}, by means of the contraction principle. Actually, V. Lafforgue proves that isomorphism for all $\F$ \emph{locally compact}, see \cite[Section 12.2.3]{AG1} for the definition. It is easy to see that compact objects are in particular locally compact.
This concludes the proof that $\gamma$ is fully faithful.

\sssec*{Step 4}

Let us now proceed to the essential the surjectivity of $\gamma$. In view of the already established fully faithfulness, it suffices to show that $\gamma$ generates the target under colimits.

\sssec*{Step 5}

Let us use again the fact that, when $L$ is a connected algebraic group, $\Dmod(\pt/L)$ is generated under colimits by a single object. This observation, together with the Birkhoff decomposition, implies that $\Dmod(\Bun_G(\bbP^1))$ is generated under colimits by the collection
$$
\M_\lambda := (j_\lambda)_!(
\G_{\lambda}
),
\hspace{.3cm}
\lambda \in \Lambda^{\dom},
$$
where:
\begin{itemize}
\item
$E_\lambda$ is the $G$-bundle on $\bbP^1$ corresponding to $t^\lambda$;
\item
$\Aut(E_\lambda) := \GO \cap \Ad_{t^\lambda} (\GR)$ is its automorphism group;
\item
 $j_\lambda$ is the locally closed embedding
$$
\pt/\Aut(E_\lambda) \simeq
\GO \backsl \GO t^\lambda \GR /\GR
\hto
\GO \backsl \GK /\GR;
$$
\item
$\G_{\lambda}$ is a generator of $\Dmod(\pt/{\Aut(E_\lambda)})$.
\end{itemize}
We will prove, by induction on the height $|\lambda|$, that for each $\lambda \in \Lambda^{\dom}$ the cocompletion of the essential image of $\gamma$ contains one $\M_\lambda$ as above. The base case is obvious: $\M_0 \simeq j_!(\omega_{BG}) = \gamma(\bbone_{\Sph_G})$.

\sssec*{Step 6}

Now recall that each $\lambda \in \Lambda^{\dom}$ yields a quasi-compact open substack $\Bun_G^{(\leq \lambda)} \subseteq \Bun_G(\bbP^1)$: this is the stack parametrizing $G$-bundles of Harder-Narasimhan coweight $\leq \lambda$, see e.g. \cite[Section 0.2.4]{DG-cptgen}.
The Birkhoff stratification restricts to the following stratification of $\Bun_G^{(\leq \lambda)}$:
$$
 \bigsqcup_{ \{\mu \in \Lambda^{\dom} | \mu \leq \lambda \}} 
 \GO \backsl \GO t^\mu \GR / \GR
 \simeq
 \bigsqcup_{ \{\mu \in \Lambda^{\dom} | \mu \leq \lambda \}} \pt/\Aut(E_\mu).
$$
Consider the point $[t^\lambda] := \GO t^\lambda \in \Gr_G(\kk)$ and its $\GO$-orbit $\Gr_G^\lambda$. We denote by $j_{\Gr_G^\lambda}: \Gr_G^\lambda \hto \Gr_G$ the locally closed embedding and by $(j_{\Gr_G^\lambda/\GO})_!$ the induced functor $\Dmod(\Gr_G^\lambda)^\GO \hto \Sph_G$.

\sssec*{Step 7}

We first prove that 
$$
\gamma \Bigt{
(j_{\Gr_G^\lambda/\GO})_!
( 
\omega_{\Gr_G^\lambda/\GO} 
)
 }
$$
is an object $!$-extended from the open substack $\Bun_G^{(\leq \lambda)}$.
Using the expression $\Av_!^{G \to \GR} \oblv^{G \to \GO}$ for $\gamma$, we obtain that
$$
\gamma \bigt{ 
(j_{\Gr_G^\lambda/\GO})_!
( 
\omega_{\Gr_G^\lambda/\GO} 
)
 }
\simeq
(m_\lambda)_! \bigt{
\omega_{\Gr_G^\lambda/G}
},
$$
where $m_\lambda : \Gr_G^\lambda/G \simeq \Gr_G^\lambda \times^G \GR/\GR \to \Gr_G/\GR$ is the natural action map. We need to show that $m_\lambda$ factors through $\Bun_G^{(\leq\lambda)}$. This is evident in view of the following relation between $\GO$-orbits and $\GR$-orbits on $\Gr_G$, see \cite[Proposition 2.3.3]{Xinwen}:
$$
\Gr_G^\lambda
\subseteq
\bigsqcup_{ \{\mu \in \Lambda^{\dom} | \mu \leq \lambda \} }
[t^\mu] \cdot \GR.
$$

\sssec*{Step 8}

Next, let $i_{\lambda}: \pt/\Aut(E_\lambda) \simeq \Bun_G^{(\lambda)} \hto \Bun_G^{( \leq \lambda)} $ be the closed embedding of the top stratum. By devissage along the stratification of $\Bun_G^{( \leq \lambda)} $ and the induction hypothesis, it suffices to prove that 
$$
\G_\lambda
:=
(i_{\lambda})^{*,\dR}
\Bigt{
\gamma \bigt{ 
(j_{\Gr_G^\lambda/\GO})_!
( 
\omega_{\Gr_G^\lambda/\GO} 
)
 }
}
\simeq
(i_{\lambda})^{*,\dR}
((m_\lambda)_!(\omega_{\Gr_G^\lambda/G}))
$$
is a generator of $\Dmod(\pt/ \Aut(E_\lambda))$. Let $P_\lambda \subseteq G$ be the stabilizer of the $G$-action on $[t^\lambda]$. Using the isomorphism (cf. \cite[Formula 2.6]{MV} and \cite[Proposition 2.3.3]{Xinwen})
$$
\Gr_G^\lambda \cap [t^\lambda] \cdot \GR \simeq [t^\lambda] \cdot G \simeq P_\lambda \backsl G,
$$
we see that the square 
\begin{equation} 
\nonumber
\begin{tikzpicture}[scale=1.5]
\node (00) at (0,0) {$  ([t^\lambda] \cdot \GO)/G \simeq \Gr_G^\lambda/G$};
\node (10) at (3.5,0) {$ \Gr_G/\GR \simeq \Bun_G(\bbP^1)$};
\node (01) at (0,1) {$\pt/P_\lambda \simeq ([t^\lambda] \cdot G) /G$};
\node (11) at (3.5,1) {$([t^\lambda] \cdot \GR)/\GR \simeq \Bun_G^{(\lambda)}$};
\path[->,font=\scriptsize,>=angle 90]
(00.east) edge node[above] {$m_\lambda$} (10.west); 
\path[->,font=\scriptsize,>=angle 90]
(01.east) edge node[above] {$f_\lambda$} (11.west); 
\path[right hook ->,font=\scriptsize,>=angle 90]
(01.south) edge node[right] {} (00.north);
\path[right hook ->,font=\scriptsize,>=angle 90]
(11.south) edge node[right] {$ $} (10.north);
\end{tikzpicture}
\end{equation}
is cartesian. We deduce that $\G_\lambda \simeq (f_\lambda)_!(
\omega_{\pt/P_\lambda})$ up to a cohomological shift. Then the assertion follows from Lemma \ref{lem: generators Dmod on BG}.
\end{proof}

\begin{cor} \label{cor:explicit formula for gamma inverse}
The functor $\gamma^{-1}: \Dmod(\Bun_G(\PP^1))  \to \Sph_G$ is given by the formula 
$$
\Av_*^{G \to \GO} \circ \oblv^{G \to G(R) }.
$$
\end{cor}

\begin{proof}
We have proven that $\gamma$ is an equivalence, written explicitly as the functor $\Av_!^{G \to G(R)} \circ \oblv^{G \to \GO }$. Hence, $\gamma^{-1}$ equals the right adjoint functor, which is $\Av_*^{G \to \GO} \circ \oblv^{G \to G(R) }$ by definition.
\end{proof}

\begin{rem}
The equivalence $\gamma$ helps prove that the two Langlands DG categories are equivalent in the case $X=\bbP^1$. Indeed, it is well-known that
$$
\LSGch(\PP^1) \simeq (\pt \times_{\gch} \pt) /\Gch=: \Omega \gch / \Gch.
$$ 
For a proof, see \cite[beginning of Section 1]{VLaff}.
Under this isomorphism, the global nilpotent cone $\NchGlob$ goes over tautologically to the nilpotent cone of $\Gch$, so that
$$
\ICoh_{\NchGlob}(\LSGch(\PP^1)) 
\simeq
\ICoh_{\Nch}(\Omega \gch / \Gch).
$$
Combining this equivalence with derived Satake and $\gamma$, we obtain the chain
\begin{equation} \label{eqn:ICoh on locsys on P1}
\ICoh_{\NchGlob}(\LSGch(\PP^1)) 
\simeq
\ICoh_{\Nch}(\Omega \gch / \Gch)
\simeq
\Sph_G
\simeq
\Dmod(\Bun_G(\bbP^1)).
\end{equation}
In the sequel, we will only need the second and the third equivalences (that is, $\Sat_G$ and $\gamma$), not the first one.
\end{rem}


\ssec{Overview of the proof of Theorem \ref{mainthm:B as unit of Sph-temp}} \label{ssec:outline of proof of theorem C}

Let us explain our strategy of the proof of Theorem \ref{mainthm:B as unit of Sph-temp}.

\sssec{Step 1: the setup}

The construction of the previous section and derived Satake show that the two functors
\begin{equation} \label{eqn: 3 equiv cats}
\ICoh_\Nch(\Omega \gch/\Gch)
\xto{\Sat_G}
\Sph_G
\xto{\gamma} 
\Dmod(\Bun_G(\PP^1))
\end{equation}
are equivalences.

\sssec{Step 2: the tempered unit}

Let $i: \pt/\Gch \hto \Omega \gch/\Gch$ denote the obvious closed embedding and $\kk_0 \in \QCoh(\pt/\Gch) \simeq \Rep(\Gch)$ the trivial $G$-representation. In view of Section \ref{sssec:monoidal str on Sph spectral}, the unit $\bbone_{\Sph_G}$ corresponds to $\Psi_{\Nch \to \gch}(i_*^\ICoh(\kk_0))$ under $\Sat_G$.
It follows that $\B$ corresponds to $\Xi_{0 \to \Nch}(i_*(\kk_0))$ under $\Sat_G$. Indeed,
\begin{eqnarray}
\nonumber
\Sat_G^{-1} (\B)
& := &
\Sat_G^{-1} \circ
\Xi_{0 \to \Nch} \circ \Psi_{0 \to \Nch}(\bbone_{\Sph_G}) 
\\
\nonumber
& \simeq &
\Xi_{0 \to \Nch} \circ \Psi_{0 \to \Nch}(
\Sat_G^{-1} (  \bbone_{\Sph_G} )) 
\\
\nonumber
& \simeq &
\Xi_{0 \to \Nch} \circ \Psi_{0 \to \Nch} \circ
\Psi_{\Nch \to \gch}(i_*^\ICoh(\kk_0))
\\
\nonumber
& \simeq &
\Xi_{0 \to \Nch} \circ \Psi_{0 \to \gch} (i_*^\ICoh(\kk_0))
\\
\nonumber
& \simeq &
\Xi_{0 \to \Nch}  (i_*(\kk_0)),
\end{eqnarray}
where the last step uses the compatibility between pushforwards and $\Psi$, see \cite[Section 3.2.12]{finiteness}.

\sssec{Step 3: properness} \label{sssec:properness of DG cats}

Recall the notion of properness for DG categories, see Section \ref{sssec:properness-definition}. We claim that the DG categories appearing in \eqref{eqn: 3 equiv cats} are proper. 
We will give two different proofs of this fact. The quickest proof, explained in Section \ref{ssec: Serre functor on automorphic side}, shows that $\Dmod(\Bun_G(\PP^1))$ is proper.
Another proof, discussed in Section \ref{ssec:main serre computation}, shows the properness of $\ICoh_\Nch(\Omega \gch/\Gch)$.

\sssec{Step 4: the Serre functor on the spectral side}

Thanks to properness, it makes sense to consider the Serre functors of the three DG categories of \eqref{eqn: 3 equiv cats}.
On the $\Gch$-side, we will prove that 
$$
\Serre_{\ICoh_\Nch(\Omega \gch/\Gch)}
(\Psi_{\Nch \to \gch}(i_*^\ICoh(\kk_0)))
\simeq
\Xi_{0 \to \Nch}( i_*(\kk_0) )[- \dim \Gch].
$$
Serre functors are obviously intertwined by equivalences of DG categories: in our case, derived Satake implies that
\begin{equation} \label{eqn:Serre of unit of Sph}
\Serre_{\Sph_G}
(\bbone_{\Sph_G})
\simeq
\bbone_{\Sph_G}^\temp
[- \dim \Gch].
\end{equation}

\sssec{Step 5: the Serre functor on the automorphic side}

On the automorphic side, we will show that
\begin{equation} \label{eqn:Serre of j! omega BG}
\Serre_{\Dmod(\Bun_G(\PP^1))}( j_!(\omega_{BG}))
\simeq
 j_*(\omega_{BG})[- \dim G],
\end{equation}
where $j: BG \hto \Bun_G(\PP^1)$ is the open embedding induced by the trivial $G$-bundle. More generally, we wil check that a certain explicit functor $\sT_{\Bun_G(\PP^1)}$ (see below for the definition) equals the Serre functor on $\Dmod(\Bun_G(\PP^1))$.

\sssec{Step 6: the conclusion}

By construction, $\gamma(\bbone_{\Sph_G}) \simeq j_!(\omega_{BG})$.
Now, comparing \eqref{eqn:Serre of unit of Sph} with \eqref{eqn:Serre of j! omega BG} and using the fact that $\dim(G) = \dim(\Gch)$, we obtain that 
$$
\gamma(\B) \simeq j_*(\omega_{BG}).
$$
Equivalently,
$$
\B \simeq \gamma^{-1}(j_*(\omega_{BG})).
$$
Using the explicit formula for $\gamma^{-1}$ from Corollary \ref{cor:explicit formula for gamma inverse}, a straightforward diagram chase along
$$
BG = \GO \backsl \GO \GR / \GR
\hto
\GO \backsl \GK / \GR
\leftto
\GO \backsl \GK / G
\to 
\GO \backsl \GK / \GO
$$
yields the claimed isomorphism $\B  \simeq (f^!)^R(\omega_{G \backsl G(R) /G})$. 

\sssec{}

Only Steps 3, 4 and 5 need further details. Steps 3 and 5 are treated immediately below, while Step 4 is the content of Section \ref{sec:Serre}.

\ssec{The Serre functor on the automorphic side} \label{ssec: Serre functor on automorphic side}

In this short section, we prove the Serre functor formula that appeared in  \eqref{eqn:Serre of j! omega BG}. For this, we need to review a few facts on the \emph{pseudo-identity} functor $\Psid_{\Y,!}$, also called \emph{Drinfeld's miraculous duality}.
We return to the general case of $X$ of arbitrary genus and we let, as is standard, $\Bun_G := \Bun_G(X)$.  
Note, however, that the Lemma \ref{lem:T on P1 is Serre} and Corollary \ref{cor: Serre on basic object of P1} are specific to $X=\bbP^1$.

\sssec{}

Let $\Y$ be an algebraic stack $\Y$ such that $\Dmod(\Y)$ is dualizable. In view of \cite[Theorem 0.2.2]{finiteness}, this condition is often satisfied in practice. 
Note also that $\Dmod(\Bun_G)$ is compactly generated, and therefore dualizable, by \cite[Theorem 0.1.2]{DG-cptgen}. Moreover, $\Bun_G$ is exhausted by quasi-compact open substacks with compactly generated DG categories of D-modules.

The dualizability of $\Dmod(\Y)$ implies that functors from the dual DG category $\Dmod(\Y)^\vee$ to $\Dmod(\Y)$ correspond precisely to objects of $\Dmod(\Y \times \Y)$.
We will consider two particularly interesting functors, called \virg{pseudo-identities}:
$$
\Psid_{\Y,*}: \Dmod(\Y)^\vee \longto \Dmod(\Y),
\hspace{.4cm}
\Psid_{\Y,!}: \Dmod(\Y)^\vee \longto \Dmod(\Y).
$$
The first one is defined by the kernel $(\Delta_\Y)_*(\omega_\Y) \in \Dmod(\Y \times \Y)$, the second one by the kernel $(\Delta_\Y)_!(k_\Y) \in \Dmod(\Y \times \Y)$. 

\sssec{}

These functors were introduced and discussed in detail in \cite[Section 4]{DG-cptgen} and \cite[Sections 6-7]{ker-adj}. Here are some relevant facts that we need:
\begin{itemize}
\item
If $\Y$ is quasi-compact, then $\Psid_{\Y,*}$ is an equivalence;
\item
if $\Psid_{\Y,!}$  is an equivalence, $\Y$ is said to be \emph{miraculous};
\item
$\Bun_G$ is miraculous, see \cite{strange} for the proof;
\item
$\Bun_G$ can be exhausted by a sequence of miraculous quasi-compact opens, see \cite[Lemma 4.5.7]{DG-cptgen}.
\end{itemize}

\sssec{}

When $\Y$ is miraculous, we can consider the functor 
$$
\sT_{\Y} :=\Psid_{\Y,*} \circ \Psid_{\Y,!}^{-1}.
$$
If $\Y$ is quasi-compact and miraculous, then $\sT_{\Y}$ is an equivalence. On the other hand, $\TBunG$ is not at all an equivalence: the argument of \cite[Theorem 7.7.2]{ker-adj}, coupled with the correction given by Corollary \ref{cor:GrG infinitely connective} shows that 
$$
\TBunG(\omega_{\Bun_G}) \simeq 0.
$$

\sssec{} \label{sssec: j! compact generators}

By \cite[Section 4]{DG-cptgen}, every compact object of $\Dmod(\Bun_G)$ can be written as $(j_U)_!(\F_U)$ for some quasi-compact miraculous open substack $j_U : U \hto \Bun_G$ and some $\F_U \in \Dmod(U)^\cpt$.
The next observation explains how $\TBunG$ interacts with $T_U$.

\begin{lem}
In the above notation, we have a canonical isomorphism
\begin{equation} \label{eqn:TBunG on !-exts}
(j_U)_* \circ \sT_U
\simeq
\TBunG \circ (j_U)_!.
\end{equation}
\end{lem}

\begin{proof}
By \cite[Lemma 4.4.12]{DG-cptgen}, we know that
$$
(j^!)^\vee \circ (\Psid_{U,!})^{-1}
\simeq
(\Psid_{\Bun_G,!})^{-1} \circ j_!.
$$
Hence, it remains to show that 
$$
\Psid_{\Bun_G, *} \circ (j^!)^\vee 
\simeq
j_* \circ \Psid_{U, *},
$$
as functors $\Dmod(U)^\vee \to \Dmod(\Bun_G)$. Each of these two functors is given by a kernel in $\Dmod(U \times \Bun_G)$. Tautologically, these two kernels are repsectively
$$
(j \times \id_{\Bun_G})^! (\Delta_{\Bun_G})_*(\omega_{\Bun_G}),
\hspace{.4cm}
(\id_U \times j)_* (\Delta_U)_*(\omega_U).
$$
To conclude, observe that these two objects match by base-change.
\end{proof}

\sssec{}

Now assume again that $X=\PP^1$. In this case, $\sT_{\Bun_G(\bbP^1)}$ is quite special. Indeed:

\begin{lem}  \label{lem:T on P1 is Serre}
The DG category $\Dmod(\Bun_G(\PP^1))$ is proper and $\sT_{\Bun_G(\bbP^1)}$ is its Serre functor.
\end{lem}

\begin{proof}
We already know that $\Dmod(\Bun_G(\bbP^1))$ is compactly generated, and Section \ref{sssec: j! compact generators} describes how compact objects look like. In view of that description, the properness of $\Dmod(\Bun_G(\bbP^1))$ is an immediate consequence of the following claim: for any quasi-compact open $U \subset \Dmod(\Bun_G(\PP^1))$, the DG category $\Dmod(U)$ is proper.
The Birkhoff decomposition guarantees that $U$ has finitely many isomorphism classes of $\kk$-points. Then the claim follows from the first part of  \cite[Theorem 2.1.5]{GYD}.

Next, let us show that $\Serre_{\Dmod(\Bun_G(\PP^1))} \simeq \sT_{\Bun_G(\PP^1)}$. We need to provide, for any $\F \in \Dmod(\Bun_G(\bbP^1))$ and any $\G \in \Dmod(\Bun_G(\bbP^1))^\cpt$, a natural isomorphism
$$
\RHom_{\Dmod(\Bun_G(\bbP^1))} (\F, \TBunG (\G) )
\simeq
\RHom_{\Dmod(\Bun_G(\bbP^1))} (\G,\F)^*.
$$
Thanks to Section \ref{sssec: j! compact generators} again, it suffices to do so for $\G$ of the form $(j_U)_!(\F_U)$, where $j_U: U \hto \Bun_G(\PP^1)$ is a miraculous quasi-compact open substack and $\F_U \in \Dmod(U)^\cpt$.
In view of the second part of \cite[Theorem 2.1.5]{GYD}, 
we know that 
$\Serre_{\Dmod(U)} 
\simeq 
\sT_U$.
Using this, we compute:
\begin{eqnarray}
\nonumber
\RHom_{\Dmod(\Bun_G(\bbP^1))}
\Bigt{
\F, (j_U)_*(\F_U)  
}
& \simeq &
\RHom_{\Dmod(U)}
\Bigt{
 (j_U)^! (\F), \F_U
 }
\\
\nonumber
& \simeq &
\RHom_{\Dmod(U)}
\Bigt{
 (j_U)^! (\F), \Serre_{\Dmod(U)} \circ (\sT_U)^{-1} (\F_U)
}
\\
\nonumber
& \simeq &
\RHom_{\Dmod(U)}
\Bigt{
 (\sT_U)^{-1}  (\F_U),  (j_U)^! (\F) 
}
^*
\\
\nonumber
& \simeq &
\RHom_{\Dmod(\Bun_G(\bbP^1))}
\Bigt{
 (j_U)_!  (\sT_U)^{-1}  (\F_U), \F 
}
^*.
\end{eqnarray}
From this, we obtain that
$$
\Serre_{\Dmod(\Bun_G(\bbP^1))}
\bigt{
(j_U)_! \circ  (\sT_U)^{-1}  (\F_U)
}
\simeq
(j_U)_*(\F_U),
$$
or, equivalently, 
$$
\Serre_{\Dmod(\Bun_G(\bbP^1))} \circ (j_U)_!
\simeq
(j_U)_* \circ \sT_U.
$$
It remains to invoke \eqref{eqn:TBunG on !-exts}.
\end{proof}

\begin{cor} \label{cor: Serre on basic object of P1}
Let $j: BG \hto \Bun_G(\PP^1)$ be the open embedding induced by the trivial $G$-bundle. The Serre functor of $\Dmod(\Bun_G(\bbP^1))$ sends $j_!(\omega_{BG})$ to $j_*(\omega_{BG})[-\dim(G)]$. 
\end{cor}

\begin{proof}
It suffices to notice that $\sT_{BG}= \id_{\Dmod(BG)}[\dim(G)]$, a fact that readily follows from the definitions and from the equivalence $\Dmod(BG) \simeq H_*(G) \mod$ of Section \ref{sssec:Dmods on BG}.
\end{proof}

\sec{The Serre functor of the spectral spherical category} \label{sec:Serre}

In this section, we compute the Serre functor of $\ICoh_{\Nch}(\Omega \gch/\Gch)$ and complete the proof of the fourth step of Section \ref{ssec:outline of proof of theorem C}.
Since Langlands duality does not appear here, we will formulate our results for $G$, keeping in mind that they have been applied to $\Gch$ in Section \ref{ssec:outline of proof of theorem C}.

Our main result is that the Serre functor of ${\ICoh_\N(\Omega \g/G)}$  equals the \emph{temperization functor} up to a cohomological shift. The proof hinges on a preliminary result, which might be of independent interest: the computation of the Serre functor of the DG category $\QCoh(\N/G)$.

\ssec{The nilpotent cone}

Let $\N$ be the nilpotent cone associated to the group $G$.
We show that the DG category $\QCoh(\N/G)$ is proper and compute its Serre functor explicitly.

\begin{lem}
The DG category $\QCoh(\N/G)$ is proper.
\end{lem}

\begin{proof}
For $\lambda \in \Lambda^{\on{dom}}$ a dominant weight, let $V_\lambda$ be the irreducible $G$-representation of highest weight $\lambda$.
Denote by $\pi: \N/G \to \pt/G$ the obvious projection. Since $\pi$ is affine, $\pi_*$ is conservative and consequently the essential image of $\pi^*: \Rep(G) \simeq \QCoh(\pt/G) \to \QCoh(\N/G)$ generates $\QCoh(\N/G)$ under colimits. More precisely, the perfect objects
$$
A_\lambda := \pi^* V_\lambda
\in \QCoh(\N/G),
\hspace{.4cm}
\lambda \in \Lambda^{\on{dom}},
$$
form a collection of compact generators of $\QCoh(\N/G)$.

Thus, it suffices to show that $\RHom_{\QCoh(\N/G)}(A_\lambda, A_\mu)$ is finite dimensional\footnote{the locution \virg{being finite dimensional} applied to a complex of vector spaces is a shortcut for \virg{having finite dimensional total cohomology}} for all $\lambda, \mu \in \Lambda^{\dom}$.
By adjunction and the projection formula,  we have:
$$
\RHom_{\QCoh(\N/G)}(A_\lambda, A_\mu)
\simeq
\RHom_{\Rep(G)}(V_\lambda, R \otimes V_\mu)
\simeq
\RHom_{\Rep(G)}(V_\lambda \otimes V_\mu^*, R),
$$
where $R:= H^0(\N,\O_\N)$, viewed as a $G$-representation in the natural way. 
Since $V_\lambda \otimes V_\mu^*$ is a direct sum of finitely many irreducible $G$-representations, it suffices to show that $\RHom_{\Rep(G)}(V_\nu, R)$ is finite dimensional for any $\nu \in \Lambda^{\dom}$.

By a theorem of B. Kostant (the original source is \cite{Kostant}, see also \cite[Theorem 6.7.4]{CG}), there is an isomorphism $R \simeq H^0(G/T,\O_{G/T})$ of $G$-representations. This implies that
\begin{equation} \label{eqn:Kostant identity}
R \simeq
\Gamma(G,\O_G)^T
\simeq
\bigoplus_{\nu \in \Lambda^{\on{dom}}}
(V_{-w_0(\nu)})^T
 \otimes V_\nu,
\end{equation}
where $w_0$ is the longest element of the Weyl group, so that $V_{-w_0(\nu)} \simeq V_\nu^*$. 
It follows that $\RHom_{\Rep(G)}(V_\nu, R) \simeq (V_{-w_0(\nu)})^T$, which is indeed finite dimensional.
\end{proof}

\begin{prop}\label{prop:Serre on QCoh N/G}
The Serre functor on $\QCoh(\N/G)$ is the functor of tensoring with the object
$$
S_0 := \ker(\O_{\N/G} \to j_* \O_{\N^\times/G})
[2 \dim \fn],
$$
where $\N^\times := \N - 0$ is the punctured nilpotent cone and 
$j: \N^\times/G \hto \N/G$ the obvious open embedding.
\end{prop}

\begin{proof}
Let us retain the notation of the previous lemma. 

\sssec*{Step 1}

By the defining property of the Serre functor, see Section \ref{sssec:properness-definition}, we just need to exhibit natural isomorphisms
$$
\RHom_{\QCoh(\N/G)}
\bigt{A_\nu, A_\lambda \otimes S_0
}
\simeq
\RHom_{\QCoh(\N/G)}
\bigt{
A_\lambda, A_\nu}^*
$$
for all $\lambda, \nu \in \Lambda^{\dom}$.
Since $A_\lambda$ is dualizable with dual $A_{-w_0(\lambda)}$, it is easy to see that  we may assume $\lambda = 0$.
Reasoning as in the previous lemma (using adjunction and the projection formula), it suffices to establish a functorial isomorphism
\begin{equation} \label{eqn:auxiliary serre for N}
\RHom_{\Rep(G)}
\bigt{V_\nu, 
\ker(R \to R^\times)
}
[2\dim \fn]
 \simeq
\RHom_{\Rep(G)}
\bigt{
V_{-w_0(\nu)}, R}^*,
\end{equation}
where we have set $R := H^0(\N,\O)$ and $R^\times := H^*(\N^\times, \O)$, both regarded as $G$-representations in the natural way.

\sssec*{Step 2}

Consider now the Springer resolution $\mu: T^*(G/B) \to \N$. In view of \cite[Theorem A]{Hesselink}, the canonical arrow
$$
\O_\N \longto
\mu_*(\O_{T^*(G/B)})
$$
is an isomorphism.\footnote{Under our conventions, $\mu_*$ denotes the derived pushforward: thus, the isomorphism $\O_\N \simeq
\mu_*(\O_{T^*(G/B)})$ means that $R^0 \mu_*(\O_{T^*(G/B)}) \simeq \O_\N$ and that $R^i \mu_*(\O_{T^*(G/B)}) \simeq 0 $ for all $ i \geq 1$.}
In other words, the nilpotent cone has rational singularities.
Pulling back $\mu$ along $j$, we obtain a map $\mu^\times: T^*(G/B)^\times \to \N^\times$, where $T^*(G/B)^\times$ is the complement of the zero section.
Thus, by base-change, the canonical map 
$$
\O_{\N^\times} 
\longto
(\mu^\times)_*(\O_{T^*(G/B)^\times})
$$
is an isomorphism, too.
Upon taking global sections, we get:
$$
R \simeq H^0(T^*(G/B), \O),
\hspace{.4cm}
R^\times \simeq H^*(T^*(G/B)^\times,\O).
$$

\sssec*{Step 3}

Since $\N$ is normal of dimension $\dim(\N) = 2\dim(\fn) \geq 2$, we have $H^0(R^\times) \simeq R$ and consequently $\ker(R \to R^\times) \simeq \tau^{\geq 1}(R^\times)[-1]$,  where $\tau^{\geq m}$ is the usual truncation functor associated to the standard t-structure on complexes of vector spaces.
Then \eqref{eqn:auxiliary serre for N} simplifies as
\begin{equation} \label{eqn:auxiliary serre for N simplified}
\RHom_{\Rep(G)}
\bigt{V_\nu, 
\tau^{\geq 1}(R^\times)
}
[2\dim \fn-1]
 \simeq
\RHom_{\Rep(G)}
\bigt{
V_{-w_0(\nu)}, R}^*.
\end{equation}
To prove our result, it suffices to show that $R^\times$ has higher cohomology only in degree $(2 \dim \fn -1)$, and that such higher cohomology decomposes (as a $G$-representation) as
\begin{equation} \label{eqn:higher cohomology Rtimes}
H^{2 \dim \fn -1}
(R^\times)
\simeq 
\bigoplus_{\nu \in \Lambda^{\dom}}
\RHom_{\Rep(G)}
\bigt{
V_{-w_0(\nu)}, R}^*
 \otimes V_\nu.
\end{equation}

\sssec*{Step 4}

Now observe that $T^*(G/B)^\times \simeq G \times^B (\fn-0)$.
In view of the $B$-equivariant isomorphism
$$
H^*(\fn-0, \O)
\simeq
\Sym \fn^* 
\oplus
(\Sym \fn \otimes \Lambda^{\dim \fn} \fn )[1- \dim \fn],
$$
we obtain that
\begin{eqnarray}
\label{eqnarray:first line}
& &
R
\simeq
H^{0}(R^\times)
\simeq
\bigt{
\Gamma(G,\O) \otimes H^{0}(\fn, \O)
}^B
\simeq
\bigoplus_{\lambda}
\RHom_{\Rep (B)}(V_\lambda, \Sym \fn^*)
\otimes 
V_\lambda,
\\
\label{eqnarray:second line}
& &
H^{>0}(R^\times)
\simeq
\bigt{
\Gamma(G,\O) \otimes H^{>0}(\fn-0, \O)
}^B
\simeq
\bigoplus_{\lambda}
\RHom_{\Rep (B)}(V_\lambda, \Sym \fn \otimes  \Lambda^{\dim \fn} \fn)
\otimes 
V_\lambda
[1-\dim \fn].
\end{eqnarray}
In the above formulas, we have used $H^{>0}(-)$ as a shortcut for $\tau^{\geq 1}(H^*(-))$.

\sssec*{Step 5}

Comparing \eqref{eqnarray:first line} with the formula \eqref{eqn:Kostant identity} obtained from Kostant's theorem, we deduce that
\begin{equation}\label{eqn:Kostant}
\RHom_{\Rep (G)}(V_\lambda, R)
\simeq
\RHom_{\Rep (B)}(V_\lambda, \Sym \fn^*)
\simeq
(V_{-w_0(\lambda)})^T.
\end{equation}
Thus, \eqref{eqn:higher cohomology Rtimes} simplifies as
\begin{equation}
H^{2 \dim \fn -1}
(R^\times)
\simeq 
\bigoplus_{\nu \in \Lambda^{\dom}}
(V_{-w_0(\nu)})^T
 \otimes V_\nu.
\end{equation}
To conclude our proof, we need to prove the above formula, as well as the fact that the only higher cohomology of $R^\times$ occurs in degree $2 \dim(\fn)-1$. By looking at \eqref{eqnarray:second line}, it is clear that both claims boil down to proving that
$$
\RHom_{\Rep (B)}(V_\nu, \Sym \fn \otimes  \Lambda^{\dim \fn} \fn)
\simeq
(V_{-w_0(\nu)})^T [-\dim(\fn)].
$$
We prove this in the next two steps.

\sssec*{Step 6}

By \cite[Section 1.3.3]{GYD}, the Serre functor for $\Rep(B)$ equals the functor $  - \otimes \Lambda^{\dim \fn} \fn[\dim \fn]$. Recall also that the defining property of $\Serre_\C$ requires the second argument to be compact, see Remark \ref{rem: remark on Serre functor}. In our case, we proceed as follows:
\begin{eqnarray}
\nonumber
\RHom_{\Rep (B)}(V_\nu, \Sym \fn \otimes \Lambda^{\dim \fn} \fn)
& \simeq &
\bigoplus_{m \geq 0}\RHom_{\Rep (B)}(V_\nu, \Sym^m \fn \otimes \Lambda^{\dim \fn} \fn)
\nonumber
\\
& \simeq &
\bigoplus_{m \geq 0}\RHom_{\Rep (B)}(\Sym^m \fn, V_\nu)^* [-\dim \fn]
\nonumber
\\
& \simeq &
\bigoplus_{m \geq 0}
\RHom_{\Rep (B)}(V_{-w_0(\nu)}, \Sym^m \fn^* )^* [-\dim \fn].
\nonumber
\end{eqnarray}
\sssec*{Step 7}

Thanks to \eqref{eqn:Kostant}, we know that $\RHom_{\Rep (B)}(V_{-w_0(\nu)}, \Sym \fn^* )$ is finite dimensional: in particular, we can replace the direct sum above with a direct product. Hence,
\begin{eqnarray}
\nonumber
\RHom_{\Rep (B)}(V_\nu, \Sym \fn \otimes \Lambda^{\dim \fn} \fn)
& \simeq &
\prod_{m \geq 0}
\RHom_{\Rep B}(V_{-w_0(\nu)}, \Sym^m \fn^* )^* [-\dim \fn]
\\
\nonumber
& \simeq &
\Bigt{
\bigoplus_{m \geq 0}
\RHom_{\Rep B}(V_{-w_0(\nu)}, \Sym^m \fn^* )
}^* [-\dim \fn]
\\
\nonumber
& \simeq &
\RHom_{\Rep B}(V_{-w_0(\nu)}, \Sym \fn^* )
^* [-\dim \fn]
\\
\nonumber
& \simeq &
(V_{-w_0(\nu)})^T [-\dim \fn],
\end{eqnarray}
where the last step used \eqref{eqn:Kostant} again.
This concludes the proof.
\end{proof}

\ssec{Shearing} \label{ssec:shifts of grading}

Since the stack $\N/G$ is endowed with a natural $\Gm$-action, we can consider the sheared version $\QCoh(\N/G)^\Rightarrow$ of $\QCoh(\N/G)$.
It turns out that $\QCoh(\N/G)^\Rightarrow$  is still proper, and the goal of this section is to compute its Serre functor.\footnote{For similar computations with the shearing operation, see \cite{strong-gluing} and  \cite{DL}.}

Compared to the result of Proposition \ref{prop:Serre on QCoh N/G}, this computation is not surprising, but it provides the necessary link between that proposition and Theorem \ref{thm:Serre on ICOH_N (LSG)} below.

\sssec{}

To be explicit, the $\Gm$-action on $\N/G$ we are using is induced by the $\Gm$-action on $(T^*(G/B))/G \simeq \fn/B$ by homotheties on $\fn$. Put another way: in the expression for $R$ appearing in \eqref{eqnarray:first line}, elements of $\Sym^m \fn^*$ are given weight $m$.

\begin{cor}
The DG category $\QCoh(\N/G)^\Rightarrow$ is compactly generated and proper.
\end{cor}

\begin{proof}
Lemma \ref{lem:conservative Rightarrows} ensures that $(\pi_*)^\Rightarrow: \QCoh(\N/G)^\Rightarrow \to \QCoh(\pt/G)^\Rightarrow \simeq \QCoh(\pt/G)$ is conservative. Hence, its left adjoint $(\pi^*)^\Rightarrow$ generates the target under colimits, so that $\QCoh(\N/G)^\Rightarrow$ is compactly generated by the objects $(\pi^*)^\Rightarrow(V_\lambda)$ for $\lambda \in \Lambda^{\dom}$. Properness then follows as in the previous section. 
\end{proof}

\begin{cor} \label{cor:sheared serre functor on N mod G}

The Serre functor on the DG category $\QCoh(\N/G)^\Rightarrow$ is the functor 
$$
\ker
\bigt{
\id
\to
 (j_*)^\Rightarrow (j^*)^{\Rightarrow}
}.
$$
\end{cor}

\begin{proof}
Denote by $R^{\Rightarrow}$ and $R^{\times, \Rightarrow}$ the sheared versions of  $R$ and $R^\times$.
Arguing as before, it suffices to construct, for each $\nu \in \Lambda^{\dom}$, an isomorphism
\begin{equation} \label{eqn:serre condition for shifted Qcoh N mod G}
\RHom_{\Rep(G)}
\Bigt{
V_\nu, 
\ker(R^\Rightarrow \to R^{\times, \Rightarrow})
}
\simeq
\RHom_{\Rep(G)}
\bigt{
V_{-w_0(\nu)}, R^\Rightarrow}^*.
\end{equation}
To determine both sides explicitly, we need to decompose both $R$ and $R^{\times}$ as $(G \times \Gm)$-representation.
In view of \eqref{eqnarray:first line}, the $(G \times \Gm)$-decomposition of $R$ is the tautological one coming from the grading of $\Sym \fn^*$:
$$
R
\simeq
\bigoplus_{\lambda, m}
\RHom_{\Rep (B)}(V_\lambda, \Sym^m \fn^*) \otimes V_\lambda.
$$
It follows that
$$
R^\Rightarrow
\simeq
\bigoplus_{\lambda, m}
\RHom_{\Rep (B)}(V_\lambda, \Sym^m \fn^*) 
\otimes V_\lambda
[-2m],
$$
so that the RHS of \eqref{eqn:serre condition for shifted Qcoh N mod G} can be rewritten as
\begin{equation} \label{eqn:RHS in a proof}
\RHom_{\Rep(G)}
\bigt{
V_{-w_0(\nu)}, R^\Rightarrow}^*
\simeq
\left(
\bigoplus_{m \in \NN}
\RHom_{\Rep (B)}(V_{-w_0\nu}, \Sym^m \fn^*) 
[-2m]
\right) ^*.
\end{equation}
Similarly, in view of \eqref{eqnarray:second line}, the $(G \times \Gm)$-decomposition of $\ker(R \to R^\times)$ is
$$
\ker(R \to R^\times)
\simeq
\bigoplus_{\lambda,m}
\Bigt{
\RHom_{\Rep (B)}(V_\lambda, \Sym^m \fn \otimes \Lambda^{\dim \fn} \fn)
}
\otimes 
V_\lambda
[-\dim \fn],
$$
with the expression in parentheses of weight $(-m- \dim \fn)$. Hence, 
$$
\ker(R^\Rightarrow \to R^{\times, \Rightarrow})
\simeq
\ker(R \to R^\times)^\Rightarrow
\simeq
\bigoplus_{\lambda,m}
\RHom_{\Rep (B)}(V_\lambda, \Sym^m \fn 
\otimes \Lambda^{\dim \fn} \fn
)
\otimes 
V_\lambda
[2m +  \dim \fn].
$$
From this, we see that the LHS of \eqref{eqn:serre condition for shifted Qcoh N mod G} equals
$$
\RHom_{\Rep(G)}
\Bigt{
V_\nu, 
\ker(R^\Rightarrow \to R^{\times, \Rightarrow})
}
\simeq
\bigoplus_{m}
\RHom_{\Rep (B)}(V_\nu, \Sym^m \fn \otimes \Lambda^{\dim \fn} \fn) 
[2m + \dim \fn].
$$
Comparing this equation with \eqref{eqn:RHS in a proof}, it remains to exhibit, for each $\nu$ and $m$, an isomorphism
$$
\RHom_{\Rep (B)}(V_\nu, \Sym^m \fn \otimes \Lambda^{\dim \fn} \fn) 
[\dim \fn]
\simeq
\RHom_{\Rep (B)}(V_{- w_0 (\nu)}, \Sym^m \fn^* )^*.
$$ 
Such isomorphism is the one induced by the Serre functor of $\Rep (B)$.
\end{proof}

\ssec{The main Serre computation} \label{ssec:main serre computation}

In this section, we finally show that the Serre functor on $\ICoh_\N(\Omega \g/G)$ equals the temperization functor up to a cohomological shift: this is the content of Theorem \ref{thm:Serre on ICOH_N (LSG)}. A key ingredient will be the following pair of Koszul duality equivalences.

\begin{lem} \label{lem:Koszul-pippo}
There are natural \virg{Koszul duality} equivalences
\begin{equation} \label{eqn:Koszul duality for proof}
 \ICoh_\N(\Omega \g/G)
  \simeq 
  \ICoh((\g^*/G)^\wedge_{\N/G})^\Rightarrow.
\end{equation}
\begin{equation} \label{eqn:Koszul duality for proof -Zer support}
 \QCoh(\Omega \g/G)
  \simeq 
  \ICoh((\g^*/G)^\wedge_{\pt/G})^\Rightarrow,
\end{equation}
where $\N$ is regarded as a subscheme of $\g^*$.
\end{lem}

\begin{rem}
As the proof below shows, on the RHS of \eqref{eqn:Koszul duality for proof} and \eqref{eqn:Koszul duality for proof -Zer support} we could replace $\ICoh$ by $\QCoh$. However, the functoriality of ind-coherent sheaves is more convenient when dealing with formal completions: the usage of Lemma \ref{lem:Koszul-pippo} in the proof of Theorem \ref{thm:Serre on ICOH_N (LSG)} will make this clear (see also Section \ref{ssec:indcoh-sing-supp}).
\end{rem}

\begin{proof}
We will only prove the first assertion; for the second one, follow the exact same steps using the inclusion $\pt \hto \g^*$ of the origin instead of the inclusion of the nilpotent cone.

Recall that we have been abusing notation: the DG category $ \ICoh_\N(\Omega \g/G)$ should be more properly denoted by $ \ICoh_{\N/G}(\Omega \g/G)$, since possible singular supports of ind-coherent sheaves live inside $\g^*/G$, and in the case at hands we are looking at the subset $\N/G \subset \g^*/G$.
For clarity, in this proof, we will use the more precise notation.
Since singular support can be computed smooth-locally (\cite[Section 8]{AG1}), we have:
$$
\ICoh_{\N/G}(\Omega \g/G)
\simeq
\ICoh(\Omega \g/G)
\ustimes{\ICoh(\Omega \g)}
\ICoh_{\N}(\Omega \g),
$$
where the two maps in the fiber product are the natural pullback and the inclusion $\Xi_{\N \hto \g^*}$, respectively. 
Now, we use the Koszul duality equivalences of \cite[Proposition 12.4.2]{AG1}:
$$
\ICoh(\Omega \g)
\simeq
\Sym(\g[-2])\mod
\simeq
\QCoh(\g^*)^\Rightarrow;
$$
$$
\ICoh(\Omega \g/G)
\simeq
(\Sym(\g[-2])\mod)^G
\simeq
\QCoh(\g^*/G)^\Rightarrow.
$$
The first of these two equivalences transforms singular support on the LHS into set-theoretic support on the RHS. This is proven in \cite[Section 9.1, especially 9.1.6 and Corollary 9.1.7]{AG1}; see also \cite[Section 2.2]{strong-gluing} for a slightly different point of view. 
In particular, 
$$
\ICoh_\N(\Omega \g)
\simeq
\QCoh((\g^*)^\wedge_\N)^\Rightarrow,
$$
where$(\g^*)^\wedge_\N$ denotes the formal completion of the closed embedding $\N \hto \g^*$.
Since $\g^*$ and $\g^*/G$ are smooth, we can identify quasi-cohererent sheaves with ind-coherent sheaves on them, so that 
$$
\ICoh(\Omega \g)
\simeq
\ICoh(\g^*)^\Rightarrow;
$$
$$
\ICoh(\Omega \g/G)
\simeq
\ICoh(\g^*/G)^\Rightarrow;
$$
$$
\ICoh_\N(\Omega \g)
\simeq
\ICoh((\g^*)^\wedge_\N)^\Rightarrow.
$$
All in all, Koszul duality yields
$$
\ICoh_{\N/G}(\Omega \g/G)
\simeq
\ICoh(\g^*/G)^\Rightarrow
\ustimes{\ICoh(\g^*)^\Rightarrow}
\ICoh((\g^*)^\wedge_\N)^\Rightarrow,
$$
which is in turn equivalent to
$$
\left (
\ICoh(\g^*/G)
\ustimes{\ICoh(\g^*)}
\ICoh((\g^*)^\wedge_\N)
\right )
^\Rightarrow.
$$
Let $q: \g^* \to \g^*/G$ be the quotient and $j: \g^* - \N \hto \g^*$ the open embedding complementary to the nilpotent cone.
To conclude the proof, it remains to show that the $\Gm$-equivariant functor
\begin{equation} \label{eqn:auxiliary Koszul duality}
  \ICoh((\g^*/G)^\wedge_{\N/G})
  \longto
\ICoh(\g^*/G)
\ustimes{\ICoh(\g^*)}
\ICoh((\g^*)^\wedge_\N)
\end{equation}
induced by the inclusion $  \ICoh((\g^*/G)^\wedge_{\N/G}) \hto
\ICoh(\g^*/G)$
is an equivalence.
This is clear: both sides of \eqref{eqn:auxiliary Koszul duality} identify with the full subcategory of $\ICoh(\g^*/G)$ spanned by those objects $\F$ such that $j^! q^! \F \simeq 0$.
\end{proof}

\sssec{}

In the above discussion, $\N$ was naturally viewed as a subscheme of $\g^*$. However, to use the Serre computations of the previous sections, we prefer to realize $\N$ as a closed subscheme of $\g$. Thus, we fix once and for all a $G$-equivariant identification $\g \simeq \g^*$. Incorporating this into the above Koszul dualities, we obtain equivalences
\begin{equation} \label{eqn:Koszul dualities better}
 \ICoh_\N(\Omega \g/G)
  \simeq 
  \ICoh((\g/G)^\wedge_{\N/G})^\Rightarrow,
  \hspace{.4cm}
 \QCoh(\Omega \g/G)
  \simeq 
  \ICoh((\g/G)^\wedge_{\pt/G})^\Rightarrow,
\end{equation}
where now the shearing is such that $(\Sym \g^*)^\Rightarrow \simeq \Sym (\g^*[-2])$.
To guide the reader through the shearings that will follow, it suffices to remember that dual Lie algebras ($\g^*$ and $\fn^*$) have weight $1$. In particular, we are in agreement with the shearing of Section \ref{ssec:shifts of grading}.

\sssec{}

Let $\iota_0: (\g/G)^\wedge_{\pt/G} \to (\g/G)^\wedge_{\N/G}$ be the map induced by the inclusion of the origin $0 \in \N$.
This map is an inf-closed embedding (see Section \ref{sssec:inf stuff}), and thus it yields the adjunction
$$ 
\begin{tikzpicture}[scale=1.5]
\node (a) at (0,1) {$\ICoh((\g/G)^\wedge_{\pt/G}) $};
\node (b) at (4,1) {$\ICoh((\g/G)^\wedge_{\N/G}) $.};
\path[ ->,font=\scriptsize,>=angle 90]
([yshift= 1.5pt]a.east) edge node[above] {$(\iota_0)_*^\ICoh$} ([yshift= 1.5pt]b.west);
\path[->,font=\scriptsize,>=angle 90]
([yshift= -1.5pt]b.west) edge node[below] {$(\iota_0)^!$} ([yshift= -1.5pt]a.east);
\end{tikzpicture}
$$
These two adjoint functors are both $\Gm$-equivariant with respect to the two natural $\Gm$-actions on both sides, hence they can be sheared.

\sssec{}

We have seen above that Koszul duality interchanges singular support and the usual set-theoretic support. Consequently, under the above equivalences, the comonad 
 $$
 \Xi_{0 \to \N} \circ  \Psi_{0 \to \N}:
 \ICoh_\N(\Omega \g/G)
 \tto
 \QCoh(\Omega \g/G)
 \hto
 \ICoh_\N(\Omega \g/G) 
$$
goes over to the comonad
\begin{equation} \label{eqn:long1}
\bigt{ (\iota_0)_*^\ICoh}^\Rightarrow
 \circ 
 \bigt{ \iota_0^! }^\Rightarrow
  :
\ICoh((\g/G)^\wedge_{\N/G})^\Rightarrow
\tto
\ICoh((\g/G)^\wedge_{\pt/G}) ^\Rightarrow
\hto
\ICoh((\g/G)^\wedge_{\N/G})^\Rightarrow.
\end{equation}

\sssec{}

Now recall that $\ICoh_\N(\Omega \g/G)$ is proper: in Section \ref{sec:GL for P1 and outline Thm C} we showed that $\ICoh_\N(\Omega \g/G)$ is equivalent to $\Dmod(\Bun_\Gch(\bbP^1))$, and in Section \ref{sssec:properness of DG cats} we proved that the latter DG category is proper.
We will nevertheless give a more direct proof: this will help us progress with the computation of the Serre functor of $\ICoh_\N(\Omega \g/G)$. 
Using Koszul duality to identify $\ICoh_\N(\Omega \g/G) 
\simeq
 \ICoh((\g/G)^\wedge_{\N/G})^\Rightarrow$, our plan will be to prove the properness of the latter DG category. The starting point is Lemma \ref{lem:cpt gen of spectral sph cat}, in which we exhibit a convenient collection of compact generators.
 
 \begin{rem}
In the sequel, we will denote by
$$ 
\begin{tikzpicture}[scale=1.5]
\node (a) at (0,1) {$\QCoh(\N/G)$};
\node (b) at (2.3,1) {$\ICoh(\N/G)$};
\path[right hook ->,font=\scriptsize,>=angle 90]
([yshift= 1.5pt]a.east) edge node[above] {$\Xi_{\N/G}$} ([yshift= 1.5pt]b.west);
\path[->>,font=\scriptsize,>=angle 90]
([yshift= -1.5pt]b.west) edge node[below] {$\Psi_{\N/G}$} ([yshift= -1.5pt]a.east);
\end{tikzpicture}
$$
the natural adjunction. Observe that both functors are $\Gm$-equivariant and therefore they can be sheared.
\end{rem}

\begin{lem} \label{lem:cpt gen of spectral sph cat}
Let $\pi: \N/G \to BG$ and $f: \N/G \to (\g/G)^\wedge_{\N/G}$ be the obvious maps. As above, their associated pullback and pushforward functors can be sheared.
We claim that the objects
\begin{equation} \label{eqn:long2}
\F_\lambda
 := 
(f_*^{\ICoh})^\Rightarrow 
(\Xi_{\N/G})^\Rightarrow
 (\pi^*)^\Rightarrow(V_\lambda),
\hspace{.4cm}
\mbox{ for all $\lambda \in \Lambda^{\dom}$,}
\end{equation}
form a collection of compact generators of $\ICoh((\g/G)^\wedge_{\N/G})^\Rightarrow$.
\end{lem}

\begin{proof}
As mentioned earlier, the objects $ (\pi^*)^\Rightarrow(V_\lambda)$ compactly generate $\QCoh(\N/G)^\Rightarrow$. Hence, it suffices to prove that the functor 
$$
(f_*^{\ICoh})^\Rightarrow 
(\Xi_{\N/G})^\Rightarrow:
\QCoh(\N/G)^\Rightarrow
\longto
\ICoh((\g/G)^\wedge_{\N/G})^\Rightarrow
$$
has a continuous and conservative right adjoint. In view of Lemma \ref{lem:conservative Rightarrows}, we can remove the shifts and instead prove that $f_*^{\ICoh} \circ \Xi_{\N/G}$ admits a continuous and conservative right adjoint. 
Since $f$ is a nil-isomorphism, the continuous functor $f^!$ is conservative and right adjoint to $f_*^\ICoh$, see \cite[Volume II, Chapter 3.3, especially Proposition 3.1.2]{Book}. Hence,
$$
\bigt{
f_*^{\ICoh} \circ
\Xi_{\N/G}
}^R
\simeq
\Psi_{\N/G} \circ f^!,
$$
which is evidently continuous. 
Let us show conservativity. 
Observe first that the functor
$$
\Upsilon_{(\g/G)^\wedge_{\N/G}}:
\QCoh ( (\g/G)^\wedge_{\N/G})
\longto
\ICoh( (\g/G)^\wedge_{\N/G})
$$
is an equivalence: this follows from Section \ref{sssec:ICoh on formal as tensor product} and descent (to take care of the stackyness).
Thus, it suffices to prove that $
\Psi_{\N/G} \circ f^! \circ \Upsilon_{(\g/G)^\wedge_{\N/G}}
$ is conservative. In general, $\Upsilon$ intertwines $*$-pullbacks of quasi-coherent sheaves with $!$-pullbacks of ind-coherent sheaves; in our case, this implies that
$$
\Psi_{\N/G} \circ f^! \circ \Upsilon_{(\g/G)^\wedge_{\N/G}}
\simeq
\Psi_{\N/G} \circ \Upsilon_{{\N/G}} \circ f^*.
$$
The quasi-smoothness of $\N/G$ guarantees that $\Psi_{\N/G} \circ \Upsilon_{{\N/G}}$ is an equivalence: it is the functor of tensoring with a shifted line bundle, see \cite[Section 7]{ICoh} or Lemma \ref{lem:anomaly for nilcone}. It remains then to prove that 
$$
f^*:
\QCoh ( (\g/G)^\wedge_{\N/G})
\longto
\QCoh(\N/G)
$$
is conservative. This can be seen in various ways. For instance, using the equivalence $\Upsilon_{(\g/G)^\wedge_{\N/G}}$ again, the assertion is equivalent to the conservativity of
$$
f^* \circ \Upsilon_{(\g/G)^\wedge_{\N/G}}
\simeq
\Upsilon_{\N/G} \circ f^!.
$$
The latter is clear: $f^!$ is conservative as mentioned earlier, while $\Upsilon_{\N/G}$ is even fully faithful.
\end{proof}

\begin{rem} \label{rem:boring}
For later usage, let us compare the object
$
\F_\lambda
 := 
(f_*^{\ICoh})^\Rightarrow 
(\Xi_{\N/G})^\Rightarrow
 (\pi^*)^\Rightarrow(V_\lambda)$
of \eqref{eqn:long2} with the following similar object:
$$
\F'_\lambda
 := 
(f_*^{\ICoh})^\Rightarrow 
 (\pi^!)^\Rightarrow
 (\Upsilon_{\pt/G})^\Rightarrow
 (V_\lambda).
 $$
We claim that these two objects differ by a cohomological shift: precisely, $\F'_\lambda \simeq \F_\lambda [\dim(\N/G)]$. To see this, observe that 
$$
\pi^! \circ \Upsilon_{\pt/G} 
\simeq 
\Upsilon_{\N/G} \circ  \pi^* 
\simeq 
\Xi_{\N/G} \circ  \Psi_{\N/G} \circ  \Upsilon_{\N/G} \circ \pi^*
\simeq
\Xi_{\N/G} \bigt{ \Psi_{\N/G}(\omega_{\N/G}) \otimes \pi^*}
\simeq
\Xi_{\N/G} \circ  \pi^* [\dim(\N/G)],
$$
the last step being an application of Lemma \ref{lem:anomaly for nilcone}.
\end{rem}

\begin{lem} \label{lem:properness of spectral sph cat}
The DG category $ \ICoh((\g/G)^\wedge_{\N/G})^\Rightarrow$ is proper.
\end{lem}

\begin{proof}
Since the $\F_\lambda$ above form a collection of compact generators, it suffices to prove
that 
$$
\RHom_{\ICoh((\g/G)^\wedge_{\N/G})^\Rightarrow}(\F_\lambda, \F_\mu)
$$ 
is finite dimensional for any pair $\lambda, \mu \in \Lambda^{\dom}$.
Thanks to the fact that $f$ is a nil-isomorphism, we obtain by adjunction that
$$
\RHom
_{\ICoh((\g/G)^\wedge_{\N/G})^\Rightarrow}
(\F_\lambda, \F_\mu)
\simeq
\RHom_{\ICoh(\N/G)^\Rightarrow}
\Bigt{
(\Xi_{\N/G} \circ  \pi^*)^\Rightarrow (V_\lambda), \U(\Tang_{(\N/G)/(\g/G)})^\Rightarrow \circ  (\Xi_{\N/G} \circ \pi^*)^\Rightarrow (V_\mu)
},
$$
where $\U(\Tang_{(\N/G)/(\g/G)})$ is the universal envelope of the Lie algebroid $\Tang_{(\N/G)/(\g/G)} \to \Tang_{\N/G}$. See \cite[Volume II, Chapter 8]{Book} for these notions: in particular, $\U(\Tang_{(\N/G)/(\g/G)})$ is a monad acting on $\ICoh(\N/G)$, and relative tangent complexes are regarded as ind-coherent sheaves\footnote{
In the present case, and in general when the relative cotangent complex $\bbL_{Y/Z}$ of a map $Y \to Z$ is perfect, $\Tang_{Y/Z}$ is obtained by applying $\Upsilon_Y$ to the quasi-coherent dual of $\bbL_{Y/Z}$.
}.

To proceed, let us compute the relative tangent complex $\Tang_{(\N/G)/(\g/G)}$ and then its universal envelope.
We will use the isomorphism
$$
\N/G \simeq \g/G \times_{\fc_G} 0,
$$
where $\fc_G \simeq \Spec(\Sym (\g^*)^G)$ is isomorphic, after our $G$-equivariant identification $\g \simeq \g^*$, to the Chevalley space of Section \ref{sssec:Chevalley space}.
Since the fiber product on the RHS is derived (as well as classical, in view of the flatness of $\g \to \fc_G$), we can compute the relative tangent complex  algorithmically:
$$
\Tang_{(\N/G)/(\g^*/G)}
\simeq
(p_{\N/G})^! (\Tang_{0 /\fc_G})
\simeq
\omega_{\N/G} \otimes \Tang_{0/ \fc_G}
\simeq
\omega_{\N/G} \otimes \Tang_{\fc_G,0}[-1].
$$
Consequently, the functor underlying the monad $\U(\Tang_{(\N/G)/(\g^*/G)})$ is just the functor of tensoring with the graded vector space $\Sym(\Tang_{\fc_G,0}[-1])$.
Using the fully faithfulness of $\Xi_{\N/G}$ and turning on the shearing, we conclude that
\begin{equation} \label{eqn:RHS of big serre computation}
\RHom
_{\ICoh((\g^*/G)^\wedge_{\N/G})^\Rightarrow}
(\F_\lambda, \F_\mu)
\simeq
\RHom_{\QCoh(\N/G)^\Rightarrow}
\bigt{
(\pi^*)^\Rightarrow (V_\lambda) ,(\pi^*)^\Rightarrow (V_\mu)
}
\otimes 
\Sym(\Tang_{\fc_G,0}[-1])^{\Rightarrow}.
\end{equation}
Thanks to the already established properness of $\QCoh(\N/G)^\Rightarrow$, it remains to check that $\Sym(\Tang_{\fc_G,0}[-1])^{\Rightarrow} $ is finite dimensional.
For this, we need to recall the weight decomposition of $\Tang_{\fc_G,0}$ (or, which is the same, of $\fc_G$, since the latter is a vector space).
We have
$$
\Tang_{\fc_G,0} 
\simeq 
\fz_G
\oplus 
\bigoplus_{i =1}^{r_G} \fl_{d_i},
$$
where $\fz_G = \on{Lie}(Z_G)$ is in weight $-1$, each $\fl_{d_i}$ is a line in weight $-d_i$ (the negative of the $i$-th fundamental invariant of the group), and $r_G$ is the semisimple rank.
It follows that 
\begin{equation} \label{eqn:shifted Chevalley space}
(\Sym  (\Tang_{\fc_G,0} [-1]))^\Rightarrow
\simeq 
\Sym (\fz_G[1])
\otimes
\bigotimes_{i=1}^{r_G}
\Sym( \fl_{d_i}[2 d_i - 1]).
\end{equation}
This is an exterior algebra with finitely many generators, hence in particular finite dimensional.
\end{proof}

\begin{thm} \label{thm:Serre on ICOH_N (LSG)}
The Serre functor on $\ICoh_\N(\Omega \g/G)$ equals the functor $\Xi_{0 \to \N} \circ  \Psi_{0 \to \N} [-\dim(G)]$.
\end{thm}

\begin{proof}
Consider again the natural map $\iota_0: (\g/G)^\wedge_{\pt/G} \to (\g/G)^\wedge_{\N/G}$.
By Koszul duality, the statement of the theorem is equivalent to the fact that the Serre functor of the DG category $\ICoh((\g/G)^\wedge_{\N/G})^\Rightarrow$ is the functor $\bigt{ (\iota_0)_*^\ICoh}^\Rightarrow
 \circ 
 \bigt{ \iota_0^! }^\Rightarrow [- \dim(G)]$.
This is what we will prove below, in several steps.

\sssec* {Step 1}
For any pair $(\lambda, \mu)$ of dominant weights, we need to provide an isomorphism
$$
\RHom
_{\ICoh((\g/G)^\wedge_{\N/G})^\Rightarrow}
\bigt{
\F_\mu, 
(\iota_{0,*}^\ICoh)^\Rightarrow \circ (\iota_0^!)^\Rightarrow (\F_\lambda)[-\dim(G)]
}
\stackrel ? \simeq
\RHom_{\ICoh((\g/G)^\wedge_{\N/G})^\Rightarrow}(\F_\lambda, \F_\mu)^*.
$$
Using Remark \ref{rem:boring}, this is equivalent to providing an isomorphism
$$
\RHom
_{\ICoh((\g/G)^\wedge_{\N/G})^\Rightarrow}
\bigt{
\F'_\mu, 
(\iota_{0,*}^\ICoh)^\Rightarrow \circ (\iota_0^!)^\Rightarrow (\F'_\lambda)[-\dim(G)]
}
\stackrel ? \simeq
\RHom_{\ICoh((\g/G)^\wedge_{\N/G})^\Rightarrow}(\F'_\lambda, \F'_\mu)^*.
$$
The reason for replacing $\F_\nu$ with $\F'_\nu$ is that the primed expressions are more amenable to the base-change manipulations of Step $2$ below. 

We have already computed the RHS in the lemma above. Taking that calculation into account and setting 
$$
W:= \Sym(\Tang_{\fc_G,0}[-1])^{\Rightarrow},
$$
it remains to prove that
\begin{equation} \label{eqn:Serre for ICoh for spectral sph category}
\hspace{-4cm}
\RHom
_{\ICoh((\g/G)^\wedge_{\N/G})^\Rightarrow}
\bigt{
\F'_\mu, 
(\iota_{0,*}^\ICoh)^\Rightarrow \circ (\iota_0^!)^\Rightarrow (\F'_\lambda)[-\dim(G)]
}
\end{equation}
$$
\hspace{6cm}
\stackrel ? \simeq \;\;
\RHom_{\QCoh(\N/G)^\Rightarrow}
\bigt{
(\pi^*)^\Rightarrow V_\lambda, (\pi^*)^\Rightarrow V_\mu
}^*
\otimes 
W^*.
$$

\sssec* {Step 2}

Let us manipulate the LHS of \eqref{eqn:Serre for ICoh for spectral sph category}. Base-change gives
$$
\iota_{0,*}^\ICoh \circ \iota_0^! (\F'_\lambda)
\simeq
\beta_*^\ICoh \alpha^!(V_\lambda),
$$
where 
$$
BG \xleftarrow{\alpha}(\N/G )^\wedge_{BG} \xto{\beta} (\g/G)^\wedge_{\N/G}
$$
are the natural maps. 
Then
$$
\RHom_{\ICoh((\g/G)^\wedge_{\N/G})^\Rightarrow}
\bigt{
\F'_\mu, 
(\iota_{0,*}^\ICoh \circ \iota_0^!)^\Rightarrow (\F'_\lambda)
}
\simeq
\RHom_{\ICoh(\N/G)^\Rightarrow}
\bigt{
(\pi^!)^\Rightarrow (V_\mu),
(
f^!
\beta_*^\ICoh \alpha^!
)^\Rightarrow
(V_\lambda)
},
$$
where abusing notation we are considering $V_\lambda$ and $V_\mu$ as objects of $\ICoh(BG)$ by means on the equivlaence $\Upsilon_{BG}$.
Consider now the fiber square
\begin{equation} 
\nonumber
\begin{tikzpicture}[scale=1.5]
\node (00) at (0,0) {$  (\N/G)^\wedge_{BG}$};
\node (10) at (4,0) {$  (\g/G)^\wedge_{\N/G}$.};
\node (01) at (0,1) {$\bigt{\N/G \times_{\g/G} \N/G  }^\wedge_{BG}$};
\node (11) at (4,1) {$\N/G$};
\node (middle) at (2.5,1) {$  (\N/G)^\wedge_{BG}$};
\path[->,font=\scriptsize,>=angle 90]
(00.east) edge node[above] {$\beta$} (10.west); 
\path[->,font=\scriptsize,>=angle 90]
(01.east) edge node[above] {$q_2$} (middle.west); 
\path[->,font=\scriptsize,>=angle 90]
(middle.east) edge node[above] {$\iota_\N$} (11.west); 
\path[->,font=\scriptsize,>=angle 90]
(01.south) edge node[right] {$q_1$} (00.north);
\path[->,font=\scriptsize,>=angle 90]
(11.south) edge node[right] {$f$} (10.north);
\end{tikzpicture}
\end{equation}
A further base-change yields $f^! \beta_*^\ICoh \simeq (\iota_\N)_*^\ICoh (q_2)_*^\ICoh (q_1)^!$.
\sssec* {Step 3}

The convolution right action of $\Omega(\fc_G) := \pt \times_{\fc_G} \pt$ on $\N/G = \g/G \times_{\fc_G} \pt$
 induces the isomorphism 
$$
(\N/G)^\wedge_{BG} \times \Omega(\fc_G)
\xto{(n, x) \mapsto (n \cdot x, n)}
\bigt{\N/G \times_{\g/G} \N/G  }^\wedge_{BG}.
$$
From this point of view, the functor $(q_2)_*^\ICoh (q_1)^!$ is the functor of acting with $\omega_{\Omega \fc_G}$ on objects of $\ICoh((\N/G)^\wedge_{BG})$.
Next, notice that the action and the projection $(\N/G)^\wedge_{BG} \times \Omega(\fc_G) \rr (\N/G)^\wedge_{BG} $ are coequalized by $\alpha: (\N/G)^\wedge_{BG} \to BG$. In other words, the group $\Omega(\fc_G) $ acts trivially on objects in the essential image of $\alpha^!$.
Hence, 
$$
f^!
\beta_*^\ICoh \alpha^!(V_\lambda)
\simeq 
(\iota_\N)_*^{\ICoh}\alpha^!(V_\lambda) 
\otimes 
(p_{\Omega \fc_G})_*^\ICoh 
(\omega_{\Omega \fc_G}).
$$ 

\sssec* {Step 4}

The above formula needs to be applied in its sheared version: this amounts to apply ${}^\Rightarrow$ to the functors and to the graded vector space $(p_{\Omega \fc_G})_*^\ICoh 
(\omega_{\Omega \fc_G})$. By Lemma \ref{lem:Omega V global sections}, we have:
$$
\Bigt{ (p_{\Omega \fc_G})_*^\ICoh 
(\omega_{\Omega \fc_G})
}
^\Rightarrow
\simeq
(\Sym(\Tang_{\fc_G,0}[-1])) ^{\Rightarrow},
$$
which is vector space $W$ defined earlier.
We obtain that
$$
\RHom
_{\ICoh((\g/G)^\wedge_{\N/G})^\Rightarrow}
\bigt{
\F_\mu, 
( \iota_{0,*}^\ICoh \circ \iota_0^! )^\Rightarrow
 (\F_\lambda)
}
\simeq
\RHom_{\ICoh(\N/G)^\Rightarrow}
\bigt{
(\pi^!)^\Rightarrow(V_\mu),
((\iota_\N)_*^{\ICoh}\alpha^!)^\Rightarrow (V_\lambda) 
}
\otimes W.
$$

\sssec* {Step 5}

It is clear that 
$$
(\iota_\N)_*^{\ICoh}\alpha^!(V_\lambda) 
\simeq 
\ker
\Bigt{
\pi^! (V_\lambda)
\longto
j_*^\ICoh j^! \pi^! (V_\lambda)
},
$$
where $j: \N^\times/G \hto \N/G$ is the inclusion of the punctured nilpotent cone. Let $R:= H^0(\N, \O_\N)$ and $R^\times := H^*(\N^\times, \O)$, viewed as $(G \times \Gm)$-representations as discussed in the previous sections.
We thus have:
\begin{eqnarray}
\nonumber
\RHom
_{\ICoh((\g/G)^\wedge_{\N/G})^\Rightarrow}
\bigt{
\F_\mu, 
( \iota_{0,*}^\ICoh \circ \iota_0^! )^\Rightarrow (\F_\lambda)
}
& & \hspace{7cm}
\end{eqnarray}
\begin{eqnarray}
& \simeq &
\ker
\Bigt{
\RHom_{\QCoh(\N/G)^\Rightarrow}
(  (\pi^* )^\Rightarrow V_\mu, ( \pi^* )^\Rightarrow V_\lambda)
\to
\RHom
_{\QCoh(\N/G)^\Rightarrow}
(
(j^* \pi^* )^\Rightarrow V_\mu, ( j^*\pi^* )^\Rightarrow V_\lambda)
} \otimes W 
\nonumber
\\
& \simeq &
\ker
\Bigt{
\RHom_{\Rep (G)}(V_\mu, R^\Rightarrow \otimes V_\lambda)
\to
\RHom_{\Rep (G)}( V_\mu, (R^\times)^\Rightarrow \otimes V_\lambda)
} \otimes W
\nonumber
\\
\nonumber
& \simeq &
\RHom_{\Rep (G)}(V_\mu, \ker (R \to R^\times)^\Rightarrow \otimes V_\lambda)
 \otimes W
 \\
 \nonumber
& \simeq &
\RHom_{\QCoh(\N/G)^\Rightarrow}
\Bigt{ 
(\pi^*)^\Rightarrow V_\mu, \Serre( (\pi^* )^\Rightarrow V_\lambda)
}
 \otimes W,
\end{eqnarray}
where the last step used Corollary \ref{cor:sheared serre functor on N mod G}. Hence, the LHS of \eqref{eqn:Serre for ICoh for spectral sph category} equals
$$
\RHom_{\QCoh(\N/G)^\Rightarrow}
\Bigt{ 
(\pi^*)^\Rightarrow V_\mu, \Serre( (\pi^* )^\Rightarrow V_\lambda)
}
 \otimes W [ - \dim G].
$$

\sssec*{Step 6}

On the other hand, the RHS of \eqref{eqn:Serre for ICoh for spectral sph category} equals
$$
\RHom_{\QCoh(\N/G)^\Rightarrow}
\bigt{
(\pi^*)^\Rightarrow V_\lambda, (\pi^*)^\Rightarrow V_\mu
}^*
\otimes 
W^*
\simeq
\RHom_{\QCoh(\N/G)^\Rightarrow}
\bigt{ 
(\pi^*)^\Rightarrow V_\mu, \Serre( (\pi^* )^\Rightarrow V_\lambda)
}
\otimes W^*.
$$
Hence, it remains to show that $W^* \simeq W[- \dim G]$. For this, recall that
$$
W
:=
(\Sym  (\Tang_{\fc_G,0} [-1]))^\Rightarrow
\simeq 
\Sym (\fz_G[1])
\otimes
\bigotimes_{i=1}^{r_G}
\Sym( \fl_{d_i}[2 d_i - 1]),
$$
with each $\fl_{d_i}$ a line. Now, the classical formula 
$$
\dim G = \dim (Z_G ) + \sum_{i} (2 d_i-1)
$$
immediately yields the claim.
\end{proof}

\begin{cor} 
The Serre functor on $\ICoh_\N(\Omega \g/G)$ sends the monoidal unit $\Psi_{\N \to \g}(i_*^\ICoh(\kk_0))$ to $\Xi_{0 \to \N} (i_*(\kk_0))[-\dim(G)]$.
\end{cor}

\sec{Proof of Theorem \ref{mainthm: omega antitemp}} \label{sec:characterize temp via ThmC, prove ThmA}

In this section, we deduce Theorem \ref{mainthm: omega antitemp} from the combination of Theorem \ref{mainthm:Hren vanishing} and Theorem \ref{mainthm:B as unit of Sph-temp}.
In more detail: using the expression of $\B$ given by Theorem \ref{mainthm:B as unit of Sph-temp}, we obtain an explicit formula, see \eqref{eqn:formula for action of B}, for the Hecke action of $\B$ on $\Dmod(\Bun_G)$. 
In particular, we obtain an explicit formula for the object $\B \star \omega_{\Bun_G}$. We then show that the simplest case of Theorem \ref{mainthm:Hren vanishing} implies that $\B \star \omega_{\Bun_G} \simeq 0$.

\ssec{Renormalized functors} \label{ssec:self-duality}

Before describing the action of the tempered unit on objects of $\Dmod(\Bun_G)$, we need some information on $\B$ itself and on the DG category $\Dmod(G \backsl \GR /G)$.

\sssec{} \label{sssec:care with the push}

We claim that the DG category $\Dmod(G \backsl \GR /G)$ is naturally monoidal under convolution.
Naively, the convolution product is defined as the pull-push along the natural correspondence
\begin{equation} \label{eqn:convo diagram for neg Hecke cat}
G \backsl \GR /G \times G \backsl \GR /G
\xleftarrow{p}
G \backsl \GR \times^G \GR / G
\xto{m}
G \backsl \GR /G,
\end{equation}
where $p$ is the obvious projection and $m$ is the arrow induced by the multiplication of $\GR$.
However, since $m$ is not schematic (but only ind-schematic), specifying the pushforward to be used requires some care.

\sssec{}

To address this complication, we first observe that $\Dmod(G \backsl \GR /G)$ is tautologically comonoidal: this structure is induced by the very same correspondence as above, just read from right to left. In this case, the above issue about the pushforward does not arise: the functor $p_{*,\dR}$ is well-defined since $p$ is schematic (it is even smooth).

Then, to turn this comonoidal structure into a monoidal one, we need:

\begin{lem}
The DG category $\Dmod(G \backsl \GR /G)$ is self-dual.
\end{lem}

\begin{proof}
Letting $G(R)^1 := \ker(G(R) \to G)$ be the kernel of the evaluation map at $t = \infty$, we have $G(R) \simeq G \ltimes \GR^1$. It follows that $G \backsl \GR /G \simeq \GR^1/G$, where the quotient on the RHS is by the adjoint action.
Consider now the indscheme structure on $\GR^1 \simeq \colim_{d \geq 0} Y_d$, obtained by embedding $G$ into some $GL_n$ and then by setting $Y_d := \GR^1 \cap \GL_n(R)_{\leq d}$. Here, $\GL_n(R)_{\leq d} \subset \GL_n(R)$ is the scheme parametrizing invertible matrices whose entries are polynomials in $t^{-1}$ of degrees $\leq d$.
It is easy to see that the adjoint action of $G$ on $\GR^1$ preserves each $Y_d$, so that
$$
G \backsl \GR /G
\simeq
\uscolim{d \geq 0} \;
Y_d/G,
$$
along the structure closed embeddings $i_{d \to d'}: Y_d/G \hto Y_{d'}/G$, for $d \leq d'$.
It follows that
\begin{equation} \label{eqn:D(G G(R) G)}
\Dmod(G \backsl \GR /G)
\simeq
\uscolim{d \geq 0, i_*} \;
\Dmod(Y_d/G).
\end{equation}
By \cite[Theorem 0.2.2]{finiteness}, the DG categories $\Dmod(Y_d/G)$ are all compactly generated. Since the transition functors  obviously preserve compactness, formal nonsense shows that their colimit $\Dmod(G \backsl \GR /G)$ is also compactly generated, and thus dualizable.
Moreover, each $\Dmod(Y_d/G)$ is self-dual: this is \cite[Corollary 8.4.3]{finiteness}. Under these self-dualities, the dual of $\Dmod(G \backsl \GR /G)$ is expressed as
\begin{equation} \label{eqn:D(G G(R) G) - dual }
\Dmod(G \backsl \GR /G)^\vee
\simeq
\lim_{d \geq 0, i^!} \;
\Dmod(Y_d/G).
\end{equation}
The latter limit can be turned into a colimit by replacing the transition functors with their left adjoints, see \cite[Chapter 5]{HTT} or \cite[Vol. 1, Chapter 1, Corollary 5.3.4]{Book}  This shows that \eqref{eqn:D(G G(R) G)} and \eqref{eqn:D(G G(R) G) - dual } match.
\end{proof}
 
\begin{rem} \label{rem:Sph self-dual}
In a similar way, one proves that $\Sph_G$ is self-dual: it suffices to apply the same method to the colimit presentation $\Dmod(\Gr_G)^\GO \simeq \colim_{n \geq 0} \Dmod(Z_n/H_n)$ that appeared in the proof of Lemma \ref{lem:Grassmannian}.
\end{rem}
 
\sssec{}

The above self-duality allows us to dualize the comonoidal structure to obtain the monoidal structure we were looking for. Concretely, the convolution product is defined as the pull-push along \eqref{eqn:convo diagram for neg Hecke cat}, but the
 pushforward to be used is $m_{*,\ren}$,  the \emph{renormalized de Rham pushforward} along $m$ (see \cite[Section 9.3]{finiteness}), which is by definition the dual of $m^!$.

\sssec{}

In general, the renormalized de Rham pushforward along a map $h: \X = \colim_{i \in \I} X_i \to \Y$ of indschemes (of ind-finite type) admits an explicit description.
Indeed, unraveling the self-dualities as in the proof above, one easily checks that $h_{*,\ren}: \Dmod(\X) \to \Dmod(\Y)$ sends $\{\F_i\}_{i \in \I} \in \lim_{i \in \I^\op} \Dmod(X_i) \simeq \Dmod(\X)$ to the object 
$$
\colim_{i \in \I} \; (h_i)_*(\F_i)
\in
\Dmod(\Y),
$$
where $h_i: X_i \hto \X \to \Y$ is the natural (schematic) map.
This functor is different from $h_!$ in general; indeed, the latter is only partially defined and, when defined, it is given by the formula
$$
h_!
\Bigt{ 
\{\F_i \}_{i \in \I} } 
\simeq
\colim_{i \in \I} \; (h_i)_!(\F_i).
$$
However, the two functors agree when all the $h_i$ are proper, that is, when the map $h: \X \to \Y$ is ind-proper. 

\begin{rem} \label{rem:Hren def slick via renorm push }
As a particularly simple example,  consider the map $p_\X: \X \to \pt$, where $\X \simeq \colim_{i \in \I} X_i$ is as above. Then 
$$
(p_\X)_{*, \ren}
\Bigt{ 
\{\F_i \}_{i \in \I} } 
\simeq
\uscolim{i \in \I} \; (p_{X_i})_{*,\dR}(\F_i).
$$
Thus, $(p_\X)_{*, \ren}$ yields a quick definition of the Borel-Moore homology of $\X$ via the formula
\begin{equation} \label{eqn:slick HBM}
\Hren(\X) \simeq (p_\X)_{*, \ren}(\omega_{\X}).
\end{equation}
\end{rem}

\sssec{}

Consider now the map
$$
f:
G \backsl \GR /G
\longto
\GO \backsl \GK /\GO
$$
and its pullback functor $f^!: \Sph_G \to \Dmod(G \backsl \GR /G)$. According to Theorem \ref{mainthm:B as unit of Sph-temp}, 
$$
\B \simeq (f^!)^R (\omega_{G \backsl \GR /G}).
$$
Using the self-duality of $\Dmod(G \backsl \GR /G)$ and that of $\Sph_G$ (see Remark \ref{rem:Sph self-dual}), we can express $\B$ slightly differently as 
\begin{equation} \label{eqn:B as renormalized}
\B \simeq f_{*,\ren} (\omega_{G \backsl \GR /G}).
\end{equation}
Indeed, we have:

\begin{lem}
With the above notation, $f_{*,\ren}: \Dmod(G \backsl \GR /G) \to \Sph_G $ is right adjoint to $f^!$.
\end{lem}

\begin{proof}
We know that $f^! \simeq j^! \circ \oblv^{G \to \GO}$. In \cite[Sections 2-3]{thesis}, we proved that $\oblv^{G \to \GO}$ is dual to its right adjoint $\Av_*^{G \to \GO}$, while $j^!$ is easily seen to be dual to $j_{*,\dR}$.
\end{proof}

\ssec{The Hecke action of the tempered unit} \label{ssec:Hecke action of B}

Let us fix $x \in X$ throughout and consider the Hecke action of $\Sph_G$ on $\Dmod(\Bun_G)$ at $x$. We remind the reader that here $X$ is a smooth projective curve of arbitrary genus, and that $\Bun_G = \Bun_G(X)$.
In this section, we provide an explicit formula for the Hecke action of $\B$ on $\Dmod(\Bun_G)$.

\sssec{} \label{sssec:Hecke action}

Let us first recall the Hecke action of $\Sph_G$ on $\Dmod(\Bun_G)$.
Denote by $\wt\Bun_G \to \Bun_G$ the $\GO$-torsor of $G$-bundles equipped with a trivialization on the formal disc around $x$.
We regard $\wt\Bun_G$ as being acted by $\GO$ on the left, so that $\Bun_G \simeq \GO \backsl  \wt\Bun_G$.
It is well-known that $\GK$ acts on $\wt\Bun_G$ by \virg{regluing}, extending the above $\GO$-action (see, for instance, \cite[Section 2.3.4]{BD-quantization}).
The Hecke action is, by definition, the pull-push along the correspondence
$$
\GO \backslash \GK / \GO
\times
\Bun_G
\leftto
\GO \backslash \GK \times^\GO \wt\Bun_G
\xto{\act }
\GO \backslash  \wt\Bun_G 
\simeq
\Bun_G,
$$
where the pushforward along $\act$ is the $!$-pushforward. As $\act$ is ind-proper, $\act_!$ agrees with the renormalized pushforward $\act_{*,\ren}$.

\sssec{}

Now we claim that the DG category $\Dmod(G \backsl \wt\Bun_G)$ admits an action of $\Dmod(G \backsl \GR /G)$. This action is given by convolution, however we have the same issue as in Section \ref{sssec:care with the push} together with a new issue: $G \backsl \wt\Bun_G$ is of infinite type. 
To avoid these issue, and to place this action on the same footing as the Hecke action above, let us proceed using the language of categorical group actions.

\sssec{}

Let $\C$ be a DG category equipped with a left action of $\GK$. 
By formal nonsense, $\Sph_G$ coacts on ${}^\GO \C$ from the left; we denote by 
$$
\coact_{\GO \to \GK}: {}^\GO \C \to \Sph_G \otimes {}^\GO \C
$$
the coaction functor. 
Now, we exploit the self-duality of $\Sph_G$ to turn this coaction into an action. 
Explicitly, given $\S \in \Sph_G$ and $c \in {}^\GO \C$, the formula for this action is 
$$
\S \star c
\simeq
\langle
\S, 
\coact_{\GO \to \GK}(c)
\rangle, 
$$
where $\langle \cdot, \cdot \rangle$ is the evaluation functor
$$
\Sph_G \otimes (\Sph_G \otimes {}^\GO \C)
\simeq
(\Sph_G \otimes \Sph_G) \otimes {}^\GO \C
\xto{\mathsf{ev} \otimes \id}
 {}^\GO \C.
$$

\sssec{}

Let us now repeat the same construction with the pair $(\GR,G)$ in place of $(\GK, \GO)$. We obtain that, given $\C$ acted upon by $\GR$, the comonoidal DG category $\Dmod(G \backsl \GR /G)$ coacts on ${}^G \C$ and this coaction can be turned into an action. We denote the coaction by $\coact_{G \to \GR, \C}$ and the action by $\ast$. In formulas, for $\F \in \Dmod(G \backsl \GR /G)$ and $d \in {}^G \C$, we have
$$
\F \ast d
\simeq
\langle
\F, 
\coact_{G \to \GR, \C}(d)
\rangle, 
$$
where $\langle \cdot, \cdot \rangle$ is the evaluation functor
$$
\Dmod(G \backsl \GR /G)  \otimes \Bigt{ \Dmod(G \backsl \GR /G) \otimes {}^G \C }
\simeq
\Bigt{  \Dmod(G \backsl \GR /G) \otimes \Dmod(G \backsl \GR /G) }
 \otimes {}^G \C
\xto{\mathsf{ev} \otimes \id}
 {}^G \C.
$$

\begin{lem} \label{lem:formula for action of B}
Let $\C$ be a DG category equipped with a left $\GK$-action, so that (as seen above) $\Sph_G$ acts on ${}^\GO\C$ and $\Dmod(G \backsl \GR /G)$ acts on ${}^G \C$. For $c \in {}^\GO \C$, we have 
\begin{equation} \label{eqn:formula for action of B}
\B \star c 
\simeq
\Av_*^{G \to \GO} 
\bigt{
\omega_{G \backsl \GR /G}
\ast
\oblv^{G \to \GO}(c)
}.
\end{equation}
\end{lem}

\begin{proof}
We use \eqref{eqn:B as renormalized} and the definition of the $\star$ action to write
$$
\B \star c 
\simeq
\big\langle
f_{*,\ren}(\omega_{G \backsl \GR /G}), \coact_{\GO \to \GK}(c)
\big\rangle.
$$
By duality, 
$$
\B \star c 
\simeq
\big\langle
\omega_{G \backsl \GR /G}, (f^! \otimes \id) \circ \coact_{\GO \to \GK}(c)
\big\rangle,
$$
where, abusing notation, $\langle \cdot, \cdot \rangle$ denotes the evaluation functor
$$
\Dmod(G \backsl \GR /G) \otimes \Dmod(G \backsl \GR /G) \otimes {}^\GO \C
\to
{}^\GO \C.
$$
A straightforward diagram chase
shows that
$$
\oblv^{G \to \GO} \circ  (f^! \otimes \id) \circ \coact_{\GO \to \GK}
\simeq
\coact_{G \to \GR} \circ \oblv^{G \to \GO}.
$$
Since $\oblv^{G \to \GO}$ is fully faithful (as a consequence of the pro-unipotence of $\ker(\GO \tto G)$), we obtain that 
$$
(f^! \otimes \id) \circ \coact_{\GO \to \GK}
\simeq
\Av_*^{G \to \GO} \circ \coact_{G \to \GR} \circ \oblv^{G \to \GO}.
$$
It follows that 
$$
\B \star c 
\simeq
\Av_*^{G \to \GO} 
\Bigt{
\langle
\omega_{G \backsl \GR /G}, \coact_{G \to \GR} \circ \oblv^{G \to \GO}(c)
\rangle
},
$$
which is isomorphic to $\Av_*^{G \to \GO} 
\bigt{
\omega_{G \backsl \GR /G}
\ast
\oblv^{G \to \GO}(c)
}$ as desired.
\end{proof}

\sssec{}

Now let $\C = \Dmod^!(\wt\Bun_G)$, where $\wt\Bun_G$ is defined as in Section \ref{sssec:Hecke action} and $\Dmod^!$ as in \cite{Sam, thesis}. The left $\GK$-action on $\wt\Bun_G$ yields a left $\GK$-action on $\C$. Unraveling the definition, the resulting $\Sph_G$-action on ${}^\GO \Dmod^!(\wt\Bun_G) \simeq \Dmod(\Bun_G)$ is precisely the Hecke action at $x$. Hence, the above lemma yields a formula for the Hecke action of the tempered unit.

\ssec{Deducing Theorem \ref{mainthm: omega antitemp}}

\sssec{}

After the discussion of Section \ref{sssec:intro-reformulate-thmA}, our strategy to prove Theorem \ref{mainthm: omega antitemp} amounts to showing that 
$$
\B \star\omega_{\Bun_G} \simeq 0 \in \Dmod(\Bun_G)
$$
whenever $G$ has  semisimple rank $\geq 1$.
In view of  \eqref{eqn:formula for action of B}, this amounts to showing that
$$
\Av_*^{G \to \GO} 
\bigt{
\omega_{G \backsl G(R) /G}
\ast
\omega_{G \backsl \wt \Bun_G}
 }
 \simeq
 0.
$$
In fact, we will prove that 
$$
\omega_{G \backsl G(R) /G}
\ast
\omega_{G \backsl \wt \Bun_G}
\in \Dmod(G \backsl \wt \Bun_G)
$$
is already the zero object. 

\sssec{}

Consider the diagram
$$
G \backsl \wt\Bun_G
\xleftarrow{\;\; \alpha \;\;}
G \backsl \GR \times^G \wt\Bun_G 
\xto{\;\; \pi \;\;}
G \backsl \wt\Bun_G,
$$
where $\alpha$ is the map induced by the $\GR$-action on $\wt\Bun_G $, and $\pi$
the projection onto the second component.
The functor
$$
(p_{G\backsl \GR})_{*,\ren} \otimes \id:
\Dmod(G\backsl \GR) \otimes \Dmod^!(\wt\Bun_G )
\longto
\Dmod^!(\wt\Bun_G )
$$
is equivariant for the obvious \virg{balanced} $G$-action on the source and the left $G$-action on the target. Abusing notation, we denote by $\pi_{*,\ren}$ the induced functor at the level of $G$-invariant categories. 
Unraveling the constructions, we have
$$
\omega_{G \backsl G(R) /G}
\ast
\omega_{G \backsl \wt \Bun_G}
\simeq
\pi_{*,\ren}
\circ \alpha^!
(\omega_{G \backsl \wt \Bun_G})
\simeq
\pi_{*,\ren} 
\bigt{
\omega_{G \backsl \GR \times^G \wt\Bun_G} 
}.
$$

\sssec{}

We need to show that the latter object is zero. It is enough to do so after applying the conservative functor $\oblv^G: \Dmod(G \backsl \wt\Bun_G) \to \Dmod(\wt\Bun_G)$.
But then we tautologically have
$$
\oblv^G \circ \pi_{*,\ren} 
\bigt{
\omega_{G \backsl \GR \times^G \wt\Bun_G} 
}
\simeq
(p_{G \backsl \GR})_{*,\ren}  \bigt{ \omega_{G \backsl \GR} }
\otimes \omega_{\wt \Bun_G},
$$
so it suffices to prove that the vector space
$$
(p_{G \backsl \GR})_{*,\ren}  \bigt{ \omega_{G \backsl \GR} }
$$ 
vanishes. 

\sssec{}

First off, this vector space is exactly $\Hren(G \backsl \GR)$, by the formula for Borel-Moore homology of Remark \ref{rem:Hren def slick via renorm push }.
Then the isomorphims
$$
G \backsl G(R)\simeq G(R)^1, 
\hspace{.5cm}
G(R) \simeq G \ltimes G(R)^1
$$
imply that
$$
\Hren(G \backsl \GR) \otimes \Hren(G) \simeq \Hren(\GR).
$$
Tensoring with a nonzero vector space, in our case $\Hren(G) \simeq H^*(G)[2 \dim(G)]$, is conservative. So it remains to prove that $\Hren(\GR) \simeq 0$:
this vanishing statement is exactly the content of Theorem \ref{mainthm:Hren vanishing} for the affine curve $\A^1$.

\sec{Proof of Theorem \ref{mainthm:estimate with equations}} \label{sec:prove ThmE}

In the previous part of the paper, we have shown that Theorem \ref{mainthm: omega antitemp} follows from the simplest case of Theorem \ref{mainthm:Hren vanishing}.
In this section, we deduce Theorem \ref{mainthm:Hren vanishing} from Theorem \ref{mainthm:omega-inf-connective}, and then prove the latter in some special cases.

\ssec{From Theorem \ref{mainthm:omega-inf-connective} to Theorem \ref{mainthm:Hren vanishing}}

It suffices to apply the following general result to the case of $\Y = G[\Sigma]$.

\begin{lem}
Let $\Y$ be an ind-affine indscheme of ind-finite type. If $\omega_\Y \in \Dmod(\Y)^{\leq -\infty}$, then $\Hren(\Y) = 0$.
\end{lem}

\begin{proof}
Let $\C$ be a DG category equipped with a t-structure. The t-structure is said to be left-complete if the functor
$$
\C 
\longto
 \lim_{n \in (\bbZ, \leq)^\op} \C^{\geq -n},
\hspace{.4cm}
c \squigto 
\{\tau^{\geq -n}(c)\}_{n \in \bbZ}
$$
is an equivalence.
When the t-structure on $\C$ is left-complete, it is obvious that $\C^{\leq -\infty} \simeq 0$: in other words, there are no nontrivial infinitely connective objects.
It is easy to see that the usual t-structure on $\Vect$ is left-complete, so that, in particular, $\Vect^{\leq -\infty} \simeq 0$. 
Hence, as 
$$
\Hren(\Y) := (p_\Y)_{*,\ren}(\omega_\Y)
$$
by Remark \ref{rem:Hren def slick via renorm push }, it suffices to show that $(p_\Y)_{*,\ren}: \Dmod(\Y) \to \Vect$ is right t-exact.

Let $\Y \simeq \colim_{k \in K} Y_k$ be an indscheme presentation, with each $Y_k$ an affine scheme. As discussed in Sections \ref{sssec:ind-oblv-D-modules} and \ref{defn t-structure}, the natural t-structure on $\Dmod(\Y)$ is defined by requiring that $\Dmod(\Y)^{\leq 0}$ be generated under colimits by the objects of the form $(i_k)_{*,\dR}(\ind_{Y_k}(C))$, for all $k$ and all $C \in \Coh(Y_k)^{\leq 0}$.  

Thus, it suffices to prove that
\begin{equation} \label{eqn:connective object}
(p_\Y)_{*,\ren}
\Bigt{
(i_k)_{*,\dR}(\ind_{Y_k}
(C)
)
}
\in \Vect^{\leq 0}
\end{equation}
for any such $C$. 
Let us simplify the composition 
\begin{equation} \label{eqn:string}
\IndCoh(Y_k)
\xto{\ind_{Y_k}}
\Dmod(Y_k)
\xto{(i_k)_{*,\dR}}
\Dmod(\Y)
\xto{
(p_\Y)_{*,\ren}}
\Vect.
\end{equation}
Observe first that $(p_\Y)_{*,\ren} \circ
(i_k)_{*,\dR} \simeq (p_{Y_k})_{*,\dR}$: indeed, the dual to $(p_\Y)_{*,\ren} \circ
(i_k)_{*,\dR} $ is $
(i_k)^! \circ (p_\Y)^! \simeq (p_{Y_k})^!$ by the definition of the renormalized pushforward.
It follows that the functor \eqref{eqn:string} is the ind-coherent pushforward along $p_{Y_k}$; furthermore, as $C$ is coherent, we obtain that
$$
(p_\Y)_{*,\ren}
\Bigt{
(i_k)_{*,\dR}(\ind_{Y_k}
(C))
}
\simeq
(p_{Y_k})_*^\ICoh(C)
\simeq
\Psi_\pt \circ (p_{Y_k})_*^\ICoh(C)
\simeq
(p_{Y_k})_*( \Psi_{Y_k} C)
\simeq
(p_{Y_k})_*(C),
$$ 
where we have used the fact that $\Psi$ intertwines quasi-coherent and ind-coherent pushforwards.
Then the claim of \eqref{eqn:connective object} is evident: thanks to the affineness of $Y_k$, the functor $(p_{Y_k})_*: \Coh(Y_k) \to \Vect$ is t-exact.
\end{proof}

\ssec{Proof of Theorem \ref{mainthm:estimate with equations}}

Let $G=SL_n$. In this case, $G$ is a closed subscheme of $\A^{n^2}$ determined by the vanishing of one equation of degree $n$. Thus, Theorem \ref{mainthm:omega-inf-connective} for $G=SL_n$ is a simple case of the following general result, which went under the name of Theorem \ref{mainthm:estimate with equations} in the introduction.

\begin{thm} \label{thm: omega inf connective by equations bound}
Let $Y \subseteq \A^N$ be a closed subscheme defined as the zero locus of $r$ polynomials $f_1, \ldots, f_r$ of degrees $n_1, \ldots, n_r$. If $n:=\sum_i n_i < N$, then $\omega_{Y[\Sigma]}$ is infinitely connective.
\end{thm}

\begin{proof}
Let $X$ be the smooth compactification of $\Sigma$, obtained by adding $h \geq 1$ points at infinity. Denote by $D_\infty = X-\Sigma$ the union of such points and by $g$ the genus of $X$.

\sssec*{Step 1}

Consider the following indscheme presentation
$$
\A^N[\Sigma]
\simeq
\uscolim{d \gg 0} \, 
\A^N[\Sigma]_{\leq d},
$$
where we have set
$$
\A^N[\Sigma]_{\leq d} := H^0(X, \O(d D_\infty))^{\oplus N}.
$$
The given closed embedding $Y \subseteq \A^N$ yields the indscheme presentation $Y[\Sigma] \simeq \colim_{d \gg 0 } Y_d$, where 
$$
Y_d := Y[\Sigma]  \cap \A^N[\Sigma]_{\leq d}.
$$
Denoting by $i_d: Y_d \to Y[\Sigma]$ the tautological closed embeddings, we deduce that
$$
\omega_{Y[\Sigma]}
\simeq 
\uscolim{d \gg 0 } \; (i_d)_{*,\dR}(\omega_{Y_d}).
$$
Since each pushforward $(i_d)_{*,\dR}$ is right t-exact by construction, it suffices to find a divergent sequence $(C_d)_{d \gg 0 }$ of natural numbers satisfying $\omega_{Y_d} \in \Dmod(Y_d)^{\leq -C_d}$.

\sssec*{Step 2}

Assume from now on that $d >\max(0,(2g-2)/h)$. In this case, the scheme $\A^N[\Sigma]_{\leq d}$ is a vector space of dimension $N(dh+1-g)$.
Let us compute the number of equations needed to specify $Y_d$ inside $\A^N[\Sigma]_{\leq d} \simeq \A^{N(dh+1-g)}$. 
An $N$-tuple $(p_i)$ of elements of $H^0(X, \O(d D_\infty))$ belongs to $Y_d$ iff $f_j(p_1, \ldots, p_N) = 0$ for each $1 \leq j \leq r$. Since $f_j$ is of degree $n_j$, the expression $f_j(p_1, \ldots, p_N)$ is an element of $H^0(X, \O(d n_j D_\infty )) \simeq \A^{dh n_j +1 -g}$.
It follows that $Y_d$ is cut out by 
$$
\sum_{j=1}^r  (dh n_j +1 -g)
=
dhn +r(1-g)
$$
equations inside $\A^{N(dh+1-g)}$. 

\sssec*{Step 3}

Now let
$$
C_d
:=
N(dh+1-g) 
-(dhn +r(1-g))
=(N-n)hd + (N-r)(1-g).
$$
Since $N-n >0$ and $h >0$ by assumption, $C_d$ goes to infinity with $d$. On the other hand, the general lemma below implies that $\omega_{Y_d} \in \Dmod(Y_d)^{\leq - C_d}$.
\end{proof}

\begin{lem} \label{lem:silly estimate}
If $Z$ is an affine scheme of the form $\A^m \times_{\A^p} \pt$, then $\omega_Z \in \Dmod(Z)^{\leq -m+p}$.
\end{lem}

\begin{proof}
We proceed by induction on $p$. In case $p =0$, we have
$$
\omega_{\A^m} \in \Dmod(\A^m)^{\heartsuit}[m].
$$ 
For the sake of completeness, let us give a proof of this well-known statement. By \cite[Proposition 4.2.11]{Crystals}, the forgetful functor $\oblv_{\A^m}: \Dmod(\A^m) \to \IndCoh(\A^m)$ is t-exact. Since it is also conservative, it suffices to prove that the ind-coherent dualizing sheaf $\omega_{\A^m}^{\ICoh}$ belongs to $\ICoh(\A^m)^\heartsuit [m]$. The equivalence $\Psi_{\A^m}: \ICoh({\A^m}) \to \QCoh({\A^m})$ is t-exact by construction, see \cite[Section 1.2]{ICoh}, and sends $\omega_{\A^m}^\ICoh$ to the shifted canonical bundle $K_{\A^m} [m]$. 
To see the latter fact, use Grothendieck duality for $\bbP^n$ to prove that $(p_{\bbP^n})^{!, \ICoh}(\kk) \simeq K_{\bbP^n}[n]$, and then further pullback along $\A^n \hto \bbP^n$.

Assume now that $p >0$. We can write $Z \simeq Z' \times_{\A^1} \pt$, for an affine scheme $Z'$ of the form $\A^m \times_{\A^{p-1}} \pt$. Hence, $\omega_{Z'} \in \Dmod(Z')^{\leq -m+p-1}$ by the induction hypothesis.
Denote by $i: Z \hto Z'$ the obvious closed embedding, with $j: U := Z'-Z \hto Z'$ as complementary open. Since $i_{*,\dR}$ is t-exact and fully faithful, it suffices to prove that
$$
i_{*,\dR}(\omega_Z) 
\in \Dmod(Z')^{\leq -m +p}.
$$
Observe that $i_{*,\dR}(\omega_Z)$ sits in the fiber sequence
$$
j_{*,\dR}(\omega_U)[-1]
\longto
i_{*,\dR}(\omega_Z)
\longto
\omega_{Z'} .
$$
Moreover, $j^!$ and $j_{*,\dR}$ are both t-exact (the latter because $j$ is affine), so that 
$$
j_{*,\dR}(\omega_U)[-1] 
\simeq
j_{*,\dR} \circ j^! (\omega_{Z'})[-1] 
\in
\Dmod(Z')^{\leq - m+p}.
$$
Combined with the bound for $\omega_{Z'}$, this yields the assertion.
\end{proof}

\sec{Proof of Theorem \ref{mainthm:omega-inf-connective} in general} \label{sec: prove ThmD in general}

In this section we prove Theorem \ref{mainthm:omega-inf-connective}: given a non-abelian connected reductive group $G$ and a smooth affine curve $\Sigma$, the dualizing sheaf $\omega_{G[\Sigma]}$ is infinitely connective for the natural t-structure on $\Dmod(G[\Sigma])$.

\ssec{Outline} \label{ssec:outline proof of thm D}

Since the proof involves a number of technical reduction steps, let us give an outline of the strategy.

\sssec{}

Any closed embedding $Z \hto Y$ of affine schemes induces a closed embedding $Z[\Sigma] \hto Y[\Sigma]$ of indschemes. The open complement is the indscheme $Y[\Sigma]^{(Y-Z)\ggen}$ of maps $\Sigma \to Y$ that land generically in $Y-Z$. 
Precisely, the functor of points of $Y[\Sigma]^{(Y-Z)\ggen}$ sends a test affine scheme $S$ to the set of maps $\phi:\Sigma_S \to Y$ for which the preimage $\phi^{-1}(Y-Z)$ is universally dense in $\Sigma_S$.

Recall that an open subset of $\Sigma_S$ is said to be universally dense if, for every geometric point $s \in S$, its pullback along $\Sigma_s \to \Sigma_S$ is dense in $\Sigma_s$.
See \cite{Barlev} for several examples of prestacks defined using the notion of universal density.

\sssec{}

Denote by 
$$
\Gcirc := N^- \times  B \simeq N^- \times T \times N
$$
the big Bruhat cell of $G$: this is a Zariski open subset of $G$. Since $G$ can be covered by translates of $\Gcirc$,  it is clear that $G[\Sigma]$ admits an open cover whose members are isomorphic to $G[\Sigma]^{\Gcirc \ggen}$. 
By \cite[Lemma 7.8.7]{ker-adj}, the t-structure on the DG category of D-modules on an indscheme is Zariski local: hence, to show that $\omega_{G[\Sigma]}$ is infinitely connective, it suffices to prove the following.

\begin{thm} \label{thm:maps to G generically in big cell}
Let $G$ be a connected reductive group of semisimple rank $\geq 1$. Then the dualizing sheaf of the indscheme $G[\Sigma]^{G^\circ \ggen}$ is infinitely connective.
\end{thm}

\sssec{}

Before proceeding, it is convenient to introduce some more notation. Let $\fSet$ be the $1$-category of \emph{nonempty} finite sets and surjective maps between them. Given $I \in \fSet$ and given $\x \in \Sigma^I(S)$ an $I$-tuple of maps from $S$ to $\Sigma$, we denote by $D_\x \subset \Sigma_S$ the incidence divisor of $\x$.

\sssec{}

For the proof of Theorem \ref{thm:maps to G generically in big cell}, we need a variant of $G[\Sigma]^{G^\circ \ggen}$ that keeps track of the locus where the rational map $\Sigma \dasharrow G^\circ$ is not defined. This is precisely what $\fSet$ and the above notation are useful for.
Namely, let $G[\Sigma]^{\Gcirc\ggen}_{\Sigma^{I,\disj}}$ be the indscheme whose $S$-points are those pairs 
$$
\bigt{
\x \in \Sigma^I(S), \, \phi: \Sigma_S \to G
}
$$
with the following two properties:
\begin{itemize}
\item
the members of the $I$-tuple $\x$ have pairwise disjoint graphs in $\Sigma_S$;
\item
$\phi$ sends $\Sigma_S- D_\x$ to $\Gcirc$.
\end{itemize}
In Section \ref{ssec:Ran marked}, we will show that the Theorem \ref{thm:maps to G generically in big cell} is a consequence of the statement below.

\begin{prop} \label{prop:maps to G, gen in Gcirc, over disjoint locus}
For any $I \in \fset$, the dualizing sheaf of $
G[\Sigma]^{G^\circ\ggen}_{\Sigma^{I,\disj}}$ is infinitely connective.
\end{prop}

\sssec{}

It remains to prove this proposition. Consider the particular case where $I$ is a singleton. Then the indscheme in question is $G[\Sigma]^{\Gcirc \ggen}_{\Sigma}$: its $S$-points are pairs $(\x \in \Sigma(S), \phi: \Sigma_S \to G)$ such that the restriction of $\phi$ to $\Sigma_S - D_\x$ factors through $\Gcirc$. 
Even more particularly, fix $x \in \Sigma(\kk)$ and consider the indscheme
$$
G[\Sigma]^{\Gcirc \ggen}_{\Sigma} \times_{\Sigma} \{x\}
\simeq
\Gcirc[\Sigma-x] \times_{G[\Sigma-x]} G[\Sigma].
$$
This indscheme parametrizes those maps $\Sigma \to G$ that send $\Sigma-x$ into the big cell.
We will approach Proposition \ref{prop:maps to G, gen in Gcirc, over disjoint locus} in two stages. In Section \ref{ssec:easier omega inf connective}, we will prove:

\begin{prop} \label{prop:maps to G, gen in Gcirc, over disjoint locus - EASIER}
For any $x \in \Sigma(\kk)$, the dualizing sheaf of $\Gcirc[\Sigma-x] \times_{G[\Sigma-x]} G[\Sigma]$ is infinitely connective.
\end{prop}

This is where the main geometric argument takes place. Then, in Section \ref{ssec:moving points}, we will adapt this argument to prove Proposition \ref{prop:maps to G, gen in Gcirc, over disjoint locus} in its full generality.

\ssec{Proof of Proposition \ref{prop:maps to G, gen in Gcirc, over disjoint locus - EASIER}} \label{ssec:easier omega inf connective}

Let $x \in \Sigma(\kk)$ be fixed once and for all. Let $t$ be a local coordinate at $x$ and denote by $\Gr_G = \Gr_{G,x}$ the affine Grassmannian at $x$. 
We have
$$
\Gcirc[\Sigma-x] \times_{G[\Sigma-x]} G[\Sigma]
\simeq
\Gcirc[\Sigma-x] \times_{G\ppart} G[[t]]
\simeq
\Gcirc[\Sigma-x] \times_{\Gr_G} \pt,
$$
where the map $\pt \to \Gr_G$ in the rightmost expression is the inclusion of the unit point.

\sssec{}

By definition, $\Gcirc \simeq N^- \times T \times N$, so that 
$$
\Gcirc[\Sigma-x]
\simeq
N^-[\Sigma-x]
\times
T[\Sigma-x]
\times
N[\Sigma-x],
$$
$$
\Gcirc \ppart
\simeq
N^- \ppart
\times
T \ppart
\times
N \ppart.
$$
Consider the maps $T[\Sigma-x] \to \Gr_{T,x}$ and $T \ppart \to \Gr_{T,x}$.
Since, at the reduced level, the $T$-Grassmannian $\Gr_{T,x}$ is discrete and isomorphic to $\Lambda$, elements of $T[\Sigma-x]$ and $T \ppart$ have a well defined \virg{type} in $\Lambda$. For $\lambda \in \Lambda$, denote by 
$$
T[\Sigma-x]^{\lambda}
:=
T[\Sigma-x] 
\ustimes{\Gr_T}
\{t^\lambda\},
\hspace{.4cm}
T \ppart ^{\lambda}
:=
T \ppart 
\ustimes{\Gr_T}
\{t^\lambda\},
$$
the corresponding closed (as well as open) subschemes of $T[\Sigma-x]$ and $T \ppart$, respectively.
We also set 
\begin{equation} \label{eqn:lambda type for maps into Gcirc}
\Gcirc[\Sigma-x]^{\lambda} := N^-[\Sigma-x]
\times
T[\Sigma-x]^{\lambda}
\times
N[\Sigma-x],
\end{equation}
\begin{equation} \label{eqn:lambda type loops into Gcirc}
\Gcirc \ppart ^{\lambda} 
:= 
N^- \ppart
\times
T \ppart ^{\lambda}
\times
N \ppart.
\end{equation}
Since $\Gcirc[\Sigma-x] \simeq \bigsqcup_{\lambda} \Gcirc[\Sigma-x]^{\lambda} $,
 the following implies (in fact, it is equivalent to) Proposition \ref{prop:maps to G, gen in Gcirc, over disjoint locus - EASIER}.

\begin{prop} \label{prop:maps to G, gen in Gcirc, over disjoint locus - LAMBDA}
For any $\lambda \in \Lambda$, the dualizing sheaf of $
\Gcirc[\Sigma-x]^{\lambda} \times_{\Gr_{G}} \pt$ is infinitely connective.
\end{prop}

\sssec{}

The remaining part of Section \ref{ssec:easier omega inf connective} is devoted to the proof of the above proposition. We first observe that
\begin{equation} \label{eqn:Lambda loops into Gcirc extendable across x}
\Gcirc[\Sigma-x]^{\lambda}
 \ustimes{\Gr_G} 
 \pt
\simeq
\Gcirc[\Sigma-x]^{\lambda}
\ustimes{\Gcirc \ppart^{\lambda}} 
\Bigt{
\Gcirc \ppart^\lambda 
\ustimes{G\ppart} 
G[[t]]
}.
\end{equation}
Now, let us represent elements of $\Gcirc[\Sigma-x]^{\lambda}$ as triples $\phi = (\phi^-, \phi^T, \phi^+)$ according to the isomorphism of \eqref{eqn:lambda type for maps into Gcirc}.
Obviously, $\phi^\pm = (\phi^-,\phi^+)$ can be viewed as a map $\Sigma-x \to \A^{|\sR|}$, where $|\sR|$ is the number of roots of $G$.
We say that $\phi^\pm$ has poles at $x$ bounded by $n$ (with $n \geq 0$) if each of the maps $\Sigma-x \to \A^1$ comprising $\phi^\pm$ is an element of $H^0(\Sigma, \O_\Sigma(nx))$.
We remark that this condition is about the pole at $x$, it has nothing to do with the poles of $\phi^{\pm}$ at the points at infinity of $\Sigma$.

\sssec{}

We will use a similar notation and terminology for elements of $\Gcirc \ppart ^\lambda$. Namely, we represent them as triples $\phi = (\phi^-, \phi^T, \phi^+)$ according to \eqref{eqn:lambda type loops into Gcirc} and say that $\phi$ has poles bounded by $n$ if so do all the Laurent series comprising $\phi^\pm$.

\begin{lem} \label{lemma:key-rep-thry}
For any $\lambda \in \Lambda$, there exists a number $e(\lambda) \in \NN$ with the following property: if $\phi = (\phi^-, \phi^T, \phi^+) \in \Gcirc \ppart^\lambda$ is contained in 
$\Gcirc \ppart^\lambda 
\times_{G\ppart} G[[t]]$, then $\phi^\pm$ has poles at $x$ bounded by $e(\lambda)$.
\end{lem}

\begin{proof}
For $G= GL_2$, this lemma is obvious. 
Indeed, letting $R$ be a test ring, choose $f,g \in R\ppart$ and $m,n \in \ZZ$ arbitrarily. If the element
$$
\left[
\begin{array}{cc}
1 & 0  \\
f & 1
\end{array}
\right]
\cdot
\left[
\begin{array}{cc}
t^m& 0  \\
0 & t^{n}
\end{array}
\right]
\cdot
\left[
\begin{array}{cc}
1 & g  \\
0 & 1
\end{array}
\right] = 
\left[
\begin{array}{cc}
t^m & g t^m   \\
f t^m & t^{n} + fg t^m
\end{array}
\right]
$$
belongs to $G(R[[t]])$, then $m \geq 0$ and the order of the poles of $f$ and $g$ must be bounded by $m$.

A similar computation, left to the reader, proves the lemma for $G=GL_d$, with $d \geq 3$. Now let $G$ be arbitrary; we will reduce to the case of $GL_d$ as follows. Pick a faithful representation $\rho: G \hto GL(V)$. By choosing an appropriate ordered basis of $V$ consisting of weight vectors, we can assume that $\rho(N)$ (respectively: $\rho(N^-)$, $\rho(T)$) is contained in the subgroup of upper triangular (respectively: lower triangular, diagonal) matrices of $GL(V) \simeq GL_{\dim(V)}$. 
In particular, $\rho$ sends the big cell of $G$ to the big cell of $GL_{\dim(V)}$. 

Now pick $\phi = (\phi^-, \phi^T, \phi^+) \in \Gcirc \ppart^\lambda 
\times_{G\ppart} G[[t]]$. The case of $GL_d$ implies that $\rho(\phi^{\pm})$ both have poles at $x$ bounded by some $e \in \bbN$. It follows easily that the same is true, possibly with a different bound $e' \in \bbN$, for the $\phi^\pm$ themselves.
\end{proof}


\sssec{}

Denote by $\Gcirc \ppart ^{\lambda,\leq n}$ the closed subspace of $\Gcirc \ppart^{\lambda}$ consisting of those $\phi$ for which the poles of $ \phi^\pm$ are bounded by $n$. 
The subspace $\Gcirc[\Sigma-x] ^{\lambda,\leq n}$ of $\Gcirc[\Sigma-x] ^{\lambda}$ is defined in the same way.
The above lemma implies that the closed embedding
$$
\Gcirc \ppart^{\lambda,\leq e(\lambda)}
\ustimes{G\ppart} 
G[[t]]
\hto
\Gcirc \ppart^\lambda 
\ustimes{G\ppart} 
G[[t]]
$$
is an isomorphism.

\sssec{}

On the other hand, the bound on the pole orders makes it obvious that the map
$$
\Gcirc \ppart^{\lambda,\leq e(\lambda)}
\to
G \ppart
\tto
\Gr_G
$$
factors through a closed \emph{subscheme} $Y_\lambda \subset \Gr_G$. 
Then:
$$
\Gcirc \ppart^{\lambda,\leq e(\lambda)}
\ustimes{G\ppart} 
G[[t]]
\simeq
\Gcirc \ppart^{\lambda,\leq e(\lambda)}
\ustimes{\Gr_G} 
\pt
\simeq
\Gcirc \ppart^{\lambda,\leq e(\lambda)}
\ustimes{Y_\lambda}
\bigt{
Y_\lambda
\ustimes{\Gr_G} 
\pt
}
\simeq\Gcirc \ppart^{\lambda,\leq e(\lambda)}
\ustimes{Y_\lambda} 
\pt.
$$

\sssec{}

Combining these two observations, we obtain a natural isomorphism
$$
\Gcirc \ppart^\lambda 
\ustimes{G\ppart} 
G[[t]]
\simeq
\Gcirc \ppart^{\lambda,\leq e(\lambda)}
\ustimes{Y_\lambda}
\pt,
$$
and thus, in view of \eqref{eqn:Lambda loops into Gcirc extendable across x}, a natural isomorphism
$$
\Gcirc[\Sigma-x]^{\lambda}
 \ustimes{\Gr_G} 
 \pt
\simeq
\Gcirc[\Sigma-x]^{\lambda}
\ustimes{\Gcirc \ppart^{\lambda}} 
\Bigt{
\Gcirc \ppart^{\lambda,\leq e(\lambda)}
\ustimes{Y_\lambda}
\pt
}
\simeq
\Gcirc[\Sigma-x]^{\lambda, \leq e(\lambda)}
\ustimes{Y_\lambda}
 \pt.
$$
Hence, to conclude the proof of Proposition \ref{prop:maps to G, gen in Gcirc, over disjoint locus - LAMBDA}, it remains to prove that the dualizing sheaf of the indscheme
$$
\Gcirc[\Sigma-x]^{\lambda, \leq e(\lambda)}
\ustimes{Y_\lambda}
 \pt
$$
is infinitely connective.

\sssec{}

Next, note that 
$$
\Gcirc[\Sigma-x]^{\lambda, \leq e(\lambda)}
\simeq
\A^{|\sR| }[\Sigma-x]^{\leq e(\lambda)}
\times
T[\Sigma-x]^{\lambda}.
$$
As usual, let $X$ be the smooth compactification of $\Sigma$ and $D_\infty$ the divisor at infinity, of cardinality $h \geq 1$.
We have:
$$
\A^{|\sR|}[\Sigma-x]^{\leq e(\lambda)}
\simeq
\colim_{m \gg 0} \,
H^0(X, e(\lambda) x + m D_\infty  )^{|\sR|}
\simeq
\A^\infty.
$$
In the same way, $T[\Sigma-x]^{\lambda}$ is an indscheme of ind-finite type, again realized as a colimit over the poset of natural numbers.
It remains to apply the following general result.

\begin{lem} \label{lem:Ainfty times something inf connective}
Let $\Y$ an indscheme of the form 
$$
\Y
=(\A^\infty \times \T) \times_Y \pt,
$$
for some scheme $Y$ of finite type and some indscheme $\T = \colim_{m \gg 0} T_m$ of ind-finite type. Then the dualizing sheaf of $\Y$ is infinitely connective.
\end{lem}

\begin{proof}
We first show that $\omega_{\A^\infty \times \T}$ is infinitely connective. 
By assumption, each $T_m$ is a scheme of finite type. It follows that there exists some $N_m  \geq 0$ such that $\omega_{T_m} \in \Dmod(T_m)^{\leq N_m}$:  to prove this, recall that the t-structure is Zariski local, argue as in Lemma \ref{lem:silly estimate} and then use the quasi-compactness of $T_m$.
Up to replacing each $N_m$ with a larger number, let us assume that the sequence $m \mapsto N_m$ is increasing and divergent.
Then we have the indscheme presentation
$$
\A^\infty \times \T
\simeq
\colim_{m \gg 0}
\bigt{ \A^{2 N_m} \times T_m },
$$
so that
$$
\omega_{\A^\infty \times \T}
\simeq
\colim_{m \gg 0} \,
(i_m)_{*,\dR} (\omega_{\A^{2N_m} \times T_m}),
$$
where $i_m: \A^{2N_m} \times T_m \hto \A^\infty \times \T$ is the natural closed embedding.
Now recall the equivalence
$$
\Dmod(\A^{2N_m} \times T_m) \simeq
\Dmod(\A^{2N_m}) \otimes \Dmod(T_m)
$$
induced by exterior tensor product, which is valid because $\Dmod(\A^{2N_m})$ is dualizable. The t-structure on the LHS corresponds to the tensor product t-structure on the RHS (the latter is defined by declaring that connective objects of the tensor product are generated under colimits by tensor products of connective objects).
Since the dualizing sheaf $\omega_{\A^{2N_m} \times T_m}$ corresponds to $\omega_{\A^{2N_m}} \boxtimes \omega_{T_m}$ and $\omega_{\A^{2N_m}} \in \Dmod(\A^{2N_m})^{\leq - 2 N_m}$, we obtain that 
$$
\omega_{\A^{2N_m} \times T_m}
\in \Dmod({\A^{2N_m} \times T_m})^{\leq - N_m}.
$$
Recall that each $(i_m)_{*,\dR}$ is right t-exact by construction; then, as $N_m$ goes to infinity with $m$, we have proven that $\omega_{\A^\infty \times \T}$ is infinitely connective.

Now let $i: \Y \hto \A^\infty \times \T$ be the obvious closed embedding. 
In view of $\omega_\Y \simeq i^! (\omega_{\A^\infty \times \T})$, it suffices to show that $i^!$ is right t-exact up to a finite shift. This boils down to proving that, for any scheme $Z$ of finite type mapping to $Y$, the $!$-pullback along $Z \times_Y \pt \hto Z$ is right t-exact up to a shift that is independent of $Z$. This follows exactly as in Lemma \ref{lem:silly estimate}: first by the Zariski-local nature of the t-structure, we can replace $Y$ with an affine open $Y'$ containing $\pt$; then we write $\pt \in Y'$ as the zero locus of finitely many equations in $Y'$.
\end{proof}

\ssec{Moving points} \label{ssec:moving points}

Now we adapt the argument of Section \ref{ssec:easier omega inf connective} to prove Proposition \ref{prop:maps to G, gen in Gcirc, over disjoint locus}. The idea is the same, but the notation heavier.

\sssec{}

For $Y$ a scheme among $G,\Gcirc, T, N, N^-$, denote by $Y[\Sigma]^{\rat}_{\Sigma^{I,\disj}}$ the indscheme with $S$-points  given by the set of tuples
$$
(\x \in \Sigma^I(S), \phi: \Sigma_S - D_\x \to Y)
$$
such that the elements of $\x$ have pairwise disjoint graphs in $\Sigma_S$.

\begin{rem}

To see that $Y[\Sigma]^{\rat}_{\Sigma^{I,\disj}}$ is indeed an indscheme, it suffices to treat the case of $Y \simeq \A^1$. In this case, the discussion of \cite[Section 2.7]{contract} applies; in particular, we have
$$
\A^1[\Sigma]^{\rat}_{\Sigma^{I,\disj}}
\simeq
\colim_{d \geq 0}
\A^1[\Sigma]^{\rat, \leq d}_{\Sigma^{I,\disj}},
$$
with each $\A^1[\Sigma]^{\rat, \leq d}_{\Sigma^{I,\disj}}$ also the indscheme of rational maps $\Sigma \dasharrow \A^1$ with poles bounded by $d$ of $\Sigma$ (and unbounded at  points at infinity of $\Sigma$).
\end{rem}

\sssec{}

Denote by $Y(\A)_{\Sigma^{I,\disj}}$ the space of meromorphic jets into $Y$ parametrized by $|I|$ disjoint points of $\Sigma$. Precisely, the set of $S$-points of $Y(\A)^{\Sigma_{I,\disj}}$ is given by
$$
(\x \in \Sigma^I(S), \phi: \wh D_\x^\circ \to Y),
$$ 
where $\wh D_\x^\circ$ is the punctured tubular neighbourhood of $D_x$.
Define $Y(\OO)_{\Sigma^{I,\disj}}$ in the same way as above, but using the non-punctured tubular neighbourhood of $D_\x$.

\sssec{}

By construction, our indscheme of interest $
G[\Sigma]^{G^\circ\ggen}_{\Sigma^{I,\disj}}$ is isomorphic to
$$
\Gcirc[\Sigma]^{\rat}_{\Sigma^{I,\disj}}
\ustimes
{
\GA_{\Sigma^{I,\disj}}
}
\GO_{\Sigma^{I,\disj}}.
$$
We will use this expression as a fiber product to prove that the dualizing of $G[\Sigma]^{G^\circ\ggen}_{\Sigma^{I,\disj}}$ is infinitely connective.

\sssec{}

As in the previous case, elements of $T[\Sigma]^{\rat}_{\Sigma^{I,\disj}}$ have a \virg{type}, which is now an $I$-tuple of elements of $\Lambda$.
For each such $\ul\lambda \in \Lambda^I$, we denote by $T[\Sigma]^{\rat, \ul\lambda}_{\Sigma^{I,\disj}}$, $\Gcirc[\Sigma]^{\rat, \ul\lambda}_{\Sigma^{I,\disj}}$ and $T(\A)_{\Sigma^{I,\disj}}^{\ul \lambda}$ the corresponding subspaces.

It suffices to prove that the dualizing sheaf of 
\begin{equation} \label{eqn:big auxiliary}
\Gcirc[\Sigma]^{\rat, \ul\lambda}_{\Sigma^{I,\disj}}
\ustimes
{
\GA_{\Sigma^{I,\disj}}
}
\GO_{\Sigma^{I,\disj}}
\simeq
\Gcirc[\Sigma]^{\rat, \ul\lambda}_{\Sigma^{I,\disj}}
\ustimes
{\Gcirc(\A)^{\ul\lambda}_{\Sigma^{I,\disj}}}
\left( 
{\Gcirc(\A)^{\ul\lambda}_{\Sigma^{I,\disj}}}
\ustimes
{
\GA_{\Sigma^{I,\disj}}
}
\GO_{\Sigma^{I,\disj}}
\right)
\end{equation}
is infinitely connective.

\begin{lem}
For $ e\geq 0$, denote by $\Gcirc(\A)^{\rat, \ul\lambda, \leq e}_{\Sigma^{I,\disj}}$ the closed subspace of $\Gcirc(\OO)^{\rat, \ul\lambda}_{\Sigma^{I,\disj}}$ whose components $\phi^\pm$ have poles bounded by $e$ at the points $\x$.
There exists $e = e(\ul \lambda)$ such that the closed embedding 
$$
{\Gcirc(\A)^{\ul\lambda, \leq e}_{\Sigma^{I,\disj}}}
\ustimes
{
\GA_{\Sigma^{I,\disj}}
}
\GO_{\Sigma^{I,\disj}}
\hto
{\Gcirc(\A)^{\ul\lambda}_{\Sigma^{I,\disj}}}
\ustimes
{
\GA_{\Sigma^{I,\disj}}
}
\GO_{\Sigma^{I,\disj}}
$$
is an isomorphism.
\end{lem}

\begin{proof}
It is immediate to see that each of the spaces appearing in the lemma, denoted generically by $\Y_{\Sigma^{I,\disj}} \to \Sigma^{I,\disj}$, is \emph{factorizable} in the sense that there exists a canonical isomorphism
$$
\Y_{\Sigma^{I,\disj}} 
\simeq
\left(
\Y_{\Sigma} \times \cdots \times \Y_{\Sigma}
\right)
\ustimes
{\Sigma^I}
\Sigma^{I,\disj}.
$$
See \cite{BD-chiral} for much more on the notion of factorization.
Moreover, these factorization isomorphisms are naturally compatible with the maps
$$
{\Gcirc(\A)^{\ul\lambda, \leq e}_{\Sigma^{I,\disj}}}
\hto
{\Gcirc(\A)^{\ul\lambda}_{\Sigma^{I,\disj}}}
\to
{G(\A)^{\ul\lambda}_{\Sigma^{I,\disj}}}
\leftto
{G(\OO)^{\ul\lambda}_{\Sigma^{I,\disj}}}.
$$
This shows it suffices to prove the lemma in the case $I$ is a singleton: we need to prove that the closed embedding
$$
{\Gcirc(\A)^{\ul\lambda, \leq e}_{\Sigma   }}
\ustimes
{
\GA_{\Sigma }
}
\GO_{\Sigma }
\hto
{\Gcirc(\A)^{\ul\lambda}_{\Sigma }}
\ustimes
{
\GA_{\Sigma}
}
\GO_{\Sigma}
$$
is an isomorphism. To check this, we can work \'etale-locally on $\Sigma$ and therefore assume, for the remainder of the proof, that $\Sigma \simeq \Spec(\kk[t])$. 
The global coordinate on $\Sigma$ allows to identify the graph of an $S$-point of $\Sigma$ with the graph of the constant map with value $0 \in \Spec(\kk[t])$.
This yields isomorphisms
$$
{\Gcirc(\A)^{\ul\lambda, \leq e}_{\Sigma   }}
\simeq
\Gcirc \ppart ^{\leq e} \times \Sigma,
\hspace{.3cm}
{\Gcirc(\A)^{\ul\lambda}_{\Sigma   }}
\simeq
\Gcirc \ppart \times \Sigma,
$$
$$
{G(\OO)^{\ul\lambda}_{\Sigma   }}
\simeq
G [[t]] \times \Sigma,
\hspace{.3cm}
{G(\A)^{\ul\lambda}_{\Sigma   }}
\simeq
G \ppart \times \Sigma,
$$
which are compatible with one another in the natural way.
Hence, the assertion reduces to that of Lemma \ref{lemma:key-rep-thry}.
\end{proof}

\sssec{}

In view of the above lemma, we can rewrite \eqref{eqn:big auxiliary} as
\begin{equation} \label{eqn:big-auxiliary-2}
\Gcirc[\Sigma]^{\rat, \ul\lambda, \leq e}_{\Sigma^{I,\disj}}
\ustimes
{
\Gr_{G,\Sigma^{I,\disj}}
}
\Sigma^{I,\disj},
\end{equation}
where
$$
\Gcirc[\Sigma]^{\rat, \ul\lambda, \leq e}_{\Sigma^{I,\disj}}
:=
\Gcirc[\Sigma]^{\rat, \ul\lambda}_{\Sigma^{I,\disj}}
\ustimes
{
{\Gcirc(\A)^{\ul\lambda}_{\Sigma^{I,\disj}}}
}
{\Gcirc(\A)^{\ul\lambda, \leq e}_{\Sigma^{I,\disj}}}.
$$
By construction, the map $\Gcirc[\Sigma]^{\rat, \ul\lambda, \leq e}_{\Sigma^{I,\disj}} \to \Gr_{G, \Sigma^{I,\disj}}$ factors as
$$
\Gcirc[\Sigma]^{\rat, \ul\lambda, \leq e}_{\Sigma^{I,\disj}}
\to 
Y^{\ul \lambda}_{\Sigma^{I,\disj}}
\hto
 \Gr_{G, \Sigma^{I,\disj}},
$$
where $Y^{\ul \lambda}_{\Sigma^{I,\disj}}$ is a closed subscheme of $\Gr_{G, \Sigma^{I,\disj}}$. Indeed, any point $(\x, \phi^\pm, \phi^T)$ of $\Gcirc[\Sigma]^{\rat, \ul\lambda, \leq e}_{\Sigma^{I,\disj}}$ has poles at $\x$ bounded by $e$.

\sssec{}

To conclude, we proceed as in the proof of Lemma \ref{lem:Ainfty times something inf connective}.
So, we need to make sure that the dualizing sheaf of $\Gcirc[\Sigma]^{\rat, \ul\lambda, \leq e}_{\Sigma^{I,\disj}}$ is infinitely connective. This boils down to proving that the same is true for $\A^1[\Sigma]^{\rat, \ul\lambda, \leq e}_{\Sigma^{I,\disj}}$.
By Riemann-Roch, the latter admits a presentation as a colimit of vector bundles over $\Sigma^{I,\disj}$ of rank growing to $\infty$. Hence, its dualizing sheaf is indeed infinitely connective.

\ssec{Ran spaces with marked points} \label{ssec:Ran marked}

In this section, we explain how Theorem \ref{thm:maps to G generically in big cell} follows from Proposition \ref{prop:maps to G, gen in Gcirc, over disjoint locus}.

\sssec{}

Denote by $X$ the smooth complete curve containing $\Sigma$. 
The complement is a finite set of \virg{points at infinity}, which we denote by $D_\infty$. We regard $D_\infty$ as a $\kk$-point of $X^\sA$, where $\sA$ is a (finite nonempty) set that has been put in bijection with $D_\infty$ once and for all.

\sssec{}

We also need the Ran space with marked points, for which we follow the discussion of \cite[Section 3.5]{contract}. Specifically, what we need is the prestack $\Ran_{X,D_\infty}$ that parametrizes the finite sets of $X$ that contain $D_\infty$.
To give a formal definition, let $\fset_\sA$ be the category whose objects are arrows of finite sets $[\sA \to I]$, 
and whose morphisms are surjections $I \tto J$ compatible with the maps from $\sA$.
The finite set $D_\infty$ gives rise to the following functor:
$$
\fset_\sA^\op \longto \Sch,
\hspace{.5cm}
[\sA \to I] 
\squigto
X^{[\sA \to I]} :=
X^I \times_{X^\sA} \{D_\infty\}.
$$
Then we have
$$
\Ran_{X,D_\infty}
\simeq
\uscolim{[\sA \to I] \in \fset_{\sA}^\op} 
\;
X^{[\sA \to I]}.
$$

\sssec{}

Now we are ready to introduce the variant of $G[\Sigma]^{\Gcirc \ggen}$ that accounts for the points where the rational map $\Sigma \dasharrow \Gcirc$ is not defined.
We will perform the construction in general: in place of $\Gcirc \subset G$, we consider an open embedding $U \subset Y$ with $Y$ an affine scheme. Alongside the indscheme $Y[\Sigma]^{U\ggen}$, we have the prestack $Y[\Sigma]^{U\ggen}_{\Ran}$ defined as follows: its set of $S$-points consists of those pairs 
$$
(\x \in \Ran_{X,D_\infty}(S), \phi: \Sigma_S \to Y)
$$ 
for which $\restr\phi{\Sigma_S - D_\x} \to Y$ factors through $U \subset Y$.

\sssec{}

Thus, $Y[\Sigma]^{U\ggen}_{\Ran}$ fibers over $\Ran_{X,D_\infty}$ and we tautologically have
$$
Y[\Sigma]^{U\ggen}_{\Ran}
\simeq
\uscolim{[\sA \to I] \in \fset_\sA^\op}
\;
Y[\Sigma]^{U\ggen}_{[\sA \to I]},
$$
where 
$$
Y[\Sigma]^{U\ggen}_{[\sA \to I]}
:=
Y[\Sigma]^{U\ggen}_{\Ran}
\ustimes{\Ran_{X,D_\infty}} X^{[\sA \to I]}.
$$

\begin{lem}
Assume that $U \subset Y$ is a basic open subset. Then, for each $[\sA \to I] \in \fset_\sA$, the forgetful map 
$$
\xi_{[\sA\to I]}: Y[\Sigma]^{U\ggen}_{[\sA \to I ]} \to Y[\Sigma]^{U\ggen}
$$
is ind-proper.
Moreover, the resulting map $\xi: Y[\Sigma]^{U\ggen}_{\Ran} \to Y[\Sigma]^{U\ggen}$ has homologically contractible fibers.
\end{lem}

\begin{rem}
By definition, see \cite[Section 2.1.5]{contract}, a map $f: \X \to \Y$ of prestacks is \emph{ind-proper} (respectively: an ind-closed embedding) if it is ind-schematic and, for each scheme $S \to \Y$, the pullback $\X \times_\Y S$ admits an indscheme presentation $\X \times_\Y S \simeq \colim_{i \in \I} Z_i$ with each $Z_i$ proper over (respectively: a closed subscheme of) $S$. 
\end{rem}

\begin{rem}
The assumption that $U \subset Y$ be a basic open subset is important, although its role is hidden in the proof of \cite[Lemma 4.5.6]{contract} (the result on which the present lemma is based).
\end{rem}


\begin{proof}
As mentioned, this is a special case of a \cite[Lemma 3.5.6]{contract}; for the sake of completeness, and to match the notations, let us give more details.
Recall from \cite[Section 4.5]{contract} the definition of the prestack 
$\Maps(X, Y)^\rat_{X^J}$: it sends a test affine scheme $S$ to the set of pairs $(\x \in X^J(S), \phi: X_S-D_\x \to Y)$. 
This is a prestack (in fact, an indscheme) over $X^J$. It is related to our $Y[\Sigma]$ by setting $J = \sA$ and by observing that
$$
\Maps(X, Y)^\rat_{X^\sA}
\ustimes{X^\sA}
\{D_\infty\}
\simeq
Y[\Sigma].
$$
Now recall the prestack $\Maps(X,U \stackrel \gen \subset Y)^\rat_{X^{J}}$: it is the open subspace of $\Maps(X, Y)^\rat_{X^J}$ cut out by the condition that $\phi$ land generically inside $U$: precisely, we require that, for any geometric point $s \in S$, the map $\phi_s: X_s - \{\x_s\} \to Y$ land generically in $U$.
As above, setting $J = \sA$, we have:
$$
\Maps(X, U \stackrel \gen \subset Y)^\rat_{X^\sA}
\ustimes{X^\sA}
\{D_\infty\}
\simeq
Y[\Sigma]^{U\ggen}.
$$
Finally, let us introduce a different version of $\Maps(X,U \stackrel \gen \subset Y)^\rat_{X^{J}}$, where we control the locus where the map $\phi$ does land in $U$. Namely, for an arbitrary map $J \to I$ of finite nonempty sets, let $\Maps(X,U \stackrel \gen \subset Y)^\rat_{X^{J \to I}}$ be the prestack whose $S$-points are triples $(\x \in X^J(S), \y \in X^I(S), \phi: X_S - D_\x \to Y)$ with $D_\x \subseteq D_\y$ and such that $\phi$ restricts to a regular map $X_S-D_\y \to U$.

Letting $J = \sA$, we see that
$$
\Maps(X,U \stackrel \gen \subset Y)^\rat_{X^{\sA \to I}}
\ustimes{X^\sA}
\{D_\infty\}
\simeq
Y[\Sigma]^{U\ggen}_{\sA \to I}.
$$
Consider now the map
$$
f_{\on{source}}{(J \to I)}:
\Maps(X,U \stackrel \gen \subset Y)^\rat_{X^{J \to I}}
\longto
\Maps(X,U \stackrel \gen \subset Y)^\rat_{X^{J}}
$$
that forgets the datum of $\y$. In \cite[Section 4.6.4]{contract}, it is proven that such a map is ind-proper for any map $J \to I$ of nonempty finite sets.
Since our $\xi_{[\sA \to I]}$ is obtained from $f_{\on{source}}{(J \to I)}$ by setting $J = \sA$ and base-changing, it is ind-proper too.
The contractibility of the fibers of $\xi$ follows exactly as in \cite[Section 4.6]{contract}: it boils down to the contractibility of the Ran space of a smooth affine curve, which in turn is due to \cite{BD-chiral}.
\end{proof}

\sssec{}

Since each map $\xi_{[\sA \to I]}$ is ind-proper and $Y[\Sigma]^{U\ggen}$ is an indscheme, it follows that each ${Y[\Sigma]^{U\ggen}_{[\sA \to I]}}$ is an indscheme too (of ind-finite type). Alternatively, one can also give a direct proof.

In any case, let us endow $\Dmod \bigt{ {Y[\Sigma]^{U\ggen}_{[\sA \to I]}} }$ with the usual right t-structure present on any indscheme of ind-finite type.
The homological contractibility of the fibers of $\xi$ imples that
$$
\omega_{Y[\Sigma]^{U\ggen}}
\simeq
\xi_! 
\bigt{
\omega_
{Y[\Sigma]^{U\ggen}_{\Ran}}
},
$$
and thus
$$
\omega_{Y[\Sigma]^{U\ggen}}
\simeq
\uscolim{I \in \fset^\op}
\;
(\xi_{[\sA \to I]})_!
\Bigt{
\omega_
{Y[\Sigma]^{U\ggen}_{[\sA \to I]}}
}.
$$

\begin{lem}
If in the above situation each $\omega_
{Y[\Sigma]^{U\ggen}_{[\sA \to I]}}$ is infinitely connective, then so is $\omega_
{Y[\Sigma]^{U\ggen}}$.
\end{lem}

\begin{proof}
In view of the above formula, it suffices to verify that each functor $(\xi_{[\sA \to I]})_!$ is right t-exact up to a finite shift.
We can write $\xi_{[\sA \to I]}$ as the composition of two obviously defined maps:
$$
Y[\Sigma]^{U\ggen}_{[\sA \to I]}
\to
Y[\Sigma]^{U\ggen} \times X^{[\sA \to I]}
\tto
Y[\Sigma]^{U\ggen}.
$$
It is proven \cite[Section 4.6.4]{contract} that the leftmost map is an ind-closed embedding, hence it follows from the definition of the t-structure that the associated $!$-pushforward is right t-exact.
On the other hand, the $!$-pushforward along the second map is right t-exact up to a shift by $|I| = \dim(X^I)$.
\end{proof}

\sssec{}

Let us come back to our group case. The big cell $\Gcirc$ is a basic open subset of $G$: the explicit map realizing this is given in \cite{MO}. Thus, the above general results apply and, in particular, the following statement implies Theorem \ref{thm:maps to G generically in big cell}.

\begin{prop} \label{prop:maps to G generically in Gcirc, over I}
For each $[\sA \to  I] \in \fset_A$, the dualizing sheaf of $G[\Sigma]^{G^\circ\ggen}_{[\sA \to I]}$ is infinitely connective.
\end{prop}

\sssec{}

Let us reduce this to Proposition \ref{prop:maps to G, gen in Gcirc, over disjoint locus}.
Considering the diagonal stratification of $X^{[\sA \to I]}$, it is easy to see that it suffices to prove the assertion for the dualizing sheaf of each stratum. 
The smallest stratum (that is, the one with $I \simeq\sA$) yields the space $\Gcirc[\Sigma]$, in which case the assertion is clear: indeed,
$$
\Gcirc[\Sigma]
\simeq
N^-[\Sigma] \times T[\Sigma] \times N[\Sigma]
\simeq
\A^\infty \times T[\Sigma],
$$
whose dualizing sheaf is infinitely connective in view of Lemma \ref{lem:Ainfty times something inf connective}.

\sssec{}

It remains to treat the other strata. Denoting by $\Sigma^{I,\disj} \subseteq \Sigma^I$ the open subscheme parametrizing $I$-tuples of distinct points in $\Sigma$, it is clear that these strata are isomorphic to
$$
G[\Sigma]^{G^\circ\ggen}_{\Sigma^{I,\disj}}
:=
G[\Sigma]^{G^\circ\ggen}_{\Sigma^I}
\ustimes{\Sigma^I}
\Sigma^{I,\disj},
$$
where $I \in \fset$ (in particular, $I \neq \emptyset$).
This proves that Proposition \ref{prop:maps to G, gen in Gcirc, over disjoint locus} implies Proposition \ref{prop:maps to G generically in Gcirc, over I}.


\begin{thebibliography}{99}

\bibitem{AG1} D.~Arinkin, D.~Gaitsgory, {Singular support of coherent sheaves and the geometric Langlands conjecture}. {\newblock Selecta Math. (N.S.) 21 (2015), no. 1, 1-199.}

\bibitem{AG2} D.~Arinkin, D.~Gaitsgory, {The category of singularities as a crystal and global Springer fibers}.
{\newblock J. Amer. Math. Soc. 31 (2018), no. 1, 135-214.}

\bibitem{BL95}
A. Beauville et Y. Laszlo. Un lemme de descente. 
Comptes Rendus Acad. Sci. Paris, vol 320, 335-340, 1995.

\bibitem{Barlev} J. Barlev. {D-modules on spaces of rational maps}, 
\newblock Compositio Math. 150 (2014), 835-876.

\bibitem{BD-chiral} A. Beilinson, V. Drinfeld, 
{Chiral algebras.}
American Mathematical Society Colloquium Publications 51 (2004).

\bibitem{BD-quantization} A. Beilinson, V. Drinfeld, 
{Quantization of Hitchin's integrable system and Hecke eigensheaves}


\bibitem{BIK} D. Benson, S. B. Iyengar, and H. Krause. Local cohomology and support for triangulated categories. Ann. Sci. Ec. Norm. Super. 41(4):573-619, 2008.


\bibitem{thesis} D. Beraldo. {Loop group actions on categories and Whittaker invariants.} Advances in Mathematics, 322 (2017) 565-636.

\bibitem{ext-whit} D. Beraldo. {\it On the extended Whittaker category}, Sel. Math. New Ser. (2019) 25: 28.

\bibitem{centerH} D. Beraldo. {\it The center of $\H(\Y)$}, 
\newblock ArXiv:1709.07867.


\bibitem{shvcatHH} D. Beraldo, {\it Sheaves of categories with local actions of Hochschild cochains},
\newblock Compositio Mathematica, 155(8), 1521-1567.







\bibitem{strong-gluing} D. Beraldo. {\it The spectral gluing theorem revisited},
\newblock EPIGA, Volume 4 (2020), Article Nr. 9.

\bibitem{DL} D. Beraldo. {\it Deligne-Lusztig duality on the stack of local systems}, \newblock ArXiv:1906.00934.



\bibitem{BF} R. Bezrukavnikov, M. Finkelberg.
\newblock Equivariant Satake category and Kostant-Whittaker reduction.
\newblock Mosc. Math. J., 2008, Volume 8, Number 1,	Pages 39-72

\bibitem{CG}
N. Chriss, V. Ginzburg. 
\newblock
Representation Theory and Complex Geometry. Birkh\"auser Basel (2010).

\bibitem{DG-cptgen} V. Drinfeld, D. Gaitsgory, 
Compact generation of the category of D-modules on the stack of G-bundles on a curve.
Cambridge Journal of Mathematics 2015 (3), 19-125.


\bibitem{finiteness} V. Drinfeld and D. Gaitsgory, On some finiteness questions for algebraic stacks, GAFA 23 (2013),149-294.






\bibitem{Faltings} G.~Faltings, {Algebraic loop groups and moduli spaces of bundles.} J. Eur. Math. Soc. 5, 41-68 (2003).

\bibitem{FGT} S. Fishel, I. Grojnowski, C. Teleman, {The strong Macdonald conjecture and Hodge theory on the loop Grassmannian.} 
Annals of Mathematics, 168 (2008), 175-220.

\bibitem{strange} D. Gaitsgory, A strange functional equation for Eisenstein series and Verdier duality on the moduli stack of bundles, arXiv:1404.6780.


\bibitem{contract} D.~Gaitsgory, {Contractibility of the space of rational maps.}
Invent. math. 191, 91-196 (2013).

\bibitem{ICoh} D.~Gaitsgory, {Ind-coherent sheaves.}
Mosc. Math. J. 13 (2013), no. 3, 399-528, 553. 

\bibitem{Outline} D. Gaitsgory, { Outline of the proof of the geometric Langlands conjecture for $GL_2$}. Asterisque.



\bibitem{ker-adj} D. Gaitsgory, Functors given by kernels, adjunctions and duality, Journal of Algebraic Geometry 25 (2016), 461-548.

\bibitem{shvcat} D.~Gaitsgory, {\it Sheaves of categories and the notion of 1-affineness},
arXiv:1306:4304.

\bibitem{Book} D.~Gaitsgory, N. Rozenblyum, {\it Studies in derived algebraic geometry}. Mathematical Surveys and Monographs, 221. American Mathematical Society, Providence, RI, 2017.


\bibitem{Crystals} D.~Gaitsgory, N. Rozenblyum, {\it Crystals and D-modules}. 

\bibitem{GYD} D. Gaitsgory, A. Yom Din, An analog of the Deligne-Lusztig duality for $(\g,K)$-modules. Adv. Math. 333 (2018), 212-265. 


\bibitem{MO}
EMB (https://mathoverflow.net/users/160632/emb), Is the big cell a principal open set?, URL (version: 2020-07-04): https://mathoverflow.net/q/364832

\bibitem{Hesselink} 
W. Hesselink, W. Cohomology and the Resolution of the Nilpotent Variety. Mathematische annalen, 223(3), 1976, 249-252.



\bibitem{Kostant}
B. Kostant, Lie Group Representations on Polynomial Rings.
\newblock
American Journal of Mathematics Vol. 85, No. 3 (Jul., 1963), 327-404.





\bibitem{VLaff} V. Lafforgue, {\it Quelques calculs reli\'{e}s \`{a} la correspondance de Langlands g\'{e}om\'{e}trique pour $\PP^1$.} \\
\newblock Available at http://vlafforg.perso.math.cnrs.fr/files/geom.pdf.

\bibitem{Laumon1} 
G. Laumon,
Transformation de Fourier g\'{e}n\'{e}ralis\'{e}e. Preprint alg-geom/9603004

\bibitem{Laumon2} 
G. Laumon,
Transformation de Fourier g\'{e}om\'{e}trique.
Preprint IHES/85/M/52(1985)


\bibitem{HTT} J.~Lurie, {\it Higher Topos Theory}, Princeton Univ. Press (2009).
\bibitem{HA} J.~Lurie, {\it Higher algebra}.
\newblock Available at \url{http://www.math.harvard.edu/~lurie}.

\bibitem{MV} I. Mirkovic, K. Vilonen,
\newblock {Geometric Langlands duality and
representations of algebraic groups
over commutative rings.}
\newblock Annals of Mathematics, 166 (2007), 95-143.


\bibitem{Sam}
S. Raskin. D-modules on infinite dimensional varieties.

\bibitem{Roth1}
M. Rothstein, Sheaves with connection on abelian varieties.
Duke Math .J .84(3),565-598 (1996)

\bibitem{Roth2}
M. Rothstein, Correction to: “Sheaves with connection on abelian varieties”.
Duke Math. J. 87(1), 205-211 (1997)

\bibitem{CT} C. Teleman, \newblock
{Borel-Weil-Bott theory on the moduli stack of G-bundles over a curve. }
Invent. math. 134, 1-57 (1998).

\bibitem{Xinwen}
X. Zhu, An introduction to affine Grassmannians and the geometric
Satake equivalence. IAS/Park City Mathematics Series.




\end{thebibliography}
\end{document}